\documentclass[12pt]{amsart}
\usepackage[top=1in, bottom=1in, left=1in, right=1in]{geometry}
\usepackage{graphicx,color}
\usepackage{amssymb,mathrsfs}
\usepackage{epstopdf}
\newcommand{\kommentar}[1]{}
\newcommand{\N}{\mathbb{N}}
\newcommand{\Z}{\mathbb{Z}}
\newcommand{\Q}{\mathbb{Q}}
\newcommand{\R}{\mathbb{R}}
\newcommand{\C}{\mathbb{C}}

\newcommand{\E}{\mathbb{E}}
\newcommand{\F}{\mathbb{F}}

\newcommand{\W}{\widetilde{W}}

\newcommand{\CA}{\mathcal{A}}

\newcommand{\CC}{\mathcal{C}}

\newcommand{\CE}{\mathcal{E}}
\newcommand{\CF}{\mathcal{F}}
\newcommand{\CG}{\mathcal{G}}
\newcommand{\CH}{\mathcal{H}}

\newcommand{\CP}{\mathcal{P}}
\newcommand{\CQ}{\mathcal{Q}}

\newcommand{\CS}{\mathcal{S}}
\newcommand{\CT}{\mathcal{T}}

\newcommand{\M}{\mathrm{M}}

\newcommand{\SC}{\mathscr{C}}

\DeclareMathOperator{\Aut}{Aut}
\DeclareMathOperator{\GL}{GL}
\DeclareMathOperator{\SL}{SL}
\DeclareMathOperator{\Gal}{Gal}
\DeclareMathOperator{\tr}{tr}

\newcommand{\dee}{\mathrm{d}}
\newcommand{\leg}[2]{\left(\frac{#1}{#2}\right)}

\newcommand{\ds}{\displaystyle}

\newtheorem{thm}{Theorem}[section]
\newtheorem{cor}[thm]{Corollary}

\newtheorem{lma}[thm]{Lemma}

\theoremstyle{remark}
\newtheorem{rmk}{Remark}[section]

\newcommand{\eq}[2]{ \begin{equation} \label{#1}\begin{split} #2 \end{split} \end{equation} }
\newcommand{\al}[1]{\begin{align} #1 \end{align} }
\newcommand{\als}[1]{\begin{align*} #1 \end{align*} }

\newcommand{\bs}\boldsymbol{}

\newcommand{\nn}{\nonumber \\}

\newcommand{\prob}{\mathbb{P}}

\DeclareMathOperator{\rad}{rad}

\DeclareMathOperator{\cond}{cond}

\newcommand{\fl}[1]{\left\lfloor#1\right\rfloor}

\renewcommand{\Re}{{\rm Re}}

\renewcommand{\mod}[1]{\,({\rm mod}\,#1)}
\renewcommand{\pmod}[1]{\,({\rm mod}\,#1)}

\definecolor{pink}{rgb}{1,.2,.6}
\definecolor{orange}{rgb}{0.7,0.3,0}
\definecolor{blue}{rgb}{.2,.6,.75}

\numberwithin{equation}{section}

\renewcommand{\labelenumi}{(\alph{enumi})}

\title{Sums of Euler products and statistics of elliptic curves}

\author{Chantal David}
\address[Chantal David]{
Department of Mathematics and Statistics\\
Concordia University\\
1455 de Maisonneuve West\\
Montr\'eal, Qu\'ebec\\
H3G 1M8\\
Canada
}
\email{cdavid@mathstat.concordia.ca}

\author{Dimitris Koukoulopoulos}
\address[Dimitris Koukoulopoulos]{D\'epartement de math\'ematiques et de statistique\\
Universit\'e de Montr\'eal\\
CP 6128 succ. Centre-Ville\\
Montr\'eal, QC H3C 3J7\\
Canada}
\email{{\tt koukoulo@dms.umontreal.ca}}
\author{Ethan Smith}
\address[Ethan Smith]{
Department of Mathematics\\
Liberty University\\
1971 University Blvd\\
MSC Box 710052\\
Lynchburg, VA  24502
}
\email{ecsmith13@liberty.edu}


\date{\today}                                           

\begin{document}

\begin{abstract}
We present several results related to statistics for elliptic curves over a finite field $\F_p$ as corollaries of a general theorem about averages of Euler products that we demonstrate. In this general framework, we can reprove known results such as the average Lang-Trotter conjecture, the average Koblitz conjecture, and the vertical Sato-Tate conjecture, even for very short intervals, not accessible by previous methods. We also compute statistics for new questions, such as the problem of amicable pairs and aliquot cycles, first introduced by Silverman and Stange. Our technique is rather flexible and should be easily applicable to a wide range of similar problems. The starting point of our results is a theorem of Gekeler which gives a reinterpretation of Deuring's theorem in terms of an Euler product involving random matrices, thus making a direct connection between the (conjectural) horizontal distributions and the vertical distributions. Our main technical result then shows that, under certain conditions, a weighted average of Euler products is asymptotic to the Euler product of the average factors.
\end{abstract}

\maketitle

\setcounter{tocdepth}{2}
\tableofcontents


\section{Introduction}\label{intro}

Given a fixed elliptic curve $E$ over $\Q$, let $a_p(E)$ denote its trace of Frobenius at the prime $p$.
In~\cite{LT:1976}, Lang and Trotter constructed a heuristic probability model to predict an asymptotic for
\eq{fixed trace problem}{
	\#\{p\le x : a_p(E)=t\},
}
where $t$ is a fixed integer and $E/\Q$ is a fixed elliptic curve without complex multiplication. Considerations based on the Sato-Tate Conjecture and the Chebotar\"ev Density Theorem led them to postulate a model of the form
\eq{LT pmf}{
f_\infty(t,p) \cdot f(t,E)
}
for the probability that $a_p(E)=t$. To make their model compatible with the Sato-Tate Conjecture, Lang and Trotter chose\footnote{We have slightly changed the Lang-Trotter notation and absorbed the factor $c_p = 1/(2 \sqrt{p})$ in $f_\infty(t,p)$ such that $f_\infty(t,p)$ now corresponds to the Sato-Tate measure $\frac{2}{\pi} \sqrt{1 - u^2}$ with the change of variable $u=t/(2 \sqrt{p})$, and then $\int_{\mathbb{R}} f_\infty(t,p) \,dt = 1$.}
\eq{discrete Sato-Tate}{
	f_\infty(t,p)  =
		\begin{cases}
			\ds \frac{1}{\pi \sqrt{p}}\sqrt{1-\left(\frac{t}{2\sqrt p}\right)^2}
				&\text{if }|t|<2\sqrt p,\\
			0&\text{otherwise}.\end{cases}
}
To make their model compatible with the Chebotar\"ev Density Theorem applied to every $M$-division field of $E$, they chose
\[
	f(t,E) = \lim_{\substack{\longleftarrow\\ M}}\frac{M\cdot |G_E(M)_t|}{|G_E(M)|},
\]
where $G_E(M)$ denotes the image of the map
\[
		\Gal(\overline\Q/\Q)\stackrel{\rho_E}{\longrightarrow}
			\prod_\ell\GL_2(\Z_\ell)\longrightarrow\GL_2(\Z/M\Z),
\]
and $G_E(M)_t$ denotes the trace $t$ elements of $G_E(M)$.
Serre showed that the image of $\rho_E$ is open in $\prod_\ell\GL_2(\Z_\ell)$, whence it follows that there exists a positive integer $M_E>1$ such that
\als{
f(t,E)
	&=\frac{M_E\cdot |G_E(M_E)_t|}{|G_E(M_E)|}\cdot
	\prod_{\ell\nmid M_E}\left(\lim_{r\rightarrow\infty}
		\frac{\ell^r\cdot |\GL_2(\Z/\ell^r\Z)_t|}{|\GL_2(\Z/\ell^r\Z)|}\right) \\
	&= \frac{M_E\cdot |G_E(M_E)_t|}{|G_E(M_E)|}\cdot    \prod_{\ell\nmid M_E}
	\frac{\ell \cdot|\GL_2(\Z/\ell\Z)_t|}{|\GL_2(\Z/\ell\Z)|} ,
}
since the ratios at each prime $\ell \nmid M_E$ are constant for all $r \ge 1$. We also remark that the infinite product over the primes $\ell\nmid M_E$ is absolutely convergent. Fixing $t$ and letting $p\rightarrow\infty$, we have that $f_\infty(t,p)\sim \frac{1}{\pi \sqrt{p}}$, and summing the probabilities leads one to an ``expected value" of
\begin{eqnarray} \label{LT-conjecture}
\sum_{p\le x}f_\infty(t,p) \cdot f(t,E) \sim C_{E,t}\sum_{p\le x}\frac{1}{2\sqrt p}
	\sim C_{E,t}\int_2^x\frac{\dee u}{2\sqrt u\log u}
\end{eqnarray}
for~\eqref{fixed trace problem}, where $C_{E,t}=\frac{2}{\pi} f(t,E)$.
This conjectural asymptotic for~\eqref{fixed trace problem} is known as the ``fixed trace" Lang-Trotter Conjecture.

Alternatively, one may ask for a ``vertical" analogue of the above problem, where one fixes the prime $p$ and allows the elliptic curve $E$ to vary over all isomorphism classes of elliptic curves over the finite field $\F_p$. Here it is natural to count the isomorphism class $E/\F_p$ with weight $1/|\Aut_p(E)|$, where $\Aut_p(E)$ denotes the $\F_p$-automorphism group of $E$ as a curve over $\F_p$. If we let $\CC_p$ be the set of isomorphism classes of elliptic curves over $\F_p$, then
\[
\sum_{E\in \CC_p} \frac{1}{|\Aut_p(E)|} = p,
\]
as was observed in \cite{Lenstra:1987}. We may thus define a probability measure on $\CC_p$ by setting
\[
\prob_{\CC_p}(A) = \frac{1}{p} \sum_{E\in A} \frac{1}{|\Aut_p(E)|}
\]
for all $A\subset \CC_p$. Often, given a property $Q$ depending only on the isomorphism class of elliptic curves over $\F_p$, we will write
\[
\prob_{\CC_p}(E\ \text{has property}\ Q)
\quad\mbox{instead of}\quad \prob_{\CC_p}(\{E\in \CC_p: E\ \text{has property}\ Q\}).
\]

The measure $\prob_{\CC_p}$ can be also interpreted as the probability measure on (isomorphism classes of) elliptic curves that is induced by the (uniform) counting measure on the set of nonsingular Weierstrass equations defined over $\F_p$.  In particular, if one were to select elliptic curves defined over $\F_p$ by uniformly choosing nonsingular Weierstrass equations at random, then one would be $3=6/2$ times more likely to select a curve with automorphism group of size $2$ than one would be to select a curve with automorphism group of size $6$.  We recall that the only possible sizes for the automorphism groups are $2$, $4$, and $6$.  Furthermore, all but a bounded number of elliptic curves over $\F_p$ have exactly $2$ automorphisms. Thus, the difference between our induced measure on $\CC_p$ and the uniform counting measure on $\CC_p$ is only $O(1/p)$.

The distribution of the elements of $\CC_p$ with a fixed trace was described by Deuring~\cite{Deu:1941} in terms of class numbers of imaginary quadratic orders, who showed that
\eq{Deuring}{
\prob_{\CC_p}(a_p(E)=t)
	=\begin{cases}
		\ds \frac{H(D(t,p))}{p}
			&\text{if }|t|<2\sqrt p,\\
			0&\text{otherwise} ,
	\end{cases}
}
where $D(t,p):=t^2-4p$ and $H(D)$ is the Kronecker class number of discriminant $D$, which we define as follows. Given a negative discriminant $D$, we define the associated Kronecker class number by
\[
H(D)=\sum_{\substack{d^2|D \\ D/d^2\equiv 0,1\mod 4}}
	\frac{h(D/d^2)}{w(D/d^2)} ,
\]
where $h(\Delta)$ denotes the (ordinary) class number of the unique imaginary quadratic order of discriminant $\Delta$, and $w(\Delta)$ denotes the cardinality of its unit group.

In~\cite{Gekeler:2003}, Gekeler gave a reinterpretation of the above ``Deuring probability mass function"~\eqref{Deuring} in terms of random matrix theory, thus making even stronger the connection between the ``vertical" fixed trace distribution and the ``horizontal" fixed trace Lang-Trotter Conjecture. We state below Theorem 5.5 of \cite{Gekeler:2003} in a slightly modified form. We present in Section \ref{matrix count section} a new proof of this result, relying on a more combinatorial approach. It also slightly strengthens Theorem 5.5 of \cite{Gekeler:2003}: as Gekeler showed, the limit defining $f_\ell(t,p)$ stabilizes for large enough $r$, and we improve on how large $r$ needs to be in the case when $\ell$ does not divide the unique fundamental discriminant dividing $D(t,p)$. (See Theorem \ref{Gekeler thm} for the precise statement concerning the stabilisation point.) Our new approach is not necessary for the claimed improvement, which would also follow from small modifications of Gekeler's proof, but we believe that it is interesting in its own right. The  improvement itself will be important in the applications of Theorem \ref{Gekeler}. The connections with random matrix theory will be discussed later on.


\begin{thm}[Gekeler]\label{Gekeler}
Let $p$ be a fixed prime number, and let $t$ be any integer. We have that
\[
\prob_{\CC_p}(a_p(E)=t) = f_\infty(t,p)\cdot\prod_\ell f_\ell(t,p),
\]
where $f_\infty(t,p)$ is defined by~\eqref{discrete Sato-Tate}, and for each prime $\ell$,
\[
f_\ell(t,p)= \lim_{r\rightarrow\infty}\frac{\ell^r\phi(\ell^r)\cdot\#\left\{\sigma\in M_2(\Z/\ell^r\Z) :
	\begin{array}{l}
	\tr(\sigma)\equiv t\pmod{\ell^r},\\
	\det(\sigma)\equiv p\pmod{\ell^r}
	\end{array}\right\}}
	{|\GL_2(\Z/\ell^r\Z)|} .
\]
\end{thm}


\begin{rmk}\label{p-adic} As we mentioned above,
\[
\int_{\R} f_\infty(t,p) \dee t =1.
\]
Moreover, as observed by \cite[Remark 3.1]{Gekeler:2003}, it is easy to see that
\[
\int_{\Z_\ell} f_\ell(t,p) \dee \mu_\ell(t) = 1,
\]
where $\mu_\ell$ denotes the Haar measure on the $\ell$-adic integers $\Z_\ell$, that is to say, the quantities $f_\ell(t,p)$ can be interpreted as probability density functions for $t$ varying over $\Z_\ell$ or over $\R$.
\end{rmk}


\begin{rmk}
As we will see later on, we have that
\[
f_\ell(t,p) = 1 +\frac{\leg{t^2-4p}{\ell}}{\ell}+ O\left(\frac{1}{\ell^2}\right)
\]
for all $\ell\nmid t^2-4p$. In particular, the infinite product $\prod_\ell f_\ell(t,p)$ converges conditionally by the Prime Number Theorem for arithmetic progressions, but it does not converge absolutely. Similar remarks apply to Theorems \ref{matrix interpretation for M_p(G)}, and~\ref{matrix interpretation for M_p(t,n)} below. The fact that the convergence is so delicate will create some technical problems in the proof of Theorems \ref{LT}--\ref{MEG} when we average the `singular series' $ f_\infty(t,p)\cdot\prod_\ell f_\ell(t,p)$. This is in contrast with the situation in \cite{Gal:1976,Gal:1981,Kow:2011}. There the authors study averages of singular series arising from the Hardy-Littlewood $k$-tuple conjectures, and such singular series are given by absolutely convergent Euler products.
\end{rmk}


The main purpose of this paper is to show how Gekeler's result can lead to new proofs of vertical distribution results in a way that is both more unified and more conceptual. Some of our results have already been in the literature with other techniques, some improve previous results in the literature and some are new. We indicate that clearly when stating our results in Section \ref{subsection-SR}. Indeed, Deuring's formula \eqref{Deuring} has been in the heart of the proof of many results about the statistics of elliptic curves, such as results about the average probability that an elliptic curve lies in a given isogeny class (average Lang-Trotter conjecture), about the probability that $a_p(E)$ lies in a given interval (vertical Sato-Tate conjecture), and about the average probability that $\#E(\F_p)$ is a prime number (average Koblitz's conjecture). The proof of these results typically involve some rather involved local computations. After these calculations have been performed, one finds that the quantities in question are asymptotic to $C\cdot M$, where $M$ is some nice function varying smoothly in the various parameters involved and $C$ is a certain infinite Euler product. Then more local calculations reveal that $C$ has a natural probabilistic interpretation in terms of random matrices. In this paper, we will show how to use Gekeler's reinterpretation of Deuring's formula to arrive directly to a result of the form $C\cdot M$, where $C$ is given already in terms of local probabilities. What is more, the local computations are now completely straightforward and intuitive. All the results we will state below are easy corollaries of a rather general result, Theorem \ref{singular series average}.


\subsection{Statements of the results} \label{subsection-SR}

The first result involves averaging Gekeler's theorem. The study of this average originated in the work of Fouvry and Murty \cite{FM:1996} and of David and Pappalardi \cite{DP:1999} in their work on the average Lang-Trotter conjecture.

\begin{thm}\label{LT} Let $t\in\Z$ and $A>0$. For $x\ge2$, we have that
\[
\sum_{p\le x} \prob_{\CC_p}(a_p(E)=t)
	= C_{\text{LT}}(t)  \cdot \int_2^x\frac{dt}{2\sqrt{t}\log t} 	
	+   O_{t,A}\left(\frac{\sqrt{x}}{(\log x)^A}\right) ,
\]
where
\[
C_{\text{LT}}(t) := \frac{2}{\pi}
		\prod_{\ell}     \frac{\ell \cdot \# \GL_2(\Z/\ell\Z)_t}{\#\GL_2(\Z/\ell \Z) }  .
\]
\end{thm}


This result gives evidence for the Lang-Trotter conjecture as stated in \eqref{LT-conjecture}. In particular, note the similarity between the constants $C_{E,t}$ and $C_{\text{LT}}(t)$. It was shown by Jones \cite{Jon:2009} that the factor $M_E$ in $C_{E,t}$ can be controlled on average and that, for any fixed $t \in \Z$, the average of the constants $C_{E,t}$ of \eqref{LT-conjecture}
over all elliptic curves over $\Q$ is indeed the constant $C_{\text{LT}}(t)$; his results also apply to the average Koblitz constant $C_{\text{twin}}$ of Theorem \ref{Koblitz}.
\begin{rmk}
It is not immediately obvious that the Euler product defining the constant $C_{\text{LT}}(t)$ converges. One could, of course, deduce this easily by calculating explicitly the factors for each prime $\ell$. This is not necessary however, since the proof of Theorem \ref{LT} implies that the factor for the prime $\ell$ satisfies the estimate $1+O(1/\ell^{3/2})$, unless $\ell$ is one of the finitely many prime divisors of some non-zero integer $B$. The size of $B$ is controlled in terms of $t$, though the exact dependence is not needed here. Similar remarks apply to the constants appearing in all subsequent theorems of this section.
\end{rmk}


Next, we show a uniform version of the vertical Sato-Tate conjecture for the distribution of the normalized traces $a_p(E)/2 \sqrt{p}$ in an interval $[\alpha, \beta] \subset[-1,1]$. For fixed $\alpha$ and $\beta$, this theorem is due to Birch \cite{Birch}, and it has been proven for shorter intervals (and thin families of curves) by Banks and Shparlinski \cite[Lemma 9]{BS:2009} and by Baier and Zhao \cite[Theorem 3]{BZ:2009}. Theorem \ref{ST} below represents an improvement over both these results, demonstrating that $a_p(E)$ is distributed according to the Sato-Tate measure in all intervals $I\subset[-1,1]$ of length $\ge p^{-1/2 +\epsilon}$.


\begin{thm}\label{ST} Fix $\epsilon>0$ and $A\ge1$. For prime $p\ge2$ and $-1 \leq \alpha \leq \beta \leq 1$ with $\beta-\alpha\ge p^{-1/2+\epsilon}$, we have that
\[
\prob_{\CC_p}\left(\alpha \le \frac{a_p(E)}{2\sqrt{p}}\le \beta\right)
	=\left(1+O_{A,\epsilon}\left(\frac{1}{(\log p)^A}\right)\right)
		\frac{2}{\pi} \int_\alpha^\beta \sqrt{1-u^2} \, \dee u .
\]
\end{thm}


Our next theorem has three parts. The first one concerns the probability that, given $p$, an elliptic curve over $\F_p$ has a prime number of points. This question was first studied by Galbraith and McKee \cite{GM:2000}, and Conjecture 1 of their paper amounts to saying that the error term of \eqref{Koblitz-1} can be controlled. This is true under standard conjectures on the distributions of the primes in short intervals, but not unconditionally. This is similar to the situation in \cite{DS-MEN, DS-MEG} for elliptic curves over $\F_p$ with a fixed number of points, or a fixed group. If we average over $p$, then it is possible to show that their conjecture holds. This is the result \eqref{Koblitz-2} below, and it is similar to the results in \cite{CDKS2}, again for elliptic curves over $\F_p$ with a fixed number of points, or a fixed group. See also the remarks before Theorem \ref{MEN} and \ref{MEG}. The third statement of Theorem \ref{Koblitz} is a new proof of a result that arose in the work of Balog, Cojocaru and David \cite{BCD:2011} on the average Koblitz conjecture. Here and in the statements of some other results, we shall use the notations
\eq{def-E}{
E(y,h;q) := \max_{(a,q)=1} \left| \sum_{\substack{y<p\le y+h \\ p\equiv a\mod{q}}}\log p
						- \frac{h}{\phi(q)}\right|
}
and
\eq{def-R}{
R(x,h;m):=  \frac{\phi(m)}{h\sqrt{x}} \sum_{q\le \exp\{(\log\log2x)^2\}}
				 \int_{x^-}^{x^+}
				 	E(y,h;qm) \dee y ,
}
where
\begin{equation} \label{x-plus-minus}
x^\pm:=x\pm2\sqrt{x}+1 .
\end{equation}


\begin{thm}\label{Koblitz} Fix $\epsilon>0$ and $A\ge1$. For $p$ prime and $h\in[p^{\epsilon}, \sqrt{p}/(\log p)^{2A+1}]$, we have that
\eq{Koblitz-1}{
\prob_{\CC_p}( |E(\F_p)| \; \text{is prime} )  = \frac{C_{GM}(p)}{\log p}
		\left(1+ O\left(  \frac{1}{(\log p)^A}
			+  (\log\log p)^{O(1)} R(p,h;1)^{1/3} \right) \right),
}
where
\[
C_{GM}(p) :=  \prod_{\ell\neq p} \left(1-\frac{1}{\ell}\right)^{-1} \cdot \frac{\#\left\{\sigma\in\GL_2(\Z/\ell\Z) :
		\begin{array}{l} \det(\sigma) + 1 - \tr(\sigma)
	 \not\equiv 0 \mod{\ell} \\ \det(\sigma)\equiv p\mod{\ell}
	 \end{array} \right\}}
	 	{ \#\{\sigma\in \GL_2(\Z/\ell \Z): \det(\sigma)\equiv p\mod{\ell}\} }  ,
\]
and the implied constants depend at most on $\epsilon$ and $A$. Moreover,
\eq{Koblitz-2}{
\sum_{p\le x} \left|
		\prob_{\CC_p}( |E(\F_p)| \; \text{is prime} )  - \frac{C_{GM}(p)}{\log p}   \right|
		\ll_A \frac{x}{(\log x)^A}
}
and
\eq{Koblitz-3}{
\sum_{p\le x} \prob_{\CC_p}( |E(\F_p)| \; \text{is prime} )
	= C_{\text{twin}} \int_2^x\frac{\dee u}{\log^2u}
			+   O_A \left(\frac{x}{(\log x)^A}  \right) ,
}
where
\[
C_{\text{twin}} :=   \prod_{\ell}    \left(1-\frac{1}{\ell}\right)^{-1} \cdot
	\frac{\#\{\sigma\in\GL_2(\Z/\ell\Z) :\det(\sigma) + 1 - \tr(\sigma)
	 \not\equiv 0 \mod{\ell} \}}{ |\GL_2(\Z/\ell \Z)| }  .
\]
\end{thm}


Next, we study the average probability that a curve $\E/\F_p$ has precisely $N$ points. Note here we must have that $|N+1-p|<2\sqrt{p}$ by Hasse's bound or, equivalently, that $N^-<p<N^+,$ where $N^-, N^+$ are defined as in \eqref{x-plus-minus}.The study of this question was initiated by the first and the third authors in \cite{DS-MEN,DS-MEN-corr} and it was continued by the three authors of the paper and Chandee in \cite{CDKS2}. The main term in Theorem \ref{MEN} below is the expected one, but it is not possible to control the error term because we do not presently know how many primes are contained in an interval as short as $(N^-,N^+)$. For the same reason, the results of \cite{DS-MEN,DS-MEN-corr} are conditional on conjectures for primes in short intervals, and the unconditional results of \cite{CDKS2}  hold only for ``most $N$". The same paper also contains an appendix written by Martin and the first and third authors, where some relevant computations involving random matrices are performed. The sum over primes $p$ runs only over the primes $p \in (N^-, N^+)$ by the Hasse bound.


\begin{thm}\label{MEN}
Fix $\epsilon>0$ and $A\ge1$. If $N\ge2$ and $h\in[N^{\epsilon}, \sqrt{N}/(\log N)^{2A+1}]$, then
\[
\sum_p \prob_{\CC_p}( |E(\F_p)| = N )   =  \frac{C(N)}{\log N}
			\left(1  +   O\left(\frac{1}{(\log N)^A}
				+ (\log\log N)^{O(1)} R(N,h;1)^{1/3}\right) \right),
\]
where  the implied constants depend at most on $\epsilon$ and $A$, and
\[
C(N) :=  \prod_{\ell}
	\lim_{r\to\infty} \frac{\ell^r\cdot \#\left\{\sigma\in\GL_2(\Z/\ell^r\Z) :
	\begin{array}{l}
	\tr(\sigma)\equiv \det(\sigma)+1-N \mod{\ell^r},\\
	\end{array}\right\}}
	{\#\GL_2(\Z/\ell^r\Z)} .
\]
Furthermore,
\[
\sum_{N \leq x} \left|
	\sum_p \prob_{\CC_p}( |E(\F_p)| = N ) -
		\frac{C(N)}{\log N}\right| \ll_A \frac{x}{(\log x)^A}.
\]
\end{thm}



Next, let $\bs p=(p_1,\dots,p_d)$ be a $d$-tuple of distinct primes. The probability that choosing a `random' elliptic curve $E/\Q$ such that the primes $p_1,\dots,p_d$ form an {\it elliptic aliquot cycle of length $d$}, that is to say, $|E(\F_{p_j})|=p_{j+1}$, for $j\in\{1,\dots,d\}$ (with the notational convention that $p_{d+1}=p_1$), is given by
\[
\alpha_d(\bs p) :=  \prod_{j=1}^d \prob_{\CC_{p_j}}( |E_j(\F_{p_j})|=p_{j+1})  .
\]
This can also be interpreted as the probability of choosing randomly and independently $d$ elliptic curves $E_1,\dots,E_d$ over $\F_{p_1},\dots,\F_{p_d}$, respectively, with the property that $|E(\F_{p_j})|=p_{j+1}$, for $j\in\{1,\dots,d\}$. Our next goal is to understand the average size of $\alpha_d(\bs p)$, a question which arose in the work of Silverman and Stange \cite{SS:2011} and has been also studied by Jones \cite{Jones:2013} and Parks \cite{Parks:2014a,Parks:2014b}. Of course, Hasse's bound implies that for $\alpha_d(\bs p)$ to be non-zero, we must have that $|p_{j+1}-p_j-1|<2\sqrt{p_j}$ for all $j\in\{1,\dots,d\}$. To this extent, we define the set
\eq{P_d(x)}{
\CP_d(x) = \{(p_1,\dots,p_d): p_1\le x,\ |p_{j+1}-p_j-1|<2\sqrt{p_j},\ 1\le j\le d\}.
}
Then we have the following estimate, which sharpens and generalizes Theorem 1.6 in \cite{Parks:2014a} and Theorem 1.4 in \cite{Parks:2014b}, and proves the vertical distribution for aliquot cycles of length $d$ for all $d \geq 2$. For the case $d=2$, the same result was proven independently by Parks in \cite{Parks:2014a} with contributions from Giri using a different technique, and we discuss in the remark after Theorem \ref{thm:aliquot} the relation between both results. A precise conjecture for the horizontal distribution was made by Jones in \cite{Jones:2013}  following the probabilistic model of Lang-Trotter, and Theorem \ref{aliquot} confirms this asymptotic.


\begin{thm}\label{aliquot} \label{thm:aliquot}
For all $x\ge2$ and any fixed $A>0$, we have that
\[
\sum_{\bs p \in \CP_d(x)} \alpha_d(\bs p)
	= C_{\text{aliquot}}^{(d)} \int_2^x \frac{du}{2\sqrt{u}(\log u)^d}
		+ O_A\left( \frac{\sqrt{x}}{(\log x)^A} \right)
	 \sim C_{\text{aliquot}}^{(d)} \frac{\sqrt{x}}{(\log x)^d} ,
\]
where $C_{\text{aliquot}}^{(d)}$ is defined the be the product of the archimedian factor
\[
\frac{2^d}{\pi^d}
	\idotsint\limits_{\substack{|t_j|\le 1\ (1\le j\le d) \\ t_1+\cdots+t_d=0 }}
		\prod_{j=1}^d \sqrt{1-t_j^2} \ \dee t_1\cdots \dee t_{d-1} \\
\]
times the singular series
\[
\prod_{\ell} \lim_{r \rightarrow \infty} \frac{\ell^{rd}\cdot \
		\# \left\{ \bs\sigma\in \GL_2(\Z/\ell^r\Z)^d :
		\begin{array}{l}
		\det(\sigma_j) +1 -\tr(\sigma_j) \equiv  \det(\sigma_{j+1}) \mod{\ell^r}\\
		\mbox{for $1\le j\le d$, where $\sigma_{d+1}=\sigma_1$}
		\end{array} \right\}}{\left| \GL_2(\Z/\ell^r \Z) \right|^d} .
\]
\end{thm}

\begin{rmk} Unlike the situation in Theorems \ref{LT}, \ref{Koblitz}, \ref{MEN} and \ref{MEG}, the sequence
\[
P_{\text{aliquot}}^{(d)}(\ell^r)
	:= \frac{\ell^{rd}\cdot \
		\# \left\{ \bs\sigma\in \GL_2(\Z/\ell^r\Z)^d :
		\begin{array}{l}
		\det(\sigma_j) +1 -\tr(\sigma_j) \equiv  \det(\sigma_{j+1}) \mod{\ell^r}\\
		\mbox{for $1\le j\le d$, where $\sigma_{d+1}=\sigma_1$}
		\end{array} \right\}}{\left| \GL_2(\Z/\ell^r \Z) \right|^d}
\]
does not seem to become constant for large enough $r$. We do prove that the sequence $P_{\text{aliquot}}^{(d)}(\ell^r)$ converges as $r\to\infty$ and that its limit satisfies the asymptotic estimate
\[
\lim_{r\to\infty}P_{\text{aliquot}}^{(d)}(\ell^r)=1+O_d\left(\frac{1}{\ell^{3/2}}\right),
\]
but we do not have a closed expression for its value. For the case $d=2$, Parks \cite{Parks:2014b}, with contributions from Giri, obtained Theorem \ref{aliquot} with a different technique, without using Gekeler's theorem, but following similar steps as in the original proofs of Theorems \ref{LT}, \ref{Koblitz}, \ref{MEN} and \ref{MEG} \cite{FM:1996, DP:1999, BCD:2011, DS-MEN, DS-MEG}. Theorem 1.4 of \cite{Parks:2014b} (or, rather, its proof) implies that
\[
\lim_{r \rightarrow \infty} P_{\text{aliquot}}^{(2)}(\ell^r)
	= 1 - \frac{(2\ell^4 +3\ell^3) (\ell-2) - (\ell-1) (\ell^4-2\ell^3-4\ell^2+1)}{(\ell-1)(\ell^2-1)^3} ,
\]
and it remains a challenge to obtain a proof of this formula by a direct calculation of $P_{\text{aliquot}}^{(2)}(\ell^r)$.
\end{rmk}


So far the questions we have introduced involved understanding the probability that $a_p(E)$ or $\#E(\F_p)$ has a certain property. Next, we will study questions about the group structure of $E(\F_p)$, where $E$ is an elliptic curve over $\F_p$. It is well-known that
\[
E(\F_p)\cong\Z/m\Z\times\Z/mk\Z
\]
for some positive integers $m$ and $k$ satisfying the Hasse bound $|p+1-m^2k|<2\sqrt p$, which can be rewritten as $N^-<p<N^+$, where $N=m^2k$. Moreover, the Weil pairing implies that such a prime $p$ must lie in the class $1\mod{m}$. The following theorem is the analogous result to Theorem \ref{Gekeler} for $\prob_{\CC_p}(E(\F_p)\cong G)$, where $G$ is a group of the form $\Z/m\Z\times\Z/mk\Z$. As in Gekeler \cite{Gekeler:2003}, our starting point is a formula similar to \eqref{Deuring} proven by Schoof, which we reinterpret probabilistically.


\begin{thm}\label{matrix interpretation for M_p(G)}\label{Gekeler for groups}
Let $p$ be a fixed prime number. Given positive integers $m$ and $k$, let $t=t(m,k)=p+1-m^2k$ and $G=\Z/m\Z\times\Z/mk\Z$. We have that
\[
\prob_{\CC_p} ( E(\F_p)\cong G) = f_\infty(t,p)\cdot\prod_\ell f_\ell(G,p),
\]
where $f_\infty(t,p)$ is defined by~\eqref{discrete Sato-Tate}, for each prime $\ell$,
\[
f_\ell(G,p)= \lim_{r\rightarrow\infty}\frac{\ell^r\phi(\ell^r) \cdot\#\left\{\sigma\in M_2(\Z/\ell^r\Z) :
	\begin{array}{l}
	\tr(\sigma)\equiv t\pmod{\ell^r},\\
	\det(\sigma)\equiv p\pmod{\ell^r},\\
	\sigma\equiv I\pmod{\ell^{\nu_\ell(m)}},\\
	\sigma\not\equiv I\pmod{\ell^{\nu_\ell(m)+1}}
	\end{array}\right\}}
	{|\GL_2(\Z/\ell^r\Z)|}  .
\]
\end{thm}


\begin{rmk}
Note that, in accordance with the Hasse bound, $f_\infty(t,p)$ vanishes unless $|t|<2\sqrt p$. Furthermore, $f_\ell(G,p)$ vanishes if $p\not\equiv 1\pmod{\ell^{\nu_\ell(m)}}$. This is all in accordance with the restriction imposed by the Weil pairing.
Therefore, the probability of choosing an elliptic curve $E/\F_p$ with group $G$ is equal to zero unless we have both $|t|<2\sqrt p$ and $p\equiv 1\pmod m$, in which case the probability is nonzero.
\end{rmk}


We shall use Theorem \ref{Gekeler for groups} to deduce two other results. The first one is a reproof of Theorem 2.5 in \cite{CDKS2}, where the constant $C(G)$ below is denoted by $K(G) \cdot |G| / |\mbox{Aut}(G)|$ in  \cite{CDKS2}. Dealing with this constant using the original technique of averaging class numbers as in \cite{CDKS2} involves lengthy unpleasant computations, whereas the new proof we present here gives directly the value of $C(G)$ as a product of matrix counts. From Theorem \ref{MEG}, one can recover the average results of \cite{DS-MEG} using some additional hypotheses on the distribution of primes in short arithmetic progression to control the error terms, and the unconditional results of
\cite{CDKS2} for ``most" groups $G$.


\begin{thm}\label{MEG}
Fix $\epsilon>0$ and $A\ge10$. Let $G=\Z/m\Z\times\Z/mk\Z$ with $k\ge2$ and $1\le m\le k^A$. If $N=m^2k$ and $h\in[mk^{\epsilon},\sqrt{N}/(\log k)^{2A+1}]$, then
\[
\sum_p \prob_{\CC_p}( E(\F_p)\cong G)   =  \frac{C(G)}{\log|G|}\left( 1+
				 O\left(  \frac{1}{(\log k)^A} + \frac{(\log\log k)^{O(1)} R(N,h;m)^{1/3}}{\log k} \right)   \right) ,
\]
where  the implied constants depend at most on $\epsilon$ and $A$, and
\[
C(G) := \prod_{\ell}
	\lim_{r\to\infty} \frac{\ell^r\cdot \#\left\{\sigma\in\GL_2(\Z/\ell^r\Z) :
	\begin{array}{l}
	\tr(\sigma)\equiv \det(\sigma)+1-N \mod{\ell^r},\\
	\sigma\equiv I\mod{\ell^{\nu_\ell(m)}},\\
	\sigma\not\equiv I\mod{\ell^{\nu_\ell(m)+1}}
	\end{array}\right\}} {\#\GL_2(\Z/\ell^r\Z)} .
\]
\end{thm}


Our last result is a reproof of a weaker version of a result due to Vl{\u{a}}du{\c{t}} \cite{Vla:1999}, who built on work by Howe \cite{Howe:1993}.


\begin{thm}\label{cyclicity}
For $p$ a prime and $A\ge1$, we have that
\[
\sum_p \prob_{\CC_p}( E(\F_p)\ \text{is cyclic})  = C_{\text{cyclic}}(p)
	+ O_A\left(\frac{1}{(\log p)^A}\right) ,
\]
where
\[
C_{\text{cyclic}}(p) := \prod_{\ell\neq p} \frac{\#\left\{\sigma\in\GL_2(\Z/\ell\Z)\setminus\{I\} :
	\det(\sigma)\equiv p\pmod{\ell} \right\}}
	{\#\left\{\sigma\in\GL_2(\Z/\ell\Z) : \det(\sigma)\equiv p\pmod{\ell}\right\}}
	= \prod_{\ell \mid (p-1)} \left(1- \frac{1}{\ell(\ell^2-1)}\right) .
\]
\end{thm}


\subsection{Outline of the paper}

Before we embark on the more technical aspects of the paper, we discuss in Section \ref{section RMT} the connection between Gekeler's theorem and Theorem \ref{Gekeler for groups} with the general equidistribution philosophy originating from the work of Deligne, saying that Frobenius elements of elliptic curves (and in general abelian varieties) are equidistributed in groups of matrices. This provides a natural explanation for the local factors $f_\ell(t,p)$ obtained by Gekeler, at least when the prime $\ell$ is small compared to $p$. A more complete analysis of Gekeler's theorem in terms of equidistribution can be found in \cite{Katz}.

In Section \ref{matrix count section}, we prove Theorems \ref{Gekeler} and \ref{Gekeler for groups} that re-express the quantities $\prob_{\CC_p}(a_p(E)=t)$ and $\prob_{\CC_p}(E(\F_p) \cong G)$ as a product of local probabilities. Theorem \ref{Gekeler} was proven by Gekeler in \cite{Gekeler:2003}, but we present a new proof of his results relying on a more combinatorial approach. We also obtain a slight improvement over his results, showing that the limits defining the local probabilities stabilize earlier in some cases, and this will be important when we apply Theorem \ref{Gekeler} to prove the results of Section \ref{intro}.

Section \ref{singular series} is devoted to stating and explaining our main technical result which deals with averages of certain Euler products. This general result provides a unified framework, under which the results of Section \ref{intro} become easy corollaries. The general result is quite technical, so we begin Section \ref{singular series} by motivating, in a non-rigorous way, our particular choice of hypotheses. Section \ref{singular series AF} then contains the general axiomatic framework in which we will work in and the statement of our first result about sums of Euler products, Theorem \ref{singular series average general}. Then, we state in Section \ref{singular series SF} a second result about sums of Euler products, Theorem \ref{singular series average}, that holds under a simplified set of axioms. This is actually the result that will be invoked in all applications. The advantage of the more general set of axioms is that it is easier to see what is required by the mechanism of the proof, something that could be useful in applications of our results beyond the scope of this paper.

In Section \ref{applications}, we use Theorem \ref{singular series average} to prove the results claimed in Section \ref{intro}. Most of them follow as easy corollaries. However, there are some subtleties when proving Theorem \ref{ST}, especially when the interval $[\alpha,\beta]$ lies very close to 1 or -1. Moreover, Theorems \ref{Koblitz} and \ref{aliquot} require a technical auxiliary result, which will be proven separately in Section \ref{auxiliary}. The main input for this auxiliary result is a theorem about primes in short arithmetic progressions proven by the second author \cite{Kou}. Finally, Section \ref{singular series proof} contains the proof of Theorems \ref{singular series average general} and \ref{singular series average}.


\subsection*{Acknowledgements}
We would like to thank Henri Darmon, Dennis Eriksson, Gerard Freixas i Montplet, Jennifer Park, James Parks and Nicolas Templier for useful conversations. The research of the first author was partially supported by the National Science and Engineering Research Council of Canada (NSERC), and of the second author by NSERC and by the Fonds de recherche du Qu\'ebec -- Nature et technologies (FRQNT).


\section{Links with equidistribution in groups of matrices} \label{section RMT}

We can re-interpret the work of Gekeler about the probability that $a_p(E) = t$ over curves over $\F_p$ in terms of standard equidistribution results for the action of $\mbox{Frob}_p(E)$, the $p$th power Frobenius, on the $\ell$-torsion subgroups $E[\ell]$ as $E$ varies over isomorphism classes of elliptic curves over $\F_p$. This is the framework of the random matrix theory philosophy, initialized by Deligne with its equidistribution theorem, and further developed by Katz and Sarnak, who refined Deligne's equidistribution theorem to predict the statistical behavior for families of curves over finite fields. In a nutshell, the conjugacy classes of the Frobenius at $p$ acting on the $\ell$-torsion subgroup $E[\ell]$ become equidistributed in $\mbox{GL}_2(\F_\ell)$ as one varies over the family of elliptic curves over $\F_p$ and $p$ becomes large enough compared to $\ell$.  In order to make the connection clear, we state a precise theorem for the equidistribution of $\mbox{Frob}_p(E)$. This is based on
\cite{CH:2013}, but other similar explicit results for this case can also be found in \cite{Achter:2006}.

Let $N$ be a positive integer, and we write $N = N' p^e$, where $(N', p) = 1$ and $e \geq 0$ is an integer. Let, also, $E$ be an elliptic curve over $\F_p$. If $e \geq 1$, we further suppose that $E$ is ordinary.
Then
\[
E[N] \cong E[N'] \times E[p^e] \cong \Z/N'\Z \times \Z/N' \Z \times \Z/p^e \Z.
\]
Choosing a basis for $E[N']$ and a generator for $E[p^e] \simeq \Z/p^e \Z,$ the action of the $\mbox{Frob}_p(E)$  is given by a pair
\[
(F,T) \in \GL_2(\Z/N'\Z) \times (\Z / {p^e} \Z)^*
\]
such that
\als{
	\det(F) &\equiv p \ \mod {N'} \\
	\tr(F) &\equiv a_p(E) \ \mod {N'} \\
	T &\equiv a_p(E) \ \mod {p^e}.
}
Then $\mbox{Frob}_p(E)$ corresponds to a  pair $(\mathcal{F}_E, \mathcal{T}_E)$ where $\CF_E$ is a conjugacy class in $\GL_2(\Z/N'\Z)$ of determinant $p$ and $\CT_E \in (\Z / {p^e} \Z)^*$. The following equidistribution theorem was proved by Castryck and Hubrechts in \cite{CH:2013}. We state their result only when $N\le p^{1/4}$. In particular, $e=0$ here. This is without loss of generality, because the result of Castryck and Hubrechts is trivial when $N>p^{1/4}$; its error term becomes $\gg1$ then, and they have to look at elliptic curves over $\F_{q}$ for $q$ a large enough power of $p$ to get the desired equidistribution of the Frobenius.


\begin{thm}\cite[Theorem 2]{CH:2013} \label{thm:CH}
Let $p$ be a prime and $N\in[1,p^{1/4}]\cap\Z$. For any conjugacy class $\mathcal{F}$ in $\GL_2(\Z/ N \Z)$ with determinant $p$, we have
\[
\left| \prob_{\CC_p}(\mathcal{F}_E \in \mathcal{F}) 	
	- \frac{\#\CF}{\#\{\sigma\in \GL_2(\Z/ N\Z): \det(\sigma)\equiv p\mod{N}\}}  \right|
		\ll \frac{N^2 \log\log N}{\sqrt{p}} .
\]
\end{thm}


We remark that Theorem \ref{thm:CH} is proven by an application to the function field Cebotar\"ev's Density Theorem applied to the modular covering $X(p^2; \zeta_{N'}) \rightarrow X(1;1)$. The same result (under some mild restrictions on $p$ and $N$) was proved by Achter \cite{Achter:2008} via a direct application of the Katz-Sarnak equidistribution theorem.

We prove a result which is related to Theorem \ref{thm:CH}. In fact, our result improves the range of validity of the asymptotic in Theorem \ref{thm:CH} to $N\le p^{1/2-\epsilon}$ when $(t^2-4p,N)=1$, since in that case there is only one conjugacy class in $\GL_2(\Z/N\Z)$ of trace $t$ and determinant $p$.


\begin{thm}\label{trace in APs}
Let $\epsilon>0$, $A\ge1$, $p$ be a prime, $N\in[1,p^{1/2-\epsilon}]\cap\Z$ and $t\in\Z$. If
\[
\lambda =  \frac{\#\left\{\sigma\in \GL_2(\Z/N\Z):
		\begin{array}{l}
			\tr(\sigma)\equiv t\mod{N},\\
			\det(\sigma)\equiv p\mod{N}
			\end{array}
			\right\}}{\#\{\sigma\in\GL_2(\Z/ N\Z):\det(\sigma)\equiv p\mod{N}\}} ,
\]
then
\[
\prob_{\CC_p}\left( a_p(E)\equiv t\mod N\right)
	= \lambda \cdot \left(1+O_{\epsilon,A}\left(\frac{1}{(\log p)^A}\right)\right) .
\]
\end{thm}

\noindent
Theorem \ref{trace in APs} will be proven in the end of Section \ref{applications}.


\section{Class Number Formulas and matrices with fixed invariants}\label{matrix count section}

In this section, we prove Theorems \ref{Gekeler} and \ref{Gekeler for groups}. We start by giving in Theorem \ref{CNF} a formula for the Kronecker Class Number that is analogous to Dirchlet's Class Number Formula. After the completion of this paper, it was brought to our attention that the same formula appears in the work of Soundararajan and Young \cite[Lemma 2.1]{SY:2013}, building on some previous work of Bykovskii \cite{Byk:1994}, and in a different context in the work of Zagier \cite{Zagier}. We include our result for completeness, whose proof is different that \cite[Lemma 2.1]{SY:2013}.

Here and for the rest of the section, given $d\in\Z$, we set
\eq{def-FD}{
N_d(m)
	&= \#\{0\le x<2m  : x^2\equiv d \mod{4m} \} \\
	&=\frac{\#\{x\mod{4m} : x^2\equiv d \mod{4m} \}}{2} .
}
If $d\equiv2,3\mod4$, then $N_d=0$, whereas if $d\equiv0,1\mod4$, then $N_d$ is a multiplicative function.


\begin{thm}\label{CNF} For $D<0$, we have that
\[
H(D) = \frac{\sqrt{|D|}}{2\pi} \prod_\ell \left(1+\frac{1}{\ell}\right)^{-1} \sum_{j=0}^\infty \frac{N_D(\ell^j)}{\ell^j} .
\]
\end{thm}


\begin{proof} If $D$ is not a discriminant, both sides of the claimed identity are 0 and thus trivially equal. Assume now that $D$ is a negative discriminant. Given a binary quadratic form $f(x,y)=ax^2+bxy+cy^2$, we set $d_f=\text{gcd}(a,b,c)$. Recall that a form $f$ is called primitive if $d_f=1$. Let $\CF_D$ be a set of representatives for the equivalence classes of binary quadratic forms of discriminant $D$ under the usual action of $\SL_2(\Z)$, and let $\CF_D^*$ be a set of representatives for the equivalence classes of {\it primitive} binary quadratic forms of discriminant $D$. We write $u(f)$ for the cardinality of the set of matrices in $\SL_2(\Z)$ that leave $f$ invariant. Note that if $f$ is a form of discriminant $D$, then $d_f^2|D$ and $f/d_f$ is a primitive form of discriminant $D/d_f^2$. By the classical correspondence between class numbers of binary quadratic forms and of quadratic orders, we have that $h(D)=\#\CF_D^*$. Also, for a primitive form of discriminant $D$, $u(f)= w(D)$, where $w(D)$ is the number of units in the order of discriminant $D$ as defined before. Thus $u(f)=w(D/d_f^2)$.

We will use the proof of the class number formula for the class number $h(D)$ to prove the theorem. We write $r_f(n)$ for the number of representations of $n$ by values of the form $f$ and set
\[
R_D(n)= \sum_{f\in \CF_D} \frac{r_f(n)}{u(f)}
	= \sum_{\substack{d^2 | D \\ d|n}} \sum_{\substack{f\in \CF_D \\ d_f=d}} \frac{r_f(n)}{u(f)}
	= \sum_{\substack{d^2 | D \\ d|n}} \sum_{g\in \CF^*_{D/d^2} } \frac{r_g(n/d)}{w(D/d^2)} .
\]
Therefore
\als{
\frac{1}{x}\sum_{n\le x} R_D(n)
	&= \sum_{d^2|D} \sum_{g\in \CF^*_{D/d^2} } \frac{1}{d\cdot w(D/d^2)}
		\cdot \frac{1}{x/d} \sum_{m\le x/d} r_g(m)  \\
	&\sim \sum_{d^2|D} \frac{h(D/d^2)}{d\cdot w(D/d^2)} \cdot \frac{2\pi}{\sqrt{|D/d^2|}}
	= H(D)\cdot \frac{2\pi}{\sqrt{|D|}}
}
as $x\to\infty$ (see, for example, \cite[p. 48-49]{D}).

On the other hand, \cite[Theorem 3.27]{NZM:1991} implies that
\[
R_D(n) = \sum_{f\in \CF_D} \frac{r_f(n)}{u(f)}  = \sum_{d^2|n} N_D(n/d^2) .
\]
In particular, $R_D$ is a multiplicative function. We write it as $R_D=1*\leg{D}{\cdot}*\alpha_D$. If $\ell\nmid D$, then $N_D(\ell^j)=1+\leg{D}{\ell}$, so $R_D(\ell) = 1+\leg{D}{\ell}$, $\alpha_D(\ell)=0$ and $|\alpha_D(\ell^j)| \le 2\tau_4(\ell^j) \ll j^3$ for $j\ge2$. Finally, if $\ell|D$, then we use the bound $\alpha_D(\ell^j)\ll j^3 \ell^{\fl{j/2}}$, which follows from the elementary bound $N_D(\ell^j)\ll \ell^{\fl{j/2}}$. It is then easy to conclude that
\als{
\sum_{n\le y} |\alpha_D(n)|
	\ll \sum_{\substack{d\le y \\ p|b\ \implies p|D}} |\alpha_D(b)| \sum_{\substack{m\le y/b \\ (m,D)=1}} |\alpha_D(m)|
	&\ll_\epsilon \sum_{\substack{b\le y \\ p|b\ \implies p|D}} |\alpha_D(b)|
		\cdot \frac{y^{1/2+\epsilon}}{b^{1/2+\epsilon}} \\
	&= y^{1/2+\epsilon} \prod_{\ell|D} \left( 1+ \frac{|\alpha_D(\ell)|}{\ell^{1/2+\epsilon}}
		+\frac{|\alpha_D(\ell^2)|}{\ell^{2(1/2+\epsilon)}}+\cdots \right) \\
	&\ll_{\epsilon,D} y^{1/2+\epsilon} ,
}
for all $y\ge1$. Also, $\sum_{n\le y}\leg{D}{n} \ll_D1$ since $D<0$ and thus $D$ is not a perfect square. Consequently, Dirichlet's hyperbola method implies that
\[
\lim_{x\to\infty} \frac{1}{x} \sum_{n\le x} R_D(n)
	= \prod_\ell \left(1-\frac{1}{\ell}\right) \sum_{j=0}^\infty \frac{R_D(\ell^j)}{\ell^j} .
\]
Finally, note that
\[
R_D(\ell^j) = \sum_{\substack{0\le i\le j \\ i\equiv j\mod{2}}} N_D(\ell^i),
\]
so that
\[
 \sum_{j=0}^\infty \frac{R_D(\ell^j)}{\ell^j}
 	= \left(1-\frac{1}{\ell^2}\right)^{-1}  \sum_{j=0}^\infty \frac{N_D(\ell^j)}{\ell^j} .
\]
Putting together the above formulas completes the proof of the theorem.
\end{proof}


Next, we turn our attention to calculating the cardinality of the sets
\[
C(t,u,n;\ell^r)
	:=\left\{\sigma\in\M_2(\Z/\ell^r\Z):
		\begin{array}{l}
			\tr(\sigma)\equiv t\pmod{\ell^r},\\
			\det(\sigma)\equiv u\pmod{\ell^r},\\
			\sigma\equiv I\pmod{\ell^{\nu_\ell(n)}}
		\end{array}\right\}
\]
and the limits
\eq{f_ell(t,u,n)}{
f_\ell(t,u,n):=
	\lim_{r\to\infty} \frac{\ell^r\phi(\ell^r) \cdot |C(t,u,n;\ell^r)| }{|\GL_2(\Z/\ell^r\Z)|}  ,
}
where $u$ now is a general integer. Given any integers $t$ and $u$, we set $D=D(t,u)=t^2-4u$. Note that, for the purposes of this discussion, we do not need to assume that $D<0$.

When $n=1$, the computation of $\#C(t,u,n;\ell^r)$ was already carried out by Gekeler in~\cite{Gekeler:2003} for $r$ sufficiently large. Also, in the case that $n=1$, the count was carried out by Castryck and Hubrechts~\cite{CH:2013} for all $r\ge 1$. Theorem \ref{Gekeler thm}(a) below gives a formula for $\#C(t,u,1;\ell^r)$ that improves slightly Theorem 4.4 of~\cite{Gekeler:2003}, in the sense that the claimed formula holds for $r>\nu_\ell(D)$. However, note that, unlike in \cite{Gekeler:2003} and in \cite{CH:2013}, we do not give an explicit formula for $\#C(t,u,1;\ell^r)$; the stated combinatorial expression suffices for our purposes and makes the exposition cleaner.

Throughout, we will be assuming that $u\equiv 1\mod{n}$ and $u+1-t\equiv 0\mod{n^2}$. This can be justified by the observation that if either of these conditions fails, then the set $C(t,u,n;\ell^r)$ will be empty for some $\ell$ dividing $n$ and $r$ large enough. Indeed, writing
\eq{matrix change of variables}{\sigma=\begin{pmatrix}
			1+ n\alpha 	&n\beta \\
			n\gamma	&1+n\delta \end{pmatrix},
}
we find that $\sigma\in C(t,u,n;\ell^r)$ if, and only if,
\eq{C(t,u,n)}{
2+n(\alpha+\delta)
			&\equiv t\pmod{\ell^r},\\
		1+n(\alpha+\delta)+n^2(\alpha\delta-\beta\gamma)
			&\equiv u\pmod{\ell^r}.
}
In particular, if $r\ge 2\nu_\ell(n)$, then we must have that $u\equiv 1\pmod{\ell^{\nu_\ell(n)}}$ and $u\equiv t-1\mod{\ell^{2\nu_\ell(n)}}$. So, from now on, we will always be working under the assumption that $u\equiv 1\pmod n$ and $u+1-t\equiv 0\mod{n^2}$, which holds trivially when $n=1$ too. Under this assumption,
\[
D\equiv t^2-4(t-1)\equiv (t-2)^2\pmod{4n^2}
\]
and $n\mid (t-2)$. Hence, it follows that $n^2\mid D$ and $D/n^2$ is a discriminant.


\begin{thm}\label{Gekeler thm}
Let $t,u\in\Z$, $D=D(t,u)=t^2-4u$ and $n\in\N$ with $u\equiv 1\mod{n}$ and $u+1\equiv t\mod{n^2}$.\begin{enumerate}
\item For $r\ge 1$ and $u'\equiv u\mod{\ell^{\nu_\ell(D)+1}}$, we have that
\[
\#C(t,u',1;\ell^r)
	= \ell^{2r} +\ell^{2r} \sum_{j=1}^{\min\{r,\nu_\ell(D)+1\}}
			\frac{N_D(\ell^j)-N_D(\ell^{j-1})}{\ell^j} .
\]
If, in addition, $r>\nu_\ell(D)$, then
\[
\frac{\ell^r\phi(\ell^r)\cdot |C(t,u',1;\ell^r)|}{|\GL_2(\Z/\ell^r\Z)|}
	= \left(1+\frac{1}{\ell}\right)^{-1} \sum_{j=0}^\infty \frac{N_{D}(\ell^j)}{\ell^j} .
\]
In particular, the limit defining $f_\ell(t,u,1)$ exists and it equals the right hand side of the above identity.

\item The sequence over $r$ defining $f_\ell(t,u,n)$ is constant for $r>\nu_\ell(D)$. In particular, $f_\ell(t,u,n)$ is well-defined. Moreover, we have the formulas
\[
f_\ell(t,u,n) = \frac{f_\ell(t_1,u_1,1)}{\ell^{\nu_\ell(n)} }
	\quad\text{and}\quad
f_\ell(t,u,n) - f_\ell(t,u,\ell n) = \frac{f_\ell^*(t_1,u_1,1)}{\ell^{\nu_\ell(n)} } ,
\]
where $t_1=(t-2)/n$, $u_1=(u+1-t)/n^2$ and $f_\ell^*(t_1,u_1,1)$ is defined as $f_\ell(t_1,u_1,1)$ with the difference that we replace $|C(t_1,u_1,1;\ell^r)|$ by $\#\{\sigma\in C(t_1,u_1,1;\ell^r): \sigma\not\equiv 0\mod{\ell}\}$. Finally, the sequence over $r$ defining $f_\ell^*(t_1,u_1,1)$ is constant for $r>\nu_\ell(t_1^2-4u_1)$.

\item If $\ell\nmid D/n^2$, then $f_\ell(t,u,\ell n)=0$ and
\[
f_\ell(t,u,n) =  \frac{1}{\ell^{\nu_\ell(n)}}
		 	\left(1-\frac{1}{\ell^2}\right)^{-1} \left(1+\frac{\leg{D/n^2}{\ell}}{\ell} \right) .
\]

\item For every $r\ge1$, we have that
\[
\#C(t,u,n;\ell^r) =  \ell^{2r-\nu_\ell(n)} + O(\ell^{2r-\nu_\ell(n)-1}) .
\]
\end{enumerate}
\end{thm}


\begin{proof}
(a) For convenience, set $D'=D(t,u')=t^2-4u'$ and note that $\nu_\ell(D')=\nu_\ell(D)$, since $D'\equiv D\mod{4\ell^{\nu_\ell(D)+1}}$. Now, note that $\#C(t,u',1;\ell^r)$ counts quadruples $(a,b,c,d)\in(\Z/\ell^r\Z)^4$ with $a+d\equiv t \mod{\ell^r}$ and $ad-bc\equiv u' \mod{\ell^r}$. Equivalently, it counts triples $(a,b,c)\in(\Z/\ell^s\Z)^3$ such that $bc\equiv a(t-a)-u'\mod{\ell^r}$. We write $b=\ell^jb'$, where $0\le j\le r$ and $b'\in(\Z/\ell^{r-j}\Z)^*$. We must have that $a(t-a)-u'\equiv 0\mod{\ell^j}$, and for each such $a$ and $b$, there are exactly $\ell^j$ possibilities for $c$. Therefore
\als{
\#C(t,u',1;\ell^r)
	&= \sum_{j=0}^r \phi(\ell^{r-j}) \cdot \#\{a\mod{\ell^r} : a^2-ta+u'\equiv 0\mod{\ell^j}\} \cdot \ell^j\\
	&= \ell^r \sum_{j=0}^r \phi(\ell^{r-j}) \cdot \#\{0\le a<\ell^j : a^2-ta+u'\equiv 0\mod{\ell^j}\} .
}
We note that $a^2-ta+u'\equiv 0\mod{\ell^j}$ if, and only if, $(2a-t)^2\equiv D'\mod{4\ell^j}$. To this end, we make the change of variable $x=2a-t$, which caries the restriction $x\equiv t\mod 2$. However, this is automatic if $x^2\equiv D'\mod 4$, and we find that
\als{
\#\{0\le a<\ell^j : a^2-ta+u'\equiv 0\mod{\ell^j}\}
	&= \#\{-t\le x<2\ell^j-t : x^2\equiv D'\mod{4\ell^j}\} \\
	&= N_{D'}(\ell^j),
}
since the function $x^2$ is $2\ell^j$-periodic mod $4\ell^j$. So, using the identity $\phi(n)=n\sum_{d|n}\mu(d)/d$, we deduce that
\als{
\#C(t,u',1;\ell^r)
	= \ell^{2r} \sum_{j=0}^r \frac{N_{D'}(\ell^j)}{\ell^j} \sum_{i=0}^{r-j} \frac{\mu(\ell^i)}{\ell^i}
	&= \ell^{2r} \sum_{i=0}^r \frac{\mu(\ell^i)}{\ell^i} \sum_{j=0}^{r-i} \frac{N_{D'}(\ell^j)}{\ell^j} \\
	&= \ell^{2r} +\ell^{2r} \sum_{i=1}^r \frac{N_{D'}(\ell^i)-N_{D'}(\ell^{i-1})}{\ell^i} .
}
Using Hensel's lemma, it is relatively easy to see that the sequence $N_{D'}(\ell^i)$ is constant for $i\ge \nu_\ell(D')+1$, so
\[
\#C(t,u',1;\ell^r)
	= \ell^{2r} +\ell^{2r} \sum_{j=1}^{\min\{r,\nu_\ell(D')+1\}}
			\frac{N_{D'}(\ell^j)-N_{D'}(\ell^{j-1})}{\ell^j} .
\]
Recall that $\nu_\ell(D')=\nu_\ell(D)$ and that $D'\equiv D\mod{4\ell^{\nu_\ell(D)+1}}$. So, if $j\le \nu_\ell(D')+1$, then $N_{D'}(\ell^j) = N_{D}(\ell^j)$, which proves the first formula in the statement of part (a). Finally, using again the fact that $N_{D}(\ell^i)$ is constant for $i\ge \nu_\ell(D)+1$, we find that if $r>\nu_\ell(D)$, then
\[
\frac{\#C(t,u',1;\ell^r)}{\ell^{2r}}
	= 1 +\sum_{j=1}^\infty
			\frac{N_{D}(\ell^j)-N_{D}(\ell^{j-1})}{\ell^j}
	=\left(1-\frac{1}{\ell}\right) \sum_{j=0}^\infty \frac{N_D(\ell^j)}{\ell^j} .
\]
Since
\[
|\GL_2(\Z/\ell^r\Z)| = \ell^{4(r-1)}(\ell^2-1)(\ell^2-\ell) =\ell^{4r} \left(1-\frac{1}{\ell}\right)^2 \left(1+\frac{1}{\ell}\right) ,
\]
the second formula of part (a) follows too.

\medskip

(b) Set $a=\nu_\ell(n)$. Making the change of variables \eqref{matrix change of variables}, we immediately see by \eqref{C(t,u,n)} that
\al{
|C(t,u,n;\ell^r)|
	&= \#\left\{
\sigma\in\M_2(\Z/\ell^{r-a}\Z) :
	\begin{array}{l}
	\tr(\sigma)\equiv t_1\pmod{\ell^{r-a}},\\
	\det(\sigma)\equiv u_1\pmod{\ell^{r-2a}}
	\end{array}
\right\}  \nn
	&= \sum_{\substack{0\le u_2<\ell^{r-a}\\ u_2\equiv u_1\pmod{\ell^{r-2a}}}}
	|C\left(t_1,u_2,1;\ell^{r-a}\right)| .  \label{formula for matrix count}
}
Set $D_1=D(t_1,u_1) = D/n^2$ and note that if $u_2\equiv u_1\mod{\ell^{r-2a}}$ and $r>\nu_\ell(D)$, then $u_2\equiv u_1\mod{\ell^{\nu_\ell(D_1)+1}}$. Therefore, if $r>\nu_\ell(D)$, then part (a) implies that
\[
|C(t,u,n;\ell^r)| = \ell^a |C(t_1,u_1,1;\ell^{r-a})|
\]
and, consequently,
\[
\frac{\ell^r\phi(\ell^r)\cdot |C(t,u,n;\ell^r)|}{|\GL_2(\Z/\ell^r\Z)|}
	= \frac{1}{\ell^a} \cdot \frac{\ell^{r-a}\phi(\ell^{r-a})\cdot |C(t_1,u_1,1;\ell^{r-a})| }{|\GL_2(\Z/\ell^{r-a}\Z)|} .
\]
Moreover, the right hand side is constant for $r-a>\nu_\ell(D_1)=\nu_\ell(D)-2a$, by part (a). In particular, it is constant if $r>\nu_\ell(D)$. This proves that the sequence over $r$ defining $f_\ell(t,u,n)$ is constant for $r>\nu_\ell(D)$ and that
\eq{Gekeler e200}{
f_\ell(t,u,n) = \frac{f_\ell(t_1,u_1,1)}{\ell^a} .
}

Next, note that
\eq{Gekeler e300}{
f_\ell(t,u,n)-f_\ell(t,u,\ell n) = \frac{f_\ell(t_1,u_1,1)}{\ell^a} - \frac{f_\ell(t_1/\ell,u_1/\ell^2,1)}{\ell^{a+1}}
}
by \eqref{Gekeler e200}, where the second term vanishes unless $\ell|t_1$ and $\ell^2|u_1$. Making the change of variables $\tau=\sigma+I$, we see that
\als{
&\#\{\sigma\in C(t_1,u_1,1;\ell^r) :\sigma\not\equiv 0\mod{\ell}\} \\
	&\quad= \#\{\tau \in C(t_1+2,u_1+t_1+1,1;\ell^r) :\tau\not\equiv I\mod{\ell}\} \\
	&\quad= |C(t_1+2,u_1+t_1+1,1;\ell^r)| - |C(t_1+2,u_1+t_1+1,\ell;\ell^r)| .
}
In particular, we see that
\als{
 f_\ell^*(t_1,u_1,1)
 	&= f_\ell(t_1+2,u_1+t_1+1,1) -  f_\ell(t_1+2,u_1+t_1+1,\ell)  \\
	&= f_\ell(t_1,u_1,1) - \frac{f_\ell(t_1/\ell,u_1/\ell^2,1)}{\ell}
}
by \eqref{Gekeler e200}, where the second term vanishes unless $\ell|t_1$ and $\ell^2|u_1$, which, together with \eqref{Gekeler e300}, demonstrates the claimed formula for $f_\ell(t,u,n)-f_\ell(t,u,\ell n)$.

It remains to prove that the limit defining $f_\ell^*(t_1,u_1,1)$ stabilizes for $r>\nu_\ell(D_1)$. As above, making the change of variables $\tau=\sigma+I$, it suffices to prove the same statement for the limits defining $f_\ell(t_1+2,u_1+t_1,1)$ and $f_\ell(t_1+2,u_1+t_1+1,\ell)$. We have already seen this that the sequence over $r$ defining the former is constant for $r>\nu_\ell(D_1)$. We will show the same for the limit defining $f_\ell(t_1+2,u_1+t_1+1,\ell)$. If $\ell|t_1$ and $\ell^2| u_1$, then this follows by the portion of part (b) already proven. Assume now that either $\ell\nmid t_1$ or $\ell^2\nmid u_1$. Then it is easy to see that $C(t_1+1,u_1+t_1+1,\ell;\ell^r)=\emptyset$ for $r\ge2$ (see, for example, the discussion preceding the statement of Theorem \ref{Gekeler thm}). Therefore, if $\nu_\ell(D_1)\ge1$, then our claim has been proven. Finally, if $\nu_\ell(D_1)=0$, then $C(t_1+2,u_1+t_1+1,\ell;\ell^r)=\emptyset$ for $r\ge1$. Indeed, if $\sigma\equiv I\mod{\ell}$, then $\tr(\sigma)\equiv 2\mod{\ell}$, and $\det(\sigma)\equiv 1\mod \ell$. So, if in addition, $\tr(\sigma)\equiv t_1+2\mod\ell$ and $\det(\sigma)\equiv u_1+t_1+1\mod\ell$, then we must have that $\ell|t_1$ and $\ell|u_1$, whence $\ell| D_1$, a contradiction. This proves that $C(t_1+2,u_1+t_1+1,\ell;\ell^r)=\emptyset$ for $r\ge1=1+\nu_\ell(D_1)$ when $\nu_\ell(D_1)=0$, thus completing the proof of part (b).

\medskip

(c) Since $\ell\nmid D_1=D/n^2$, we see immediately that $C(t,u,\ell n;\ell^r)=\emptyset$ for large enough $r$, by the discussion preceding the theorem, so $f_\ell(t,u,\ell n)=0$. Finally, part (b) and the first formula in part (a) imply that
\[
f_\ell(t,u,n) = \frac{f_\ell(t_1,u_1,1)}{\ell^{\nu_\ell(n)}}
	=  \frac{1}{\ell^{\nu_\ell(n)}(1-1/\ell^2)} \left(1+
			\frac{N_{D_1}(\ell)-1}{\ell} \right) .
\]
Since $N_{D_1}(\ell)= 1+\leg{D_1}{\ell}$ when $\ell\nmid D_1$, the claimed formula for $f_\ell(t,u,n)$ follows.

\medskip

(d) This follows by \eqref{formula for matrix count} and the fact that
\[
|C(t_1,u_2,1;\ell^s)|=\ell^{2s}+O(\ell^{2s-1}),
\]
which is a simple consequence of part (a) together with the fact that $N_{D_1}(\ell^j)\ll \ell^{\fl{j/2}}$. (See, also Theorem 7 in \cite{CH:2013}.)
\end{proof}


It is now straightforward to deduce Theorems \ref{Gekeler} and \ref{Gekeler for groups}. As an intermediate step, we fix an integer $n$ and ask for the proportion of elliptic curves $E/\F_p$ with $a_p(E)=t$ and $E(\F_p)[n]\cong \Z/n\Z\times\Z/n\Z$. Then Lemma 4.8 and Theorem 4.9 of~\cite{Sch:1987} essentially say that
\[
\prob_{\CC_p}(a_p(E)=t,\, E(\F_p)[n]\cong  \Z/n\Z\times\Z/n\Z )
	= \frac{H(D/n^2)}{p}
\]
if $|t|<2\sqrt p,\, n\mid p-1$ and $n^2\mid p+1-t$; otherwise, this probability equals 0. As before, the conditions $n\mid p-1$ and $n^2\mid p+1-t$ together imply that $n^2\mid D$ and that $D/n^2$ is a negative discriminant. Thus, the Kronecker class number $H(D/n^2)$ is well-defined. As a direct corollary of Theorem \ref{CNF} and parts (a) and (b) of Theorem \ref{Gekeler thm}, we have the following result.


\begin{cor}\label{matrix interpretation for M_p(t,n)}
Let $p$ be a fixed prime number, and let $t$ and $n$ be any integers with $n\ge 1$. Then
\[
\prob_{\CC_p}(a_p(E)=t,\, E(\F_p)[n]\cong  \Z/n\Z\times\Z/n\Z)
	=  f_\infty(t,p)\cdot \prod_\ell f_\ell(t,p,n),
\]
where $f_\infty(t,p)$ is defined by~\eqref{discrete Sato-Tate} and $f_\ell(t,p,n)$ is defined by \eqref{f_ell(t,u,n)}.
\end{cor}


Taking $n=1$ in Corollary \ref{matrix interpretation for M_p(t,n)} yields Theorem \ref{Gekeler}. Lastly, we show how to deduce Theorem~\ref{matrix interpretation for M_p(G)}.


\begin{proof}[Proof of Theorem~\ref{matrix interpretation for M_p(G)}]
If
\[
G=G_{m,k}:=\Z/m\Z\times\Z/mk\Z,
\]
then the principle of inclusion-exclusion and Corollary~\ref{matrix interpretation for M_p(t,n)} imply that the probability of choosing an elliptic curve with group $G=G_{m,k}$ is given by
\als{
\prob_{\CC_p}(E(\F_p)\cong G)
	&= \sum_{j^2|k} \mu(j) \prob_{\CC_p}(a_p(E)=t,\,  E(\F_p)[jm]\cong G_{jm,1}) \\
	&= f_\infty(t,p)\cdot \sum_{\substack{j^2\mid k\\ jm\mid p-1}}\mu(j)\prod_\ell f_\ell(t,p,jm),
}
where $\mu(j)$ denotes the usual M\"obius function. By the definition of $f_\ell(t,p,jm)$, we have that $f_\ell(t,p,jm) = f_\ell(t,p,\ell^{\nu_\ell(j)}m)$. Therefore, if $\CP$ denotes the set of primes $\ell$ with $\ell^2|k$ and $\ell|(p-1)/m$ and $\CS(\CP)$ the set of integers composed only of primes from $\CP$, then
\als{
\prob_{\CC_p}(E(\F_p)\cong G)
	&=  f_\infty(t,p)
\sum_{ j \in \CS(\CP) } \mu(j) \prod_{\ell| j} f_\ell(t,p,\ell m)
	\prod_{\ell\nmid j} f_\ell(t,p,m) \\
	&= f_\infty(t,p) \cdot \left(\prod_{\ell\notin \CP} f_\ell(t,p,m)\right)
		\prod_{\ell|\CP} (f_\ell(t,p,m) - f_\ell(t,p,\ell m))
}
by inclusion-exclusion. Note that if $\ell\notin \CP$, then $f_\ell(t,p,\ell m)=0$ by the discussion preceding Theorem \ref{Gekeler thm}. So we deduce that
\[
\prob_{\CC_p}(E(\F_p)\cong G)
	= f_\infty(t,p) \cdot
		\prod_{\ell} (f_\ell(t,p,m) - f_\ell(t,p,\ell m))
	=f_\infty(t,p) \prod_\ell f_\ell(G,p) ,
\]
which completes the proof of the theorem.
\end{proof}


\section{Sums of Euler products}\label{singular series}

In this section, we provide a unified framework under which Theorems \ref{LT}-\ref{MEG} fall. Before we state the general set-up in which we will work in, we use Theorem \ref{LT} as a working example to describe our main idea.

In view of Theorem \ref{Gekeler}, Theorem \ref{LT} (or, rather, a soft version of this theorem) is reduced to showing that
\[
\sum_{p\le x} f_\infty(t,p)
	\prod_{\ell} f_\ell(t,p) \sim C_{\text{LT}}(t)  \frac{\sqrt{x}}{\log x}  \quad(x\to\infty) ,
\]
where $C_{\text{LT}}(t)$ is as in the statement of this theorem. We set
\[
\delta_\ell(a) = {\bf 1}_{\ell\nmid a} \cdot (f_\ell(t,a)-1).
\]
Since
\[
f_\infty(t,p) f_p(t,p) \sim \frac{1}{\pi\sqrt{p}}
\]
for large primes $p$, we find that
\[
\sum_{p\le x} f_\infty(t,p)  \prod_{\ell} f_\ell(t,p)
	\sim \sum_{p\le x} \frac{1}{\pi \sqrt{p}} \prod_{\ell} (1+\delta_\ell(p)) ,
\]
which reduces Theorem \ref{LT} to showing that
\[
\sum_{p\le x}  \frac{1}{\sqrt{p}} \prod_{\ell} (1+\delta_\ell(p))
	 \sim  \frac{\pi \cdot C_{\text{LT}}(t) }{2} \sum_{p\le x} \frac{1}{\sqrt{p}}
	 	\sim \pi \cdot  C_{\text{LT}}(t)  \cdot \frac{\sqrt{x}}{\log x},
\]
where the last estimate is a consequence of the Prime Number Theorem. If we define a probability measure on the primes $p\le x$ via the relation
\[
\E_{p\le x}[f(p)] = \frac{\sum_{p\le x}f(p)/\sqrt{p}}{\sum_{p\le x}1/\sqrt{p}} ,
\]
then we need to show that
\[
\E_{p\le x}\left[\prod_{\ell} (1+\delta_\ell(p))\right] \sim \frac{\pi}{2} \cdot C_{\text{LT}}(t)
	= \prod_\ell \frac{\ell \cdot \#\{\sigma\in \GL_2(\Z/\ell\Z) : \tr(\sigma)\equiv t\mod{\ell}\}}
		{|\GL_2(\Z/\ell\Z)|} .
\]
Now, $\delta_\ell(p)$ depends only on the congruence class of $p \mod{\ell^r}$ for some appropriate $r$, and the residue classes in which $p$ lies modulo powers of different primes $\ell^r$ should behave independent from each. Therefore, it is reasonable to expect that
\[
\E_{p\le x}\left[\prod_{\ell} (1+\delta_\ell(p))\right]
	\sim \prod_{\ell} \bigg(1+\E_{p\le x}\left[\delta_\ell(p)\right] \bigg) .
\]
Note that if $\ell\nmid p$, then $\delta_\ell(p) = \lim_{r\to\infty} \Delta_{\ell^r}(p)$, where
\[
\Delta_{\ell^r}(p) = -1 + \frac{\phi(\ell^r)\ell^r\cdot \#\left\{ \sigma\in \GL_2(\Z/\ell^r\Z)  :
			\begin{array}{l}
				\tr(\sigma)\equiv t\mod{\ell^r} \\
				\det(\sigma)\equiv p \mod{\ell^r}
			\end{array}\right\} }
		{ |\GL_2(\Z/\ell^r\Z)|} .
\]
Clearly, the function $\Delta_{\ell^r}$ is $\ell^r$-periodic and its mean value over $(\Z/\ell^r\Z)^*$ is
\als{
\frac{1}{\phi(\ell^r)}
	\sum_{a\in(\Z/\ell^r\Z)^*} \Delta_{\ell^r}(a)
		&= -1 + \frac{\ell^r\cdot \#\{ \sigma\in \GL_2(\Z/\ell^r\Z)  :
				\tr(\sigma)\equiv t\mod{\ell^r} \} }
		{ |\GL_2(\Z/\ell^r\Z)|}  \\
		&= -1 + \frac{\ell\cdot \#\{ \sigma\in \GL_2(\Z/\ell\Z)  :
				\tr(\sigma)\equiv t\mod{\ell} \} }
		{ |\GL_2(\Z/\ell\Z)|}   .
}
Since the primes $p$ are well distributed in reduced arithmetic progressions mod $\ell^r$, we should then have that
\als{
\E_{p\le x}[\delta_\ell(p)]\sim \Delta_\ell
		&:=  -1 + \frac{\ell\cdot \#\{ \sigma\in \GL_2(\Z/\ell\Z)  :
				\tr(\sigma)\equiv t\mod{\ell} \} }
		{ |\GL_2(\Z/\ell\Z)|}   ,
}
which yields Theorem \ref{LT} heuristically.

Of course, there are several stumbling blocks in the road map laid above. First of all, the assumption that different primes behave independently from each other is only true asymptotically, and for small primes. Therefore, the first thing we need to do is to truncate the product $\prod_{\ell}(1+\delta_\ell(p))$. This can be indeed accomplished because of Theorem \ref{Gekeler thm}(c), which implies that
\eq{tail-hyp}{
\delta_\ell(p) = \frac{\leg{t^2-4p}{\ell}}{\ell} + O\left(\frac{1}{\ell^2}\right),
}
unless $\ell$ is one of the finitely many prime divisors of $t^2-4p$. Estimating sums of the form
\eq{tail}{
\sum_{\ell>z} \frac{\leg{t^2-4p}{\ell}}{\ell}
}
is related to our knowledge about the zeroes of the Dirichlet $L$-function associated to the character $(t^2-4p \, |\ \cdot \ )$. The Generalized Riemann Hypothesis would imply that the sum in \eqref{tail} is small as soon as $z>(\log d)^{2+\epsilon}$, where $d$ is the conductor of the character $(t^2-4p\, | \ \cdot\ )$. However, unconditionally, we only know that the sum in \eqref{tail} is small for $z>\exp\{d^{\epsilon}\}$, which is a much stronger restriction. This problem can be rectified by appealing to zero-density estimates which guarantee that, for most $p$, the sum in \eqref{tail} is small as soon as $z>(\log d)^A$, with $A$ a large enough constant. This is good enough for our purposes and allows us for most primes $p\le x$ to replace the product $\prod_{\ell}(1+\delta_\ell(p))$ by $\prod_{\ell\le (\log x)^A}(1+\delta_\ell(p))$ with a very small total error. Then, we expand this short product to find that
\[
\sum_{p\le x}\frac{1}{\sqrt{p}} \prod_{\ell\le (\log x)^A}  (1+\delta_\ell(p))
	= \sum_{\ell|n\,\Rightarrow\, \ell\le (\log x)^A}\mu^2(n) \sum_{p\le x} \frac{ \delta_n(p)}{\sqrt{p}} ,
\]
where, for convenience, we have set $\delta_n(p)=\prod_{\ell|n}\delta_\ell(p)$. The next crucial step is that, for all $p$ with $\ell^r\nmid t^2-4p$, Theorem \ref{Gekeler thm}(c) implies that $\delta_\ell(p)= \Delta_{\ell^r}(p)$, and the function $\Delta_{\ell^r}$ is $\ell^r$-periodic. Setting
\[
\Delta_q(a) = \prod_{\ell^r\|q} \Delta_{\ell^r}(a),
\]
we find that
\[
 \sum_{p\le x} \frac{ \delta_n(p)}{\sqrt{p}}
 	= \sum_{q\in\N,\ \rad(q)=n} \sum_{\substack{p\le x \\ \nu_{\ell}(t^2-4p)=\nu_\ell(q)-1 \\ \forall \ell|n}}
			\frac{ \Delta_q(p)}{\sqrt{p}}
	= \sum_{\substack{ q\in\N \\  \rad(q)=n }}
		\sum_{a\in \CH(q)} \Delta_q(a)
		\sum_{\substack{p\le x \\ p\equiv a\mod{q}}} \frac{1}{\sqrt{p}} ,
\]
where
\[
\CH(q)
	= \{a\in (\Z/ q\Z)^* : \ell^r \nmid t^2-4a,\ \ell^{r-1}\mid t^2-4a
			\quad\text{whenever}\ \ell^r\|q\} .
\]
We then use the Bombieri-Vinogradov theorem in order to control the number of primes in arithmetic progressions on average. We also have to use some more trivial arguments when the modulus $q$ is too large, exploiting the fact that this is a $(\log x)^A$-smooth number and there are very few such numbers. We are then left with the task of showing that
\[
\sum_{r=1}^\infty  \frac{1}{\phi(\ell^r)} \sum_{a\in \CH(\ell^r)} \Delta_{\ell^r}(a) =  \Delta_\ell .
\]
Indeed, we have that
\[
\sum_{r=1}^R  \frac{1}{\phi(\ell^r)} \sum_{a\in \CH(\ell^r)} \Delta_{\ell^r}(a)
	= \frac{1}{\phi(\ell^R)} \sum_{\substack{a\in(\Z/\ell^R\Z)^* \\ \ell^R\nmid t^2-4a}} \Delta_{\ell^R}(a)
	= \Delta_\ell- \frac{1}{\phi(\ell^R)} \sum_{\substack{a\in(\Z/\ell^R\Z)^* \\ \ell^R |  t^2-4a}} \Delta_{\ell^R}(a) ,
\]
which is easily seen to tend to $\Delta_\ell$ as $R\to\infty$, since the congruence $t^2-4a\not\equiv 0\mod{\ell^R}$ has at most 8 solutions $a\mod{\ell^R}$.


\subsection{General axiomatic framework} \label{singular series AF}

We describe here in rather abstract terms the general set-up in which we work to show Theorems \ref{LT}-\ref{MEG}. We fix a natural number $d$ and a set
\[
\CA\subset([-X,X]\cap\Z)^d,
\]
where $X$ is some parameter that we consider given from now on. In general, we denote $d$-dimensional vectors with bold letters, e.g. $\bs x$ or $\bs a$, and we index their coordinates as $\bs x=(x_1,\dots,x_d)$, $\bs a=(a_1,\dots,a_d)$, etc. Moreover, given $\bs a,\bs b\in\Z^d$ and $q\in\N$, we write $\bs a\equiv \bs b\mod q$ if $a_j\equiv b_j\mod q$, for all $j\in\{1,\dots,d\}$. Similarly, we write $\bs a\mod q$ to denote the vector $(a_1\mod q,\dots,a_d\mod q)$.

In addition, we fix a set of complex weights $(w_{\bs a})_{\bs a\in \CA}$ and we set
\[
W= \sum_{\bs a \in \CA} w_{\bs a} .
\]
In practice, we cannot handle weights for which $W$ is significantly smaller than $\sum_{\bs a\in\CA}w_{\bs a}$. We simply allow $w_{\bs a}$ to be complex numbers to gain some extra flexibility. Next, for each prime $\ell$, we consider a set $\CG(\ell)\subset (\Z/\ell\Z)^d$, and for $q\in\N$ we set
\begin{equation} \label{def-Gq}
\CG(q) = \left\{\bs g\in (\Z/q\Z)^d :
	\bs g\mod \ell \in \CG(\ell)\ \mbox{for all primes $\ell|q$}\right\} .
\end{equation}
We think of $\CA$ as being well-distributed among the elements of $\CG(q)$. To this end, we set
\[
E(\CA;q) = \max_{\bs g\in \CG(q)} \left| \sum_{\substack{ \bs a\in\CA \\ \bs a\equiv \bs g\mod{q}  }} w_{\bs a}  - \frac{W}{|\CG(q)|} \right|
\]
for all $q\in\N$. In order to avoid working with very sparse sets $\CA$, we also assume that
\eq{G}{
|\CG(\ell)| \gg \ell^d \quad(\ell\ \text{prime})  .
}
Finally, we set
\eq{Q}{
Q = \exp\{(\log\log X)^2\}
}
and we assume that $\CA$ does not contain many more elements than it should in each residue class $\bs g \in\CG(q)$ for $q\le Q$. To handle the case when we don't have good estimates for $W$ (see, for example, the proof of Theorem \ref{Koblitz}), we assume the existence of a quantity $\W$, which we heuristically think of comparable size with $W$, such that:
\eq{bt for A}{
\sum_{\substack{ \bs a\in \CA \\ \bs a\equiv \bs g\mod{q}  }} |w_{\bs a}|
	\ll \frac{\W}{|\CG(q)| }  \quad(q\le Q,\ \bs g\in\CG(q)) .
}

In many of the applications, $d=1$ and $\CA$ is taken to be the set of primes in an interval or the set of integers in an interval and, respectively, $\CG(q)=(\Z/q\Z)^*$ or $\CG(q)=\Z/q\Z$, so the reader can work with these two simple examples in mind.

We are going to average certain Euler products over our set $\CA$. We consider a sequence of complex numbers $\{\delta_\ell(\bs a): \bs a\in\CA,\ \ell\ \text{prime}\}$ and set
\[
P_{\bs a}= \prod_{\ell\ \text{prime}}(1+\delta_\ell(\bs a)) .
\]
Our goal is to estimate the sum
\[
\sum_{\bs a\in \CA} w_{\bs a} P_{\bs a}.
\]
Of course, we do not even know whether the infinite product $P_{\bs a}$ converges, so we certainly need to impose some conditions on the numbers $\delta_\ell(\bs a)$. Indeed, we assume that there is an absolute constant $\eta>0$ and an integer $k\ge0$ such that the following conditions hold:

\begin{enumerate}
\renewcommand{\labelenumi}{(\arabic{enumi})}
\item $\delta_\ell(\bs a) = 0$ if $\bs a\mod{\ell}\notin \CG(\ell)$.

\item  $\delta_\ell(\bs a)\ll 1/\ell$ if $\bs a\mod \ell \in\CG(\ell)$.

\item For each $j\in\{1,\dots,k\}$ and each $\bs a\in \CA$, there are non-principal Dirichlet characters $\chi_{j,\bs a}$ mod $M_{j,\bs a}$, an integer $L_{\bs a}\ge1$, and complex coefficients $\lambda_{j,\bs a}$ such that for all primes $\ell\nmid L_{\bs a}$, we have that $\bs a\mod \ell\in \CG(\ell)$ and
\[
\delta_\ell(\bs a)
	= \frac{\lambda_{1,\bs a}\chi_{1,\bs a}(\ell)+\cdots+\lambda_{k,\bs a}\chi_{k,\bs a}(\ell)}{\ell}
		+ O\left(\frac{1}{\ell^{1+\eta}}\right) .
\]

\item For every prime $\ell\le Q$ and every exponent $r\ge1$, there is a set $\CE(\ell^r)\subset (\Z/\ell^r\Z)^d$ and a function $\Delta_{\ell^r}: \Z^d\to \C$ such that:
	\begin{enumerate}
		\item $\Delta_{\ell^r}$ is $\ell^r$-periodic;
		\item $\delta_\ell(\bs a)=\Delta_{\ell^r}(\bs a)$ if $\bs a\mod{\ell^r} \in \CG(\ell^r)\setminus \CE(\ell^r)$, where
			the sequence of sets $\{\bs a\in \Z^d: \bs a\mod{\ell^r} \in \CE(\ell^r)\}_{r\ge1}$
			is decreasing and its intersection is contained in $\Z^d\setminus\CA$;
		\item $\Delta_{\ell^r}$ vanishes on average over the set $\CG(\ell^r)$ as $r\to\infty$, that is to say\footnote{Our assumption that $\Delta_{\ell^r}$ is 0 on average does not harm generality significantly. The case when its average is some number $\Delta_\ell$ follows from the case $\Delta_\ell=0$ by considering sequence $\delta_\ell'(\bs a)$ instead, also supported on those $\bs a\in \CG(\ell)$, where $1+\delta_\ell(\bs a) = (1+\Delta_\ell)(1+\delta_\ell'(\bs a))$ when $\bs a\mod\ell\in \CG(\ell)$. We would also need to assume then that the series $\sum_{\ell} |\Delta_\ell|$ converges fast enough. This argument will be used later on.}
			\[
			\lim_{r\to\infty} \frac{1}{|\CG(\ell^r)|} \sum_{\bs a\in \CG(\ell^r)} \Delta_{\ell^r}(\bs a) = 0 ;
			\]
		\item $\| \Delta_{\ell^r}\|_\infty \ll_\ell 1$, for all $r\ge1$.
	\end{enumerate}
\item For all $j\in\{1,\dots,k\}$ and all $\bs a\in \CA$, we have $M_{j,\bs a}\le X^{O(1)}$, $\omega(L_{\bs a})\le(\log X)^{O(1)}$ and $\lambda_{j,\bs a}\ll1$.
\end{enumerate}

\noindent Under these assumptions and notations, we have the following result.


\begin{thm}\label{singular series average general}
Assume the above set-up and fix $\epsilon>0$, $\lambda>1$ and $C\ge1$. Then
\[
\sum_{\bs a\in \CA} w_{\bs a} P_{\bs a}  =  W
	+ O\left(  \frac{ e^{O(S)}  \W }{(\log X)^C} + M X^\epsilon
		+ (\log\log X)^{O(1)} e^{O(S)}  \W^{1/\lambda} E^{1-1/\lambda} \right) ,
\]
where
\[
M= \max_{\substack{1\le j\le k \\ c\ge2}}
	\sum_{\substack{ a\in\CA \\ \cond(\chi_{j,\bs a})=c}} |w_{\bs a}| ,
\qquad
E = \sum_{q\le Q} q^{d-1} E(\CA;q)  ,
\]
and
\[
S = \sum_{\ell \le Q} \sum_{r=1}^\infty \frac{(|\CE(\ell^r)| / \ell^{r(d-1)})^\lambda}{\ell^{r+1}} ,
\]
with $Q$ is defined by \eqref{Q} and $\cond(\chi)$ denoting the conductor of the Dirichlet character $\chi$. All implied constants depend at most on $d,k,\lambda,\eta,\epsilon,C$ and the implicit constants in conditions {\rm(1)\,-\,(5)} above and in relations \eqref{G} and \eqref{bt for A}.
\end{thm}

\begin{rmk}
When $k=0$, then Property (3) states that $\delta_\ell(\bs a)=1+O(1/\ell^{1+\eta})$ whenever $\ell\nmid L_{\bs a}$. This is much stronger than the bound \eqref{bound-delta-n} which is used in the proof of Theorem \ref{singular series average general} in Section \ref{singular series proof}, and
in this case, the error term $e^{O(S)}\widetilde{W}/(\log X)^C$ can be replaced by $e^{O(S)}\widetilde{W}/X^\alpha$ for some $\alpha>0$, provided that \eqref{bt for A} holds with $Q=X^\beta$ for some $\beta>0$ (and then $\alpha$ depends on $\beta$). The key observation is that, if $\delta_n(\bs a)=\prod_{\ell|n} \delta_\ell(\bs a)$, we then have that
\[
\sum_{n=1}^\infty \mu^2(n) n^\epsilon |\delta_n(\bs a)|
	\le \prod_{\ell|L_{\bs a}} \left(1+\frac{O(1)}{\ell^{1-\epsilon}}\right)
		\prod_{\ell\nmid L_{\bs a}} \left(1+\frac{O(1)}{\ell^{1+\eta-\epsilon}}\right)
	\ll \prod_{\ell|L_{\bs a}} \left(1+\frac{1}{\ell^{1-\epsilon}}\right)^{O(1)},
\]
for any fixed $\epsilon<\eta$. Since $\omega(L_{\bs a})\le(\log X)^K$ for some $K\ge1$ by property (5) above, we deduce that
\[
\prod_{\ell|L_{\bs a}} \left(1+\frac{1}{\ell^{1-\epsilon}}\right)
	\ll \prod_{\ell \le (\log X)^{K/(1-\epsilon)}} \left(1+\frac{1}{\ell^{1-\epsilon}}\right)
	\ll_{\epsilon,K} e^{(\log X)^{\epsilon K/(1-\epsilon)}} = X^{o(1)} ,
\]
provided that $\epsilon<1/(K+1)$. This allows us to handle the tails of various summations in the proof via Rankin's trick. For example, we have the bound
\[
\sum_{n>N} \mu^2(n) |\delta_n(\bs a)|
	\le \frac{1}{N^\epsilon} \sum_{n=1}^\infty \mu^2(n) n^\epsilon |\delta_n(\bs a)|
	= \frac{X^{o(1)}}{N^\epsilon},
\]
which is good enough for our purposed by choosing $\epsilon<\min\{\eta,1/(K+1)\}$ and $N$ to be an appropriate power of $X$. We do not pursue this strengthening of the error term when $k=0$ since it is not necessary for our purposes.
\end{rmk}


\subsection{A simplified set of axioms} \label{singular series SF}

For the applications we have in mind, we can simplify further the conditions of Theorem \ref{singular series average}. Instead of \eqref{G}, we suppose that $|\CG(\ell)|$ is usually very close to $\ell^d$:
\begin{equation}\label{G-strong}
\#\CG(\ell) = \ell^d + O(\ell^{d-\eta}),
\end{equation}
where $\eta > 0$ is some fixed absolute constant. We notice once and for all that $\CG(\ell) = (\Z/\ell \Z)^d$ and $\CG(\ell) = ((\Z/\ell \Z)^*)^d$ satisfy this condition.

We continue assuming conditions (1) and (2), but conditions (3)-(5) are simplified as we describe below. Here and for the rest of this paper, given a polynomial $f\in\Z[x_1,\dots,x_n]$, we write $\SC(f)$ for its content, that is to say, the greatest common divisor of its coefficients, and $H(f)$ for its height, that is to say the maximum absolute modulus of its coefficients. With this notation, we postulate the existence of some polynomials $D_j(x_1,\dots,x_d)$, $1\le j\le k$, and $F(x_1,\dots,x_k)$ over $\Z$, an integer $L\ge1$ and some complex coefficients $\lambda_j$, $1\le j\le k$ satisfying the following hypotheses:

\begin{itemize}
\item[(3')] For each $j\in\{1,\dots,k\}$ and each $\bs a\in \CA$, $D_j(\bs a)$ is a discriminant (i.e. $D_j(\bs a)$ is not a perfect square and $D_j(\bs a)\equiv0,1\mod{4}$) and if $\ell\nmid L \cdot D_1(\bs a)\cdots D_k(\bs a)$, then $\bs a\mod \ell\in \CG(\ell)$ and
\[
\delta_\ell(\bs a)
	= \frac{ \lambda_1\leg{D_1(\bs a)}{\ell} +  \cdots + \lambda_k\leg{D_k(\bs a)}{\ell} }{\ell}
		+ O\left(\frac{1}{\ell^{1+\eta}}\right) .
\]

\item[(4')] For every prime $\ell$ and every exponent $r\ge1$,
there is a function $\Delta_{\ell^r}: \Z^d\to \C$ such that:
	\begin{enumerate}
		\item $\Delta_{\ell^r}$ is $\ell^r$-periodic;
		\item $\delta_\ell(\bs a)=\Delta_{\ell^r}(\bs a)$ if $\bs a\mod{\ell^r} \in
			\{ \bs g\in \CG(\ell^r): F(\bs g)\not\equiv 0\mod{\ell^r}\}$;
		\item $\Delta_{\ell^r}$ has a mean value as $r\to\infty$ over $\CG(\ell^r)$, that is to say there is
		a $\Delta_\ell\in\C$ such that
			\[
			\lim_{r\to\infty} \frac{1}{|\CG(\ell^r)|} \sum_{\bs a\in \CG(\ell^r)}
				\Delta_{\ell^r}(\bs a) = \Delta_\ell .
			\]
		Moreover, $|1+\Delta_\ell|\gg 1$.
		\item $\| \Delta_{\ell^r} \|_\infty \ll 1/\ell$, for all $r\ge1$.
	\end{enumerate}
\item[(5')] We have\footnote{The polynomials $F, D_1,\dots,D_k$ and the parameters $L$ and $\lambda_j$ might depend on $X$ or on some other parameters whose size is controlled by $X$, and we are majoring here this dependence.} $\omega(L)\le (\log X)^{O(1)}$, $\SC(F)\ll1$, and $\lambda_j\ll1$
for $1\le j\le k$. Moreover, $F(\bs a)\neq0$ for all $\bs a\in \CA$. Finally, for each $j\in\{1,\dots,k\}$, we have that $H(D_j)\le X^{O(1)}$ and the polynomials $\pm D_j/\SC(D_j)$ are not perfect squares in the ring $\Z[x_1,\dots,x_d]$.
\end{itemize}


\begin{thm}\label{singular series average} \label{dimitris-magical}
Assume that the above simplified set-up holds and fix $\epsilon>0$ and $C\ge1$. Moreover, suppose that $\CA\subset \{\bs a\in \Z^d:\bs a\mod{\ell}\in\CG(\ell)\}$ for each prime $\ell\le Q$. Then the infinite product
\[
P:= \prod_\ell (1+\Delta_\ell)
\]
converges absolutely and we have that
\[
\sum_{\bs a\in \CA} w_{\bs a} P_{\bs a}  =  P \cdot \left( W
	+ O\left(  \frac{\W}{(\log X)^C} + M X^\epsilon
		+ (\log\log X)^{O(1)} \W^{m/(m+1)} E^{1/(m+1)} \right) \right),
\]
where $m=\deg(F)$,
\[
M= \max_{\substack{1\le j\le k \\ n\neq0}}
	\sum_{\substack{\bs a\in\CA \\ D_j(\bs a)/n\ \text{is a square}}} |w_{\bs a}| ,
\qquad
E = \sum_{q\le Q}  q^{d-1} E(\CA;q)  ,
\]
and $Q$ is defined by \eqref{Q}. All implied constants depend at most on $d,k,\eta,\epsilon,C,m$ and the implicit constants in conditions {\rm (1), (2), (3'), (4') and (5')}, and in relations \eqref{bt for A} and  \eqref{G-strong}.
\end{thm}


\section{Applications of Theorem \ref{singular series average} } \label{applications}

This section is devoted to the proof of Theorems \ref{LT}-\ref{MEG}.


\subsection{The average Lang-Trotter conjecture} We prove here Theorem \ref{LT}. By a dyadic decomposition argument, it is enough to show that
\[
\sum_{x<p\le 2x} \prob_{\CC_p}(a_p(E)=t) =C_{\text{LT}}(t)
	\int_x^{2x}\frac{\dee u}{2\sqrt{u}\log u} + O_t\left(\frac{\sqrt{x}}{(\log x)^A}\right),
\]
We set
\[
w_p=  f_p(t,p) f_\infty(t,p)
	= \frac{1}{\pi \sqrt{p} }  +   O_t\left(\frac{1}{x}\right) .
\]
Furthermore, we set
\eq{def-delta}{
\delta_\ell(p) =
\begin{cases}
	\ds f_\ell(t,p) - 1
		&\text{if}\ \ell\neq p,\\
		0			
		&\text{otherwise} .
\end{cases}
}
With this notation, we see that
\eq{prob-Gekeler}{
\prob_{\CC_p}(a_p(E)=t) = w_p \prod_\ell (1+\delta_\ell(p)) .
}
We shall apply Theorem \ref{singular series average} with $d=k=1$, $\CA=\{x<p\le 2x\}$, $\CG(\ell)=(\Z/\ell\Z)^*$, $F(a)=t^2-4a$, $D_1(a)=a^2(t^2-4a)$ (the factor $a^2$ is added to guarantee that if $\ell\nmid D_1(a)$, then $a\mod \ell\notin\CG(\ell)$), $L=1$ and $X=2x$. We need to check that the necessary conditions are satisfied. Condition (1) holds by definition and conditions (2) and (3') follow from parts (d) and (c) of Theorem \ref{Gekeler thm}, respectively. Condition (4') holds with
\[
\Delta_{\ell^r}(a) = -1 + \frac{\ell^r\phi(\ell^{r}) \cdot\#\left\{\sigma\in\GL_2(\Z/\ell^r\Z) :
	\begin{array}{l}
	\tr(\sigma)\equiv t\pmod{\ell^r},\\
	\det(\sigma)\equiv a\pmod{\ell^r}
	\end{array}\right\}}
	{|\GL_2(\Z/\ell^r\Z)|}  ,
\]
which satisfies conditions (4'a)-(4'd) with average value
\als{
\Delta_\ell &:= -1 + \lim_{r\to\infty} \frac{\ell^r \cdot  \#\{ \sigma\in\GL_2(\Z/\ell^r\Z) :
	\tr(\sigma)\equiv t \mod {\ell^r}\} }
			{|\GL_2(\Z/\ell^r\Z)|} \\
		&=-1 +
		\frac{\ell \cdot  \#\{ \sigma\in\GL_2(\Z/\ell \Z) :  \tr(\sigma)\equiv t \mod \ell\} }{|\GL_2(\Z/ \ell \Z)|} .
}
(Here we use parts (b) and (d) of Theorem \ref{Gekeler thm} to see conditions (4'b) and (4'd), respectively.) Finally, it is easy to verify condition (5'), and relation \eqref{bt for A} follows easily by the Brun-Titchmarsh inequality with $\W=\sqrt{x}/\log x$.

In conclusion, we may apply Theorem \ref{dimitris-magical}. This will complete the proof of Theorem \ref{LT}, as long as we can control the quantities $W$, $E$ and $M$ that appear there. We use the Prime Number Theorem to see that
\als{
W = \sum_{x<p\le 2x} w_p
	 =  \frac{2}{\pi} \sum_{x<p\le 2x} \left(\frac{1}{2\sqrt{p}}
	 	+ O_t\left(\frac{1}{x}\right)\right) \\
 	= \frac{2}{\pi} \int_{x}^{2x} \frac{\dee u}{2\sqrt{u}\log u}
		+ O_t\left(\frac{\sqrt{x}}{(\log x)^A}\right) .
}
We use the Bombieri-Vinogradov theorem to see that $E\ll \sqrt{x}/(\log x)^B$ for any fixed $B$, and finally, we have that
\[
M  = \max_{n\le -4} \sum_{\substack{x<p\le 2x \\ (t^2-4p)/n\ \text{is a square}}} |w_p|
	\ll  \max_{n\le -4}  \frac{\#\{m\in\Z: 4x< |n|m^2+t^2 \le 8x\}}{\sqrt{x}}
	\ll 1 ,
\]
an estimate that is good enough for our purposes. This completes the proof of Theorem \ref{LT}.


\subsection{The vertical Sato-Tate conjecture}

In this section, we prove Theorem \ref{ST}. Clearly, it suffices to consider the case when $p^{\epsilon-1/2}\le \beta-\alpha\le 2p^{\epsilon-1/2}$; the general case will follow by dividing the interval $[\alpha,\beta]$ into shorter intervals. We start by noting that
\[
\prob_{\CC_p}\left(\alpha\le \frac{a_p(E)}{2\sqrt{p}}\le \beta\right)
	= \sum_{2\alpha\sqrt{p}\le t\le 2\beta\sqrt{p}} \prob_{\CC_p}(a_p(E)=t)  .
\]
So here $p$ is fixed and the averaging is performed over $t\in I:= [2\alpha\sqrt{p},2\beta\sqrt{p}]$. To this end, we let $\CA= I\cap\Z$, $\CG(\ell)=\Z/\ell\Z$, $w_t = f_\infty(t,p)f_p(t,p)$ and $\delta_\ell(t) = {\bf 1}_{\ell\neq p}\cdot(f_\ell(t,p) -1)$,
so that
\[
\prob_{\CC_p}\left(\alpha\le \frac{a_p(E)}{2\sqrt{p}}\le \beta\right) = \sum_{t\in \CA} w_t \prod_\ell (1+\delta_\ell(t)) .
\]

We are going to apply Theorem \ref{singular series average} with $k=d=1$, $D_1(t)=F(t)=t^2-4p$, $L=p$ and $X=2\sqrt{p}$. We need to check that the necessary conditions are satisfied. Condition (1) holds by definition and conditions (2) and (3') follow from parts (d) and (c) of Theorem \ref{Gekeler thm}, respectively. Condition (4') holds with
\[
\Delta_{\ell^r}(t)
		= -1 + \frac{\ell^r \cdot\#\left\{\sigma\in\GL_2(\Z/\ell^r\Z) :
			\begin{array}{l}
			\tr(\sigma)\equiv t\pmod{\ell^r},\\
			\det(\sigma)\equiv p\pmod{\ell^r}
			\end{array}\right\}}
			{\#\left\{\sigma\in\GL_2(\Z/\ell^r\Z) : \det(\sigma)\equiv p\pmod{\ell^r}\right\}}
\]
when $\ell\neq p$ and $\Delta_{p^r}(t)=0$, which satisfies conditions (4'a)-(4'd) with $\Delta_\ell=0$ (here we use parts (a) and (d) of Theorem \ref{Gekeler thm} to see conditions (4'b) and (4'd), respectively). Moreover, it is easy to verify condition (5').

In conclusion, we may indeed apply Theorem \ref{singular series average}, provided that we verify that \eqref{bt for A} is satisfied, which is not as obvious as before. We do this below (we shall take $\W = W$), and we also estimate the quantities $W$, $E$ and $M$ appearing in Theorem \ref{singular series average}. This is a bit more delicate than in the proof of Theorem \ref{LT}. Without loss of generality, we assume that $\alpha\ge0$; the case $\alpha<0$ is treated in an analogous way.

Let $\eta=1-\alpha\ge \beta-\alpha \ge p^{-1/2+\epsilon}$. Since $\alpha \geq 0$, we immediately have that
\eq{ST e1}{
\sqrt{1-u^2} \le \sqrt{1-\alpha^2} \asymp \sqrt{\eta}  \quad(\alpha\le u\le\beta) .
}
Moreover, we claim that
\eq{ST e2}{
 \frac{2}{\pi} \int_\alpha^\beta \sqrt{1-u^2} \dee u \asymp \sqrt{\eta}(\beta-\alpha) \asymp \frac{\sqrt{\eta}}{p^{1/2-\epsilon}} .
}
The implicit upper bound follows immediately by \eqref{ST e1}. For the lower bound, we separate two cases. Firstly, if $\beta\le1-\eta/2$, then we immediately see that $\sqrt{1-u^2}\asymp \sqrt{\eta}$ for all $u\in[\alpha,\beta]$. Finally, if $\beta\ge 1-\eta/2 = \alpha +
\eta/2$, then $\eta\ge\beta-\alpha\ge\eta/2$. Therefore,
\[
\frac{2}{\pi} \int_\alpha^\beta \sqrt{1-u^2} \dee u
	\ge \frac{2}{\pi} \int_\alpha^{\alpha+\eta/2} \sqrt{1-u^2} \,\dee u
	\gg \eta \cdot \sqrt{1-(\alpha+\eta/2)} \asymp \eta\cdot \sqrt{\eta} \asymp \sqrt{\eta}(\beta-\alpha) ,
\]
which completes the proof of \eqref{ST e2}.

Next, for any $q\le p^{\epsilon/2}$ and $a\in\Z$, partial summation implies that
\als{
\sum_{\substack{ t\in I \\ t\equiv a\mod{q} }} f_\infty(t,p)
		&=\frac{1}{\pi\sqrt{p}} \int_{2\alpha\sqrt{p}}^{2\beta\sqrt{p}}
		\sqrt{ 1- \left(\frac{t}{2\sqrt{p}}\right)^2}  \, \dee \left( \frac{t}{q} + O(1) \right)  \\
	&=\frac{2}{\pi q} \int_\alpha^\beta \sqrt{ 1- u^2}  \, \dee u
		 + O\left(  \sqrt{ \frac{\eta}{p} }
		 	+ \frac{1}{\sqrt{p}} \int_{2\alpha\sqrt{p}}^{2\beta\sqrt{p}}
				\left| \frac{\dee}{\dee t} \sqrt{ 1- \left(\frac{t}{2\sqrt{p}}\right)^2} 	 \right| \dee t \right)  \\
	&= \frac{2}{\pi q} \int_\alpha^\beta \sqrt{1-u^2} \, \dee u
	 + O\left( \sqrt{ \frac{\eta}{p} }  \right)\\
&
	 =\left(1+O\left(\frac{q}{p^{\epsilon}} \right)\right)
	   \frac{2}{\pi q} \int_\alpha^\beta \sqrt{1-u^2} \, \dee u .
}
where we used \eqref{ST e2} to get the last line.
Since $f_p(t,p)=1+O(1/p)$, we deduce that
\[
\sum_{\substack{ t\in I \\ t\equiv a\mod{q} }} w_t
	 =\left(1+O\left(\frac{q}{p^{\epsilon}} \right)\right)
	   \frac{2}{\pi q} \int_\alpha^\beta \sqrt{1-u^2} \, \dee u .
\]
In particular, \eqref{bt for A} holds for all $q\le p^{\epsilon/2}$ with $\W=W$. Moreover, for the quantity $E$ appearing in the statement of Theorem \ref{singular series average}, we have that
\[
E \ll \frac{e^{(\log\log 2p)^2}}{p^\epsilon} \cdot \int_\alpha^\beta \sqrt{1-u^2} \dee u .
\]
Finally, for the quantity $M$, we have the estimate
\als{
M = \max_{n\le -4} \sum_{\substack{t\in \CA \\ D(t)/n\ \text{is a square}}} w_t
	&\ll \sqrt{\frac{\eta}{p} }  \max_{n\le-4} \#\{(t,m)\in\Z: t^2-nm^2=4p \} \\
	&\ll \sqrt{\frac{\eta}{p} }
	\asymp \frac{1}{p^{\epsilon}} \cdot \int_\alpha^\beta \sqrt{1-u^2} \, \dee u ,
}
which is good enough for our purposes. This completes the proof of Theorem \ref{ST}.


\subsection{Elliptic curves with a prime number of points}\label{Koblitz-proof}
We show here how to prove Theorem \ref{Koblitz}. First, we deal with the proof of \eqref{Koblitz-1}. We have that
\eq{Koblitz-gekeler}{
\prob_{\CC_p}(|E(\F_p)|\ \text{prime})
	&= \sum_{\substack{q\ \text{prime} \\ p^-<q<p^+}}
		\prob_{\CC_p}(|E(\F_p)|=q) \\
	&= \sum_{\substack{q\ \text{prime} \\ p^-<q<p^+}} f_\infty(p+1-q, p) \prod_{\ell} f_\ell(p+1-q, p)
}
using Theorem \ref{Gekeler}(a). We set $\CG(\ell) = (\Z/\ell\Z)^*$,
\[
\delta_\ell(a)
 =	\begin{cases}
		f_\ell(p+1-a,p)  - 1
			&\text{if}\ \ell\nmid pa,\\
		0	&\text{otherwise},
\end{cases}
\]
\[
\CA = \{q\ \text{prime}: p^-<q<p^+\}
\]
and
\als{
w_q 	&= f_\infty(p+1-q,p) f_p(p+1-q,p) f_q(p+1-q,p) \\
	& =  \frac{1}{\pi\sqrt{p}}\sqrt{1-\left(\frac{p+1-q}{2\sqrt{p}}\right)^2}
			 \left( 1+O\left(\frac{1}{p}\right)\right) .
}
With this notation, we find that
\[
\prob_{\CC_p}(|E(\F_p)|\ \text{prime})
	= \sum_{q\in \CA} w_q \prod_\ell (1+\delta_\ell(q)) .
\]
We are going to apply Theorem \ref{singular series average} with $d=k=1$, $F(a) = (p+1-a)^2-4p$, $D_1(a)=a^2F(a)$, $L=p$ and $X=2p$. We need to check that the necessary conditions are satisfied. As in the previous applications, condition (1) holds by definition, and conditions (2) and (3') follow immediately by Theorem \ref{Gekeler thm}. The same result implies that condition (4') holds with
\[
\Delta_{\ell^r}(a)
	=
		\ds -1 + \frac{\ell^r \cdot \#\left\{ \sigma\in \GL_2(\Z/\ell^r\Z) :  \begin{array}{l}
	\tr(\sigma) \equiv p+1-a  \mod{\ell^r}, \\
	\det(\sigma)\equiv p  \mod{\ell^r}
	\end{array} \right\} }
	{\#\{\sigma\in \GL_2(\Z/\ell^r\Z): \det(\sigma)\equiv p\mod{\ell^r}\} }
\]
if $\ell\neq p$ and with $\Delta_{p^r}(a)=0$, which satisfies condition (4'c) with
\als{
\Delta_\ell	&=
		-1 +\lim_{r\to\infty} \frac{ \ell^r\cdot \#\left\{\sigma\in \GL_2(\Z/\ell^r\Z) :
			 \begin{array}{l}
			 (\det(\sigma)+1-\tr(\sigma),\ell^r)=1, \\
			 \det(\sigma)\equiv p  \mod{\ell^r}
			 \end{array} \right\}}
			 {\phi(\ell^r)\cdot \#\{\sigma\in \GL_2(\Z/\ell^r\Z): \det(\sigma)\equiv p  \mod{\ell^r}\}} \\
		&= -1 + \frac{\ell}{\ell-1} \cdot \frac{\#\left\{\sigma\in \GL_2(\Z/\ell\Z) :
			 \begin{array}{l}
			 (\det(\sigma)+1-\tr(\sigma),\ell)=1, \\
			 \det(\sigma)\equiv p  \mod \ell
			 \end{array}\right\}}
			 {\#\{\sigma\in \GL_2(\Z/\ell\Z): \det(\sigma)\equiv p  \mod \ell \}}
}
when $\ell\neq p$ and with $\Delta_p=0$. Finally, it is easy to see that condition (5') holds too, and relation \eqref{bt for A} follows easily by the Brun-Titchmarsh inequality with $\W=1/\log p$.

In conclusion, we have shown that Theorem \ref{singular series average} is indeed applicable. This will complete the proof of \eqref{Koblitz-1}, as long as we can control the quantities $W$, $M$ and $E$ there. For $M$, we have that
\eq{Koblitz-M}{
M & = \max_{n\le -4} \sum_{\substack{p^-<q<p^+ \\ D_1(q)/n\ \text{is a square}}} w_q \\
	&\ll \max_{n\le -4}\frac{\#\{(q,m): (p+1-q)^2- nm^2 = 4p \}}{\sqrt{p}} \ll \frac{1}{\sqrt{p}} ,
}
an estimate that is good enough for our purposes. Finally, for $W$ and for $E$, we use the following result with $N=p$.


\begin{lma}\label{W2}
For $1\le b\le h\le \sqrt{N}$ and $(a,b)=1$, we have that
\[
\sum_{\substack{N^-<q< N^+ \\ q\ \text{prime}\\ q\equiv a\mod b }}
			\sqrt{1-\left(\frac{N+1-q}{2\sqrt{N}}\right)^2}
	= \frac{\pi\sqrt{N}}{\phi(b)\log N}
			+ O\left( \frac{h}{b} + \frac{1}{h\log N} \int_{N^-}^{N^+} E(y,h;b) \dee y \right) ,
\]
where $E(y,h;b)$ is defined in \eqref{def-E}.
\end{lma}


\begin{proof} This result follows by Lemma 7.1 in \cite{CDKS2}, since if $N=m^2k$ and $q\in(N^-,N^+)$, then the quantity $d(q)$ there equals
\[
d(q) = \frac{(N+1-q)^2-4N}{m^2} = 4k\cdot \left( \left(\frac{N+1-p}{2\sqrt{N}}\right)^2-1\right) .
\]
\end{proof}


\noindent
Lemma \ref{W2} then implies that, for any $h\in[p^\epsilon,\sqrt{p}/(\log p)^{2A+1}]$, we have
\[
W  =  \frac{1}{\log p}
			+ O\left( \frac{1}{(\log p)^{2A+1}}
				+ \frac{1}{\sqrt{p}\log p} \int_{p^-}^{p^+} \frac{E(y,h;1)}{h} \dee y \right)
\]
and
\als{
E	&\ll \sum_{b\le e^{(\log\log2p)^2}} \left(\frac{1}{(\log p)^{2A+1} b }
				+ \frac{1}{\sqrt{p}\log p} \int_{p^-}^{p^+} \frac{E(y,h;b)}{h} \dee y \right)    \\
	&\ll \frac{1}{(\log p)^{2A}} + \sum_{b\le e^{(\log\log2p)^2}}
				\frac{1}{\sqrt{p}\log p} \int_{p^-}^{p^+} \frac{E(y,h;b)}{h} \dee y .
}
So, relation \eqref{Koblitz-1} follows.

\medskip

Next, we pass to the proof of relations \eqref{Koblitz-2} and \eqref{Koblitz-3}. In both of these estimates, we note that, by a dyadic decomposition argument, we may restrict the range of $p$ to the interval $(x/2,x]$. Relation \eqref{Koblitz-2} is then an easy consequence of Lemma \ref{bv-short} below and of \eqref{Koblitz-1} with $h=x^{1/3}$, say. Finally, we prove \eqref{Koblitz-3}, which has been reduced to proving that
\[
\sum_{x/2<p\le x}  \prob_{\CC_p}( |E(\F_p)| \; \text{is prime} )  = C_{\text{twin}} \int_{x/2}^x \frac{\dee u}{\log^2u}
	+   O_A\left(\frac{x}{(\log x)^A}\right) .
\]
We could use \eqref{Koblitz-2} or we could work directly with a sum over two primes, $p$ and $q$. We choose the second approach. As before, our starting point is \eqref{Koblitz-gekeler}. We then set $\CG(\ell) = (\Z/\ell\Z)^*\times (\Z/\ell\Z)^*$ and
\[
\delta_{a,b}(\ell)
 =	\begin{cases}
		f_\ell(a+1-b,a)  - 1
			&\text{if}\ \ell\nmid ab,\\
		0	&\text{otherwise},
\end{cases}
\]
\[
\CA = \{(p,q):x/2<p\le x,\ |q-p-1|<2\sqrt{p}\},
\]
and
\als{
w_{p,q} 	&= f_\infty(p+1-q,p) f_p(p+1-q,p) f_q(p+1-q,p) \\
		& =  \frac{1}{\pi\sqrt{p}}\sqrt{1-\left(\frac{p+1-q}{2\sqrt{p}}\right)^2}
			 \left( 1+O\left(\frac{1}{x}\right)\right) .
}
With this notation, we find that
\[
\sum_{x/2<p\le x}\prob_{\CC_p}(|E(\F_p)|\ \text{prime})
	= \sum_{(p,q)\in \CA} w_{p,q} \prod_\ell (1+\delta_\ell(p,q)) .
\]

We are going to apply Theorem \ref{singular series average} with $d=2$, $k=1$, $F(a,b) = (a+1-b)^2-4a$, $D_1(a,b) = a^2b^2F(a,b)$, $L=1$ and $X=2x$. We need to check that the necessary conditions are satisfied. As in the previous applications, condition (1) holds by definition, and conditions (2), (3') and (4') follow by Theorem \ref{Gekeler thm} with
\[
\Delta_{\ell^r}(a,b) = -1 + \frac{\phi(\ell^{r})\ell^r\cdot \#\left\{ \sigma\in \GL_2(\Z/\ell^r\Z) :  \begin{array}{l}
	\det(\sigma)\equiv a  \mod{\ell^r},\\
	\det(\sigma)+1-\tr(\sigma) \equiv b  \mod{\ell^r}\end{array} \right\} }
	{|\GL_2(\Z/\ell^r\Z)|} ,
\]
which satisfies condition (4'c) with
\als{
\Delta_\ell	&=
		-1 +\lim_{r\to\infty} \frac{\ell^r\cdot \#\{\sigma\in \GL_2(\Z/\ell^r\Z):
			 (\det(\sigma)+1-\tr(\sigma),\ell^r)=1 \}}{\phi(\ell^r)\cdot |\GL_2(\Z/\ell^r\Z)|}  \\
		&= -1+ \frac{\ell}{\ell-1} \frac{\#\{\sigma\in \GL_2(\Z/\ell\Z):
			 \det(\sigma)+1-\tr(\sigma)\not\equiv 0\mod{\ell}\}}{|\GL_2(\Z/\ell\Z)|}  .
}
Finally, condition (5') is easy to verify too, and relation \eqref{bt for A} follows easily by the Brun-Titchmarsh inequality with $\W=x/(\log x)^2$.

In conclusion, we have shown that Theorem \ref{singular series average} is indeed applicable. This will complete the proof of Theorem \ref{Koblitz}, as long as we can control the quantities $W$, $M$ and $E$ there. The estimation of $W$ and of $E$ will be carried out in Lemma \ref{W} below. Finally, we have that $M\ll \sqrt{x}$, as in \eqref{Koblitz-M}. So relation \eqref{Koblitz-3} follows.


\subsection{Elliptic curves with a given number of points}
We demonstrate here Theorem \ref{MEG}. In view of Theorem \ref{Gekeler}, we have that
\[
\sum_p \prob_{\CC_p} ( |E(\F_p)|=N)
	= \sum_{N^-<p<N^+ } f_\infty(p+1-N,p) \prod_\ell f_\ell(N,p) .
\]
We let $\CG(\ell)=(\Z/\ell\Z)^*$, $\CA=\{p\ \text{prime}: |p-1-N|<2\sqrt{N}\}$ and
\eq{w-calc}{
w_p = f_\infty(p+1-N,p) f_p(p+1-N,p)
	&= \frac{\sqrt{4p-(p+1-N)^2}}{2\pi p}
	\left(1+O\left(\frac{1}{N}\right)\right)  \\
	&= \frac{\sqrt{4N-(N+1-p)^2} }{2\pi p}
	\left(1+O\left(\frac{1}{N}\right)\right)  \\
	&=  \frac{1}{\pi \sqrt{N}} \sqrt{1- \left(\frac{N+1-p}{2\sqrt{N}}\right)^2}
	\left(1+O\left(\frac{1}{\sqrt{N}}\right)\right)
}
for $p\in(N^-,N^+)$. Moreover, we set
\[
\delta_\ell(a)
	=\begin{cases}
		f_\ell(a+1-N,a) - 1
			&\text{if}\ \ell\nmid a,\\
		0	&\text{otherwise},
\end{cases}
\]
so that
\[
\sum_p \prob_{\CC_p}(|E(\F_p)|=N) = \sum_{p\in \CA} w_p \prod_\ell (1+\delta_\ell(p)) .
\]

We are going to apply Theorem \ref{singular series average} with $d=k=1$ and $F(x)=(x+1-N)^2-4x = (x-1-N)^2-4N$, $D_1(x)=x^2F(x)$, $L=1$ and $X=2N$. We need to check that the necessary conditions are satisfied. Condition (1) holds by definition, and conditions (2), (3') and (4') follow by Theorem \ref{Gekeler thm}, as before, with
\[
\Delta_{\ell^r}(a) = -1 + \frac{\ell^r\phi(\ell^{r})\cdot
		\#\left\{ \sigma\in \GL_2(\Z/\ell^r\Z) :
			\begin{array}{l}
				\det(\sigma)\equiv a  \mod{\ell^r},\\
				\tr(\sigma)\equiv \det(\sigma)+1-N\mod{\ell^r}
			\end{array} \right\} }
		{ |\GL_2(\Z/\ell^r\Z)|}
\]
and
\[
\Delta_\ell= -1 + \lim_{r\to\infty} \frac{\ell^r\cdot
		\#\{\sigma\in\GL_2(\Z/\ell^r\Z)   : \tr(\sigma)\equiv \det(\sigma)+1-N\mod{\ell^r}\}}
			{|\GL_2(\Z/\ell^r\Z)|} .
\]
Finally, it is easy to verify condition (5') too, and relation \eqref{bt for A} follows easily by the Brun-Titchmarsh inequality with $\W=1/ \log N$.

In conclusion, we have shown that Theorem \ref{singular series average} is indeed applicable. This will complete the proof of Theorem \ref{MEN}, as long as we can control the quantities $W$, $M$ and $E$ in Theorem \ref{singular series average}. For $M$, we have that
\als{
M &= \max_{n\le -4} \sum_{\substack{p\in\CA \\ D(p)/n\ \text{is a square}}} w_p \\
	&\ll \max_{n\le -4}\frac{\#\{(p,m) \in\N^2: (p-1-N)^2- n m^2=4N\}}{\sqrt{N}} \ll_\epsilon N^{-1/2+\epsilon} ,
}
an estimate that is good enough for our purposes. Finally, we need to estimate $W$ and $E$. This is accomplished using Lemma \ref{W2}, as in the proof of relation \eqref{Koblitz-1} in Section \ref{Koblitz-proof} above.


\subsection{Elliptic aliquot cycles}
In this subsection, we prove Theorem \ref{aliquot}. By a dyadic decomposition argument, it suffices to show that
\[
\sum_{\bs p \in \CP_d'(x)} \alpha_d(\bs p)
	= C_{\text{aliquot}}^{(d)} \int_x^{2x} \frac{\dee u}{2\sqrt{u}(\log u)^d}
		+ O_A\left( \frac{\sqrt{x}}{(\log x)^A} \right) ,
\]
where
\[
\CP_d'(x) := \{(p_1,\dots,p_d): x<p_1\le 2x,\ |p_{j+1}-p_j-1|<2\sqrt{p_j}\ (1\le j\le d)\ \text{with}\ p_{d+1}=p_1\}.
\]
Theorem \ref{Gekeler}(a) implies that
\als{
\alpha_d(\bs p) = \prod_{j=1}^d \left(f_\infty(p_j+1-p_{j+1},p_j)\prod_{\ell} f_\ell(p_j+1-p_{j+1},p_j) \right)
		= w_{\bs p} \prod_\ell(1+\delta_\ell{(\bs p}) ) ,
}
where
\als{
w_{\bs p} &:=  \prod_{j=1}^d  \prod_{\ell \in\{p_1,\dots, p_d,\infty\}}f_\ell(p_j+1-p_{j+1},p_j)  \\
	 	&= \prod_{j=1}^d \frac{1}{\pi\sqrt{p_j}} \sqrt{1- \left( \frac{p_j+1-p_{j+1}}{2\sqrt{p_j}}\right)^2 }
			\left(1+O\left(\frac{1}{x}\right) \right)
}
and $\delta_\ell(\bs p)$ is defined as
\[
\delta_\ell(\bs p):=
	\begin{cases}
		-1+ \prod_{j=1}^d f_\ell(p_j+1-p_{j+1},p_j)  	
			&\text{if}\ \ell\nmid p_1\cdots p_d, \\
		0	&\text{otherwise}.
	\end{cases}
\]

We are going to apply Theorem \ref{singular series average} with $d$, the length of the aliquot cycle here, playing the role of the dimension $d$ there. We also take $k=d$,
\[
D_j(x_1,\dots,x_d) =  x_j^2((x_j+1-x_{j+1})^2-4x_j) \quad (1\le j\le d,\ x_{d+1}=x_1),
\]
$F(\bs x) = \prod_{j=1}^k ((x_j+1-x_{j+1})^2-4x_j)$, $L=1$ and $X=3x$. We need to check that the necessary conditions are satisfied. As in the previous examples, condition (1) holds by definition, and conditions (2) and (3') follow immediately by Theorem \ref{Gekeler thm}. Next, we check condition (4'). If ${\bs a} = (a_1, \dots, a_d) \in \Z^d$, then we set
\[
\Delta_{\ell^r}(\bs a) = -1 +
			\frac{\phi(\ell^{2r})^d \cdot \# \left\{ \bs\sigma\in \GL_2(\Z/\ell^r\Z)^d :
		\begin{array}{l}
			\det(\sigma_j)\equiv a_j \mod{\ell^r} \quad\text{and} \\
			\det(\sigma_j) +1 -\tr(\sigma_j) \equiv  \det(\sigma_{j+1}) \mod{\ell^r}\\
			\mbox{for $1\le j\le d$, where $\sigma_{d+1}=\sigma_1$}
		\end{array} \right\}}{ \left| \GL_2(\Z/\ell^r \Z) \right|^d} .
\]
Theorem \ref{Gekeler thm} implies that conditions (4'a), (4'b) and (4'd) are satisfied. We also need to prove that (4'c) holds, that is to to say that the limit
\als{
\Delta_\ell
		&:=-1 +\lim_{r\to\infty} \frac{\ell^{rd}\cdot
			\# \left\{ \bs\sigma\in \GL_2(\Z/\ell^r\Z)^d :
				\begin{array}{l}
				\det(\sigma_j) +1 -\tr(\sigma_j) \equiv  \det(\sigma_{j+1}) \mod{\ell^r}\\
				\mbox{for $1\le j\le d$, where $\sigma_{d+1}=\sigma_1$}
				\end{array} \right\}}
			{\left| \GL_2(\Z/\ell^r \Z) \right|^d} .
\kommentar{
		&=-1 + \frac{\ell^d\cdot
			\# \left\{ \bs\sigma\in \GL_2(\Z/\ell\Z)^d :
				\begin{array}{l}
				\det(\sigma_j) +1 -\tr(\sigma_j) \equiv  \det(\sigma_{j+1}) \mod{\ell}\\
				\mbox{for $1\le j\le d$, where $\sigma_{d+1}=\sigma_1$}
				\end{array} \right\}}
			{\left| \GL_2(\Z/\ell \Z) \right|^d}}
}
exists. As noticed in the remark after the statement of Theorem \ref{thm:aliquot}, this step was not necessary in the proof of the other theorems as the limits were stabilizing for small values of $r$ (often $r=1$), which does not seem the case for aliquot cycles.

In view of Theorem \ref{Gekeler thm}(a), proving that the sequence
\[
T_r:= \frac{1}{\ell^{rd}} \sum_{\bs a\in (\Z/\ell^r\Z)^d}
	\prod_{m=1}^d   \left( 1+ \sum_{j=1}^{\min\{r,\nu_\ell(G_m(\bs a))+1\}}
			\frac{N_{G_m(\bs a)}(\ell^j)-N_{G_m(\bs a)}(\ell^{j-1})}{\ell^j}\right)
\]
where $G_m(\bs a):=(a_m+1-a_{m+1})^2-4a_m$, is convergent suffices to conclude that $\Delta_\ell$ is well-defined. We will show that $(T_r)_{r\ge1}$ is a Cauchy sequence. To this end, we consider two large integers $r>s$ and relate $T_r$ to $T_s$. We note that
\eq{aliquot-c}{
T_r = \frac{1}{\ell^{rd}} \sum_{\bs a\in (\Z/\ell^r\Z)^d}
	\prod_{m=1}^d   \left( 1+ \sum_{j=1}^{\min\{s,\nu_\ell(G_m(\bs a))+1\}}
			\frac{N_{G_m(\bs a)}(\ell^j)-N_{G_m(\bs a)}(\ell^{j-1})}{\ell^j}\right) +o_{s\to\infty}(1) .
}
Indeed, either $\nu_\ell(G_m(\bs a)) \le s-1\le r$ for all $m$, in which case there is nothing to prove for that summand, or $\min\{v_\ell(G_m(\bs a))+1,r\}\ge s$ for some $m$, in which case we use the elementary bound $N_D(\ell^j) \ll \ell^{j/2}$ to control the size of the tails of sum over $j$ in the definition of $T_r$. Now, note that the main term in \eqref{aliquot-c} only depends on $G_m(\bs a)\mod{\ell^s}$, that is to say on $\bs a\mod{\ell^s}$. Therefore
\als{
T_r &=  \frac{1}{\ell^{sd}} \sum_{\bs a\in (\Z/\ell^s\Z)^d}
	\prod_{m=1}^d   \left( 1+ \sum_{j=1}^{\min\{s,\nu_\ell(D_m(\bs a))+1\}}
			\frac{N_{D_m(\bs a)}(\ell^j)-N_{D_m(\bs a)}(\ell^{j-1})}{\ell^j}\right)
			+o_{s\to\infty}(1) \\
	&=T_s+o_{s\to\infty}(1),
}
which means that $(T_r)_{r\ge1}$ is a Cauchy sequence, so a convergent sequence. So condition (4') does hold.

Finally, it is easy to verify condition (5') too, and relation \eqref{bt for A} follows easily by the Brun-Titchmarsh inequality with $\W=\sqrt{x}/(\log x)^d$. In conclusion, Theorem \ref{singular series average} is indeed applicable. This completes the proof of Theorem \ref{aliquot}, as long as we can control the quantities $W$, $M$ and $E$ there. The estimation of $W$ and of $E$ will be carried out in Lemma \ref{W} below. Finally, we note that $M \le M_1+\cdots+M_d$, where
\[
M_j   = \max_{n\le -4} \sum_{\substack{\bs p\in \CP_d'(x) \\ ((p_j+1-p_{j+1})^2-4x_j)/n\ \text{is a square}}} w_{\bs p} ,
\]
and we bound each of the $M_j$'s individually. We shall demonstrate the argument for $M_1$, the details for the estimation of $M_2,\dots,M_d$ being very similar. We make the change of variables $h_i=p_i-p_{i-1}$ for $i\in\{2,\dots,d\}$. Then $(p_1+1-p_2)^2-4p_1 = (h_2-1)^2-4p_1$, so that
\als{
M_1 &\ll \max_{n\le -4 }
			\frac{\#\left\{(p_1,h_2,\dots,h_d)\in\Z^d:  \begin{array}{l}
				1\le p_1\le x,\ |h_i| \ll \sqrt{x}\ (2\le i\le d) \\
				((h_2-1)^2-4p_1)/n\ \text{is a square}
				\end{array}\right\} }
				{x^{d/2}}
		\ll 1,
}
which we obtain by noting that, for fixed $h_2,\dots,h_d$, there are always $\ll\sqrt{x}$ choices for $p_1$. Similarly, we may show that $M_j\ll 1$, $2\le j\le d$. We conclude that $M\ll1$, an estimate that is good enough for our purposes. This completes the proof of Theorem \ref{aliquot}.


\subsection{Elliptic curves with a given group structure} We demonstrate here Theorem \ref{MEG}. We recall that $G = \Z/m\Z \times \Z/mk\Z$.
In view of Theorem \ref{Gekeler for groups}, we have that
\[
\sum_p \prob_{\CC_p} (E(\F_p)\cong G)
	= \sum_{\substack{N^-<p<N^+ \\ p\equiv 1\mod{m}}} f_\infty(G,p) \prod_\ell f_\ell(G,p) .
\]
We let
\[
\CA=\left\{ \frac{p-1}{m} : p\ \text{prime},\ p\equiv 1\mod{m},\ N^-<p<N^+  \right\}
\]
and
\[
\CG(\ell) = \{a\in\Z/\ell\Z: 1+am\not\equiv 0\mod{\ell}\} .
\]
Moreover, if $a\in\CA$ and $p=1+am$ is the associated prime, then we define
\[
w_a = f_\infty(G,p) f_p(G,p)
	= \frac{1}{\pi\sqrt{N}} \sqrt{1-\left(\frac{N+1-p}{2\sqrt{N}}\right)^2}
	\left(1+O\left(\frac{1}{\sqrt{N}}\right)\right)
\]
for $p\in (N^-,N^+)$, as in \eqref{w-calc}. Finally, we set $\nu_\ell=\nu_\ell(m)$ and
\[
\delta_\ell(a) =	
	\begin{cases}
		\ds -1 + \ell^{\nu_\ell} \cdot f_\ell(G,1+am)
			&\text{if}\ \ell\nmid 1+am \\
		0	&\text{otherwise},
	\end{cases}
\]
so that
\begin{eqnarray} \label{def-prob}
\sum_p \prob_{\CC_p} (E(\F_p)\cong G)
	= \frac{1}{m} \sum_{a\in \CA} w_{\bs a} \prod_\ell (1+\delta_\ell(a)) .
\end{eqnarray}
If $\ell\nmid 1+am$, then applying Theorem \ref{Gekeler thm}(b) with $t=p+1-N=am-m^2k+2$, $u=1+am$ and $n=m$, so that $t_1=a-mk$ and $u_1=k$, implies that
\eq{groups formula rewrite}{
1+\delta_\ell(a)
	= \lim_{r\to\infty} \frac{\ell^r\phi(\ell^r) \cdot\#\left\{\sigma\in M_2(\Z/\ell^r\Z) :
		\begin{array}{l}
			\tr(\sigma) \equiv a-mk \pmod{\ell^r},\\
			\det(\sigma)\equiv k \pmod{\ell^r} \\
			\sigma\not\equiv 0\mod{\ell}
		\end{array}\right\}}
	{|\GL_2(\Z/\ell^r\Z)|} .
}
Moreover, the same result implies that the above limit stabilizes for $r>\nu_\ell((a+mk)^2-4k)$.

We are going to apply Theorem \ref{singular series average} with the parameters $d$ and $k$ there both equal to 1, $F(a) = (a-mk)^2-4k$, $D_1(a)= (1+am)^2F(a)$, $L=1$ and $X=2\sqrt{k}$. (We allow ourself a slight double notation for one line here, as the $k$ in the definition of $D_1,F$ and $X$ is the $k$ from the statement of Theorem \ref{MEG}, not the parameter in Theorem \ref{singular series average}.)
 We need to check that the necessary conditions are satisfied. As before, condition (1) holds by definition, and conditions (2) and (3') follow immediately by Theorem \ref{Gekeler thm}. Next, in view of relation \eqref{groups formula rewrite}, we take
\[
\Delta_{\ell^r}(a) =  -1
	+ \frac{\phi(\ell^r)\ell^r \cdot\#\left\{\sigma\in M_2(\Z/\ell^r\Z) :
		\begin{array}{l}
			\tr(\sigma) \equiv a-mk \pmod{\ell^r},\\
			\det(\sigma)\equiv k \pmod{\ell^r} \\
			\sigma\not\equiv 0\mod{\ell}
		\end{array}\right\}}
	{|\GL_2(\Z/\ell^r\Z)|} .
\]
Note that $\det(I+m\sigma)\equiv 1+m\tr(\sigma)+m^2\det(\sigma)\equiv 1+am\mod{\ell}$, so we naturally take
\als{
\Delta_\ell &= -1 + \lim_{r\to\infty} \frac{\phi(\ell^r)\ell^r \cdot\#\left\{
		\begin{array}{l}
			\sigma\in M_2(\Z/\ell^r\Z) \\
			\ell\nmid \det(I+m\sigma)
		\end{array}
			:
		\begin{array}{l}
			\det(\sigma)\equiv k \pmod{\ell^r} \\
			\sigma\not\equiv 0\mod{\ell}
		\end{array}\right\}}
	{|\CG(\ell^r)|\cdot |\GL_2(\Z/\ell^r\Z)|} \\
	&= -1 + \frac{\phi(\ell^{\nu_\ell})}{\ell^{\nu_\ell}}
		\lim_{r\to\infty} \frac{\ell^r \cdot\#\left\{
		\begin{array}{l}
			\sigma\in M_2(\Z/\ell^r\Z) \\
			\ell\nmid \det(I+m\sigma)
		\end{array}
			:
		\begin{array}{l}
			\det(\sigma)\equiv k \pmod{\ell^r} \\
			\sigma\not\equiv 0\mod{\ell}
		\end{array}\right\}}
	{|\GL_2(\Z/\ell^r\Z)|} \\
	& =-1 +\frac{\phi(\ell^{\nu_\ell})}{\ell^{5\nu_\ell}} \lim_{r\to\infty} \frac{\ell^r \cdot\#\left\{
		\begin{array}{l}
			\sigma\in M_2(\Z/\ell^{r+\nu_\ell}\Z) \\
			\ell\nmid \det(I+m\sigma)
		\end{array}
			:
		\begin{array}{l}
			\det(\sigma)\equiv k \pmod{\ell^r} \\
			\sigma\not\equiv 0\mod{\ell}
		\end{array}\right\}}
	{|\GL_2(\Z/\ell^r\Z)|} ,
}
where we used the formula $|\CG(\ell^r)|=\phi(\ell^r)\ell^{\nu_\ell}/\phi(\ell^{\nu_\ell})$ to get from the first line to the second line. Making the change of variable $g=I+m\sigma$, which determines $g$ mod $\ell^{r+2\nu_\ell}$, we find that
\als{
\Delta_\ell &= -1 + \frac{\phi(\ell^{\nu_\ell})}{\ell^{5\nu_\ell}} \lim_{r\to\infty}
			\frac{\ell^r \cdot\#\left\{
			g \in \GL_2(\Z/\ell^{r+2\nu_\ell}\Z)
			:
			\begin{array}{l}
			\det(g)+1-\tr(g)\equiv N\mod{\ell^{r+2\nu_\ell}} \\
			g\equiv I\mod{\ell^{\nu_\ell}} \\
			g\not\equiv I\mod{\ell^{\nu_\ell+1}}
		\end{array}\right\}}
	{ |\GL_2(\Z/\ell^r\Z)| } \\
	&= -1 + \phi(\ell^{2\nu_\ell})   \lim_{r\to\infty} \frac{\ell^r \cdot\#\left\{
			g \in \GL_2(\Z/\ell^r \Z)
			:
			\begin{array}{l}
			\det(g)+1-\tr(g)\equiv N\mod{\ell^r} \\
			g\equiv I\mod{\ell^{\nu_\ell}} \\
			g\not\equiv I\mod{\ell^{\nu_\ell+1}}
		\end{array}\right\}}
	{|\GL_2(\Z/\ell^r\Z)|} .
}
Finally, it is easy to verify condition (5') too, and relation \eqref{bt for A} follows easily by the Brun-Titchmarsh inequality with $\W=1/(\phi(m)\log k)$.

In conclusion, we have show that Theorem \ref{singular series average} is indeed applicable. Theorem \ref{MEG} then follows as long as we can control the quantities $W$, $M$ and $E$ in the statement of Theorem \ref{singular series average}. For $M$, we have that
\als{
M &\le \max_{n\le -4} \sum_{\substack{a\in\CA \\ D_1(a)/n\ \text{is a square}}} w_a \\
	&\ll \max_{n\le -4}\frac{\#\{(a,b) \in\N^2: (a-mk)^2- n b^2=4k\}}{\sqrt{N}}
		\ll_\epsilon \frac{k^\epsilon}{\sqrt{N}},
}
an estimate that is good enough for our purposes. an estimate that is good enough for our purposes. Finally, the quantities $W$ and $E$ are estimated using Lemma \ref{W2}, as in the proof of relation \eqref{Koblitz-1} in Section \ref{Koblitz-proof} above.
For $h \in [mk^\epsilon, \sqrt{N}/(\log{k})^{2A+1}]$, we have
\als{
W = \sum_{\bs a \in \CA} w_{\bs a}
	&= \sum_{\substack{ N^- < p < N^+ \\ p \equiv 1 \mod m }}
		\frac{1}{\pi\sqrt{N}} \sqrt{1-\left(\frac{N+1-p}{2\sqrt{N}}\right)^2}
	\left(1+O\left(\frac{1}{\sqrt{N}}\right)\right) \\
	&= \frac{1}{\phi(m)\log N}
		+ O\left( \frac{1}{m (\log N)^{2A+1}}
			+ \frac{1}{\sqrt{N}\log N} \int_{N^-}^{N^+}
			\frac{E(y,h;m)}{h} \dee y \right) .
}
Similarly, if $(1+bm, q) =1$, then
\als{
\sum_{\substack{\bs a \in \CA \\ \bs a \equiv b \mod q  }} w_{\bs a} 	
	&= \sum_{\substack{ N^- < p < N^+ \\  p \equiv 1+ bm \mod{qm} }}
		\frac{1}{\pi\sqrt{N}} \sqrt{1-\left(\frac{N+1-p}{2\sqrt{N}}\right)^2}
	\left(1+O\left(\frac{1}{\sqrt{N}}\right)\right) \\
	&= \frac{1}{\phi(qm)\log N}
		+ O\left( \frac{1}{qm (\log N)^{2A+1}} \right) \\
	&\quad	
			+O\left( \frac{1}{h\sqrt{N}\log N} \int_{N^-}^{N^+}
				 E(y, h; qm)  \dee y \right) .
}
Since $\phi(qm)=|\CG(q)|\phi(m)$, we deduce that
\als{
E \ll \frac{1}{m(\log k)^{2A}} + \frac{R(N,h;m)}{\phi(m)\log N}
}
This completes the proof of Theorem \ref{MEG}.


\subsection{Elliptic curves with a cyclic group of points} We demonstrate here Theorem \ref{cyclicity}. In view of Theorem \ref{Gekeler for groups}, we have that
\als{
\prob_{\CC_p} ( E(\F_p)\ \text{is cyclic})
	&= \sum_{p^-<k<p^+} \prob_{\CC_p} ( E(\F_p)\cong \Z/k\Z)  \\
	&= \sum_{p^-<k<p^+ } f_\infty(p+1-k,p) \prod_\ell f_\ell(\Z/k\Z,p) .
}
We let $\CG(\ell)=\Z/\ell\Z$, $\CA=(p^-,p^+)\cap\Z$, $w_k= f_\infty(p+1-k,p)f_p(p+1-k,p)$ and $\delta_\ell(k)={\bf 1}_{\ell\neq p}\cdot (f_\ell(p+1-\Z/k\Z,p) - 1)$,
so that
\[
\prob_{\CC_p} ( E(\F_p)\ \text{is cyclic})
	= \sum_{k\in \CA} w_k \prod_\ell (1+\delta_\ell(k)) .
\]

We are going to apply Theorem \ref{singular series average} with the parameters $d$ and $k$ both equal to 1, and with $D_1(x)=F(x)=(p+1-x)^2-4p$, $L=p$ and $X=2p$. We need to check that the necessary conditions are satisfied. Condition (1) holds by definition, and conditions (2), (3') and (4') follow by Theorem \ref{Gekeler thm} with
\[
\Delta_{\ell^r}(a) = -1 + \frac{\phi(\ell^{r})\ell^r \cdot
		\#\left\{ \sigma\in \GL_2(\Z/\ell^r\Z) :
			\begin{array}{l}
				\tr(\sigma)\equiv p+1-k\mod{\ell^r}, \\
				\det(\sigma)\equiv p \mod{\ell^r},\\
				\sigma\not\equiv I \mod{\ell}
			\end{array} \right\} }
		{ |\GL_2(\Z/\ell^r\Z)|}
\]
when $\ell\nmid p$ and with $\Delta_{p^r}(a)=0$, which satisfies condition (4'c) with
\als{
\Delta_\ell &= -1 + \lim_{r\to\infty} \frac{\phi(\ell^r)\cdot
			\#\left\{\sigma\in\GL_2(\Z/\ell^r\Z)   :
				\begin{array}{l}	
				\det(\sigma)\equiv p \mod{\ell^r},\\
				\sigma\not\equiv I\mod{\ell}
				\end{array}
				\right\}}
			{|\GL_2(\Z/\ell^r\Z)|} \\
		&= -  \lim_{r\to\infty} \frac{\phi(\ell^r)\cdot
			\#\left\{\sigma\in\GL_2(\Z/\ell^r\Z)   :
				\begin{array}{l}	
				\det(\sigma)\equiv p \mod{\ell^r},\\
				\sigma\equiv I\mod{\ell}
				\end{array}
				\right\}}
			{|\GL_2(\Z/\ell^r\Z)|}
}
when $\ell\neq p$ and with $\Delta_p=0$. Finally, condition (5') is easy to verify too, and relation \eqref{bt for A} holds with $\W=1$.

Before we proceed further, note that
\[
\Delta_\ell = - \frac{ {\bf 1}_{\ell|(p-1)} }{\ell(\ell^2-1)} .
\]
It's easy to see that this is true when $\ell\nmid p-1$ (the condition $\sigma\equiv I\mod{\ell}$ would imply $\det(\sigma)\equiv1\mod \ell$). Finally, if $\ell|p-1$, then note that $\#\{\tau\in M_2(\Z/\ell^{r-1}\Z): \det(I+\ell \tau)\equiv p\mod{\ell^r}\} = \ell^{3(r-1)}$. Therefore
\[
\Delta_\ell = -\phi(\ell^r)  \frac{\ell^{3(r-1)}}{\ell^{4(r-1)} (\ell^2-\ell)(\ell^2-1) }
	= - \frac{\ell-1}{(\ell^2-\ell)(\ell^2-1)} = -\frac{1}{\ell(\ell^2-1)},
\]
as claimed.

In conclusion, we have shown that Theorem \ref{singular series average} is indeed applicable. This will complete the proof of Theorem \ref{cyclicity}, as long as we can control the quantities $W$, $M$ and $E$ in Theorem \ref{singular series average}. For $M$, we have that
\als{
M &= \max_{n\le -4} \sum_{\substack{p^-<k<p^+ \\ D(k)/n\ \text{is a square}}} w_k \\
	&\ll \max_{n\le -4}\frac{\#\{(k,m) \in\N^2: (p+1-k)^2- n m^2=4p\}}{\sqrt{p}} \ll p^{-1/2} ,
}
an estimate that is good enough for our purposes. Finally, we need to estimate $W$ and $E$. By partial summation, we have that
\als{
\sum_{\substack{p^-<k<p \\ k\equiv a\mod q}} w_k
	&=  \frac{1}{\pi \sqrt{p}} \int_{p^-}^{p^+} \sqrt{1-\left(\frac{p+1-t}{2\sqrt p}\right)^2}
		\dee \left(\frac{t}{q} + O(1) \right)
	=\frac{1}{q} +O\left(\frac{1}{\sqrt{p}}\right) .
}
So Theorem \ref{cyclicity} follows.


\subsection{Restricting the trace to an arithmetic progression} We demonstrate here Theorem \ref{trace in APs}. Again, Theorem \ref{Gekeler} implies that
\als{
\prob_{\CC_p} ( a_p(E)\equiv t\mod{N} )
	&= \sum_{\substack{p^-<s<p^+\\ s\equiv t\mod{N}}} \prob_{\CC_p} ( a_p(E) = s) \\
	&= \sum_{\substack{p^-<s<p^+\\ s\equiv t\mod{N}}}  f_\infty(s,p) \prod_\ell f_\ell(s,p) .
}
We let $\CG(\ell)=\Z/\ell\Z$, $\CA=\{a\in\Z: p^-<t+aN<p^+\}$, $w_a= f_\infty(t+Na,p) f_p(t+Na,p)$ and $\delta_\ell(a)= {\bf 1}_{\ell\neq p} \cdot (f_\ell(t+Na,p) - 1)$,
so that
\[
\prob_{\CC_p} ( a_p(E)\equiv t\mod N)
	= \sum_{a\in \CA} w_a \prod_\ell (1+\delta_\ell(a)) .
\]

We let $g=(t^2-4p,N)$. We are going to apply Theorem \ref{singular series average} with $d=k=1$, $D_1(x)=(t+Na)^2-4p$, $F(x) = D_1(x)/g$, $L=p$ and $X=2p$. We need to check that the necessary conditions are satisfied. Condition (1) holds by definition, and conditions (2), (3') and (4') follow by Theorem \ref{Gekeler thm} with
\[
\Delta_{\ell^r}(a) = -1 + \frac{\ell^{r+\nu_\ell(g)}\cdot
		\#\left\{ \sigma\in \GL_2(\Z/\ell^{r+\nu_\ell(g)}\Z) :
			\begin{array}{l}
				\tr(\sigma)\equiv t+Na\mod{\ell^{r+\nu_\ell(g)}}, \\
				\det(\sigma)\equiv p \mod{\ell^{r+\nu_\ell(g)}}
			\end{array} \right\} }
		{  \#\{\sigma\in\GL_2(\Z/\ell^{r+\nu_\ell(g)}\Z):\det(\sigma)\equiv p\mod{\ell^{r+\nu_\ell(g)}}\} }
\]
when $\ell\neq p$ and with $\Delta_{p^r}(a)=0$, which satisfies condition (4'c) with
\als{
\Delta_\ell= -1 + \lim_{r\to\infty}  \frac{\ell^{\nu_\ell(N)}\cdot
		\#\left\{ \sigma\in \GL_2(\Z/\ell^{r+\nu_\ell(g)}\Z) :
			\begin{array}{l}
				\tr(\sigma)\equiv t\mod{\ell^{\nu_\ell(N)}} , \\
				\det(\sigma)\equiv p \mod{\ell^{r+\nu_\ell(g)}}
			\end{array} \right\} }
		{\#\{\sigma\in\GL_2(\Z/\ell^{r+\nu_\ell(g)}\Z):\det(\sigma)\equiv p\mod{\ell^{r+\nu_\ell(g)}}\}}
}
when $\ell\neq p$ and with $\Delta_p=0$. Finally, for condition (5'), we note that $\SC(F)\ll1$ since $N<p$, and in relation \eqref{bt for A} we take $\W=1/N$.

Before we proceed further, we simplify $\Delta_\ell$. If $\ell\nmid N$, then it is easy to see that $\Delta_\ell=0$. Assume, now, that $\ell|N$. In particular, $\ell\neq p$. Making the change of variable $r+\nu_\ell(g)=s+\nu_\ell(N)$, we find that
\als{
\Delta_\ell= -1 + \lim_{s\to\infty}  \frac{\ell^{\nu_\ell(N))} \cdot
		\#\left\{ \sigma\in \GL_2(\Z/\ell^{s+\nu_\ell(N)}\Z) :
			\begin{array}{l}
				\tr(\sigma)\equiv t\mod{\ell^{\nu_\ell(N)}} , \\
				\det(\sigma)\equiv p \mod{\ell^{s+\nu_\ell(N)}}
			\end{array} \right\} }
		{\#\{\sigma\in\GL_2(\Z/\ell^{s+\nu_\ell(N)}\Z):\det(\sigma)\equiv p\mod{\ell^{s+\nu_\ell(N)}}\}} .
}
For each $\sigma_0\in \GL_2(\Z/\ell^{\nu_\ell(N)}\Z)$ with $\tr(\sigma_0)\equiv t\mod{\ell^{\nu_\ell(N)}}$ and $\det(\sigma_0)\equiv p\mod{\ell^{\nu_\ell(N)}}$, it is easy to see that there are precisely $\ell^{3s}$ matrices $\sigma\in \GL_2(\Z/\ell^{s+\nu_\ell(N)}\Z)$ with $\sigma\equiv \sigma_0\mod{\ell^{\nu_\ell(N)}}$ and $\det(\sigma)\equiv p \mod{\ell^{s+\nu_\ell(N)}}$. Therefore,
\als{
\Delta_\ell= -1 + \frac{\ell^{\nu_\ell(N))} \cdot
		\#\left\{ \sigma\in \GL_2(\Z/\ell^{\nu_\ell(N)}\Z) :
			\begin{array}{l}
				\tr(\sigma)\equiv t\mod{\ell^{\nu_\ell(N)}} , \\
				\det(\sigma)\equiv p \mod{\ell^{\nu_\ell(N)}}
			\end{array} \right\} }
		{\#\{\sigma\in\GL_2(\Z/\ell^{\nu_\ell(N)}\Z):\det(\sigma)\equiv p\mod{\ell^{\nu_\ell(N)}}\}} ,
}
and hence the Chinese Remainder Theorem implies that
\[
\prod_{\ell}(1+\Delta_\ell)
	= \frac{N\cdot
		\#\left\{ \sigma\in \GL_2(\Z/N\Z) :
			\begin{array}{l}
				\tr(\sigma)\equiv t\mod N, \\
				\det(\sigma)\equiv p \mod N
			\end{array} \right\} }
		{\#\{\sigma\in\GL_2(\Z/N\Z):\det(\sigma)\equiv p\mod N\}} .
\]

The above discussion will complete the proof of Theorem \ref{trace in APs} as long as we can control the quantities $W$, $M$ and $E$ in Theorem \ref{singular series average}. For $M$, we have that
\als{
M &= \max_{n\le -4} \sum_{\substack{a\in \CA \\ D_1(a)/n\ \text{is a square}}} w_a \\
	&\ll \max_{n\le -4}\frac{\#\{(a,m) \in\N^2: (t+Na)^2-ngm^2=4p\}}{\sqrt{p}} \ll \frac{1}{\sqrt{p}} ,
}
an estimate that is good enough for our purposes. Finally, we need to estimate $W$ and $E$. As in the proof of Theorem \ref{cyclicity}, partial summation implies that
\als{
\sum_{\substack{p^-<t+aN<p \\ a\equiv b\mod q}} w_a
	=\frac{1}{Nq} + O\left(\frac{1}{\sqrt{p}}\right) .
}
So Theorem \ref{trace in APs} follows by taking the parameter $\epsilon$ in the statement of Theorem \ref{singular series average} to be small enough.


\section{Proof of Theorems \ref{singular series average general} and \ref{singular series average}}\label{singular series proof}

In this section, we prove Theorems \ref{singular series average general} and \ref{singular series average}.
We start with the former. Before embarking on its proof, we state an auxiliary result, which is an application of zero-density estimates of $L$-functions, first observed by Elliott (see, also, \cite[Proposition 2.2]{GS}). For a proof of it in the stated form, see \cite[Lemma 2.3]{CDKS2}.


\begin{lma} \label{lemmashortproduct}
Let $\alpha \geq 1$ and $y\ge3$. There is a set $\mathcal{E}_\alpha(y)\subset[1,y]\cap\Z$ of at most $y^{2/\alpha}$ integers such that if $\chi$ is a Dirichlet character modulo $d\le\exp\{(\log y)^2\}$ whose conductor does not belong to $\mathcal{E}_\alpha(y)$, then
\[
L(1,\chi) = \prod_{\ell\le (\log y)^{8\alpha^2}} \left(1-\frac{\chi(\ell)}{\ell}\right)^{-1} \left(1 + O_\alpha\left(\frac1{(\log y)^\alpha}\right)\right).
\]
\end{lma}


\begin{proof}[Proof of Theorem \ref{singular series average general}] All implied constants might depend on the various parameters mentioned in the end of the statement of Theorem \ref{singular series average general}. We may assume that $X$ is large enough. Firstly, we use Lemma \ref{lemmashortproduct} to truncate $P_{\bs a}$ and replace it by
\[
P_{\bs a}^{(z)} := \prod_{\ell\le z} (1+\delta_\ell(\bs a)) .
\]
We apply this result with $\alpha=A^2$ and $y=X^A$, where $A$ is a constant to be chosen later. We assume that $A$ is large enough, so that $y\ge M_{j,\bs a}$ for all $\bs a\in\CA$ and all $j\in\{1,\dots,d\}$. Let $z=(\log y)^{8\alpha^2}=(A\log X)^{8A^4}\le e^{(\log\log X)^2}=Q$ for $X$ large enough. If the conductor of $\chi_{j,\bs a}$ does not belong to the exceptional set $\CE_{A^2}(X^A)$ for all $j\in\{1,\dots,d\}$, then conditions (2) and (3) imply that
\[
\sum_{\ell>z} \delta_\ell(\bs a)
 = \sum_{j=1}^k \lambda_{j,\bs a}\sum_{\ell>z} \frac{\chi_{j,\bs a}(\ell)}{\ell}
 	+ O\left(\sum_{\ell>z}\frac{1}{\ell^{1+\eta}} + \sum_{\substack{\ell| L_{\bs a} \\ \ell>z }} \frac{1}{\ell}\right)
 \ll \frac{1}{(\log X)^{C+1}},
\]
provided that $A$ is large enough, since $\omega(L_{\bs a})\le (\log X)^{O(1)}$ by condition (5).
Moreover, condition (2) implies that $P_{\bs a}^{(z)}\ll (\log z)^{O(1)}\ll (\log\log X)^{O(1)}$. Therefore
\[
P_{\bs a} = P_{\bs a}^{(z)} + O\left(\frac{1}{(\log X)^C}\right)
\]
for $\bs a\in \CA$ with $\cond(\chi_{j,\bs a})\notin \mathcal{E}_{A^2}(X^A)$ for all $j\in\{1,\dots,k\}$. Finally, when this last condition fails for some $j\in\{1,\dots,k\}$, then we recall that $\delta_\ell(\bs a)\ll 1/\ell$. So, if $\delta$ is small enough in terms of $\epsilon$, then
\[
P_{\bs a}
	\ll  X^{\epsilon/2} \prod_{\ell>\exp\{X^\delta\}} |1+\delta_\ell(\bs a)|
\]
by condition (2). So, using conditions (3) and (5), and and the Prime Number Theorem for arithmetic progressions, which implies that $\sum_{\ell> \exp\{q^\alpha\}} \chi(\ell)/\ell \ll_\alpha 1$ for fixed $\alpha>0$ and a non-principal character mod $q$, we find that
\als{
P_{\bs a}
	&\ll X^{\epsilon/2}
		\exp\left\{  \Re\left( \sum_{j=1}^k \lambda_{j,\bs a} \sum_{\ell>\exp\{X^\delta\} }
			\frac{\chi_{j,\bs a}(\ell)}{\ell}  \right)
			+ O\left(\sum_{\ell>\exp\{X^\delta\}} \frac{1}{\ell^{1+\eta}}
				+ \sum_{\substack{\ell| L_{\bs a} \\ \ell>\exp\{X^\delta\} }} \frac{1}{\ell}\right)
					\right\} \\
	&\ll X^{\epsilon/2} .
}
Hence,
\[
\sum_{a\in \CA} w_{\bs a}P_{\bs a}
	= \sum_{a\in \CA} w_{\bs a}P_{\bs a}^{(z)}
	+ O\left(   \frac{W}{(\log X)^B}
		+ X^{\epsilon/2} \sum_{j=1}^k
		\sum_{\substack{\bs a\in \CA  \\ \cond(\chi_{j,\bs a}) \in \mathcal{E}_{A^2}(X^A)}}
			|w_{\bs a}| \right) .
\]
Since
\[
\sum_{\substack{\bs a\in \CA \\ \cond(\chi_{j,\bs a}) \in \mathcal{E}_{A^2}(X^A) }} |w_{\bs a}|
	\le M \cdot |\CE_{A^2}(X^A)| \le M X^{2/A} ,
\]
where $M$ is defined in the statement of Theorem \ref{singular series average general}, we find that
\eq{mainthm e1}{
\sum_{\bs a\in \CA} w_{\bs a}P_{\bs a}
	= \sum_{\bs a\in \CA} w_{\bs a}P_{\bs a}^{(z)}
		+ O\left(    \frac{W}{(\log X)^C} + MX^\epsilon    \right)
}
by taking $A\ge 4/\epsilon$.

Next, we turn to the estimation of the main term. We define multiplicatively, for $q\in\N$, $\bs a\in \CA$ and $\bs h \mod q$,
\[
\delta_q(\bs a)
	= \prod_{\ell|q} \delta_\ell(\bs a)
\quad\text{and}\quad
\Delta_{q}(\bs h) = \prod_{\ell^r\| q} \Delta_{\ell^r}(\bs h) .
\]
It then follows from conditions (2) and (4b) that
\begin{align} \label{bound-delta-n}
|\delta_q(\bs a)|
	&\le \frac{c_1^{\omega(q)}}{\rad(q)} \quad\text{and}\\
\label{bound-Delta-q}
|\Delta_{q}(\bs h)| &\le \frac{c_1^{\omega(q)}}{\rad(q)}
	\quad \mbox{if $\bs h \in \CG(\ell^r) \setminus \CE(\ell^r)$ whenever $\ell^r\| q$,}
\end{align}
where $c_1$ is some absolute constant. With this notation, we have that
\[
P_{\bs a}^{(z)} = \sum_{P^+(n)\le z} \mu^2(n)\delta_n({\bs a})   .
\]

Moreover, for each $\bs a\in \CA$, we set
\[
\nu_{\ell,\bs a}=  \min\{ r\ge1 :  \bs a \mod{\ell^r} \notin \CE(\ell^r) \}
\quad\text{and}\quad
q_{n,\bs a} = \prod_{\ell|n} \ell^{\nu_{\ell,\bs a}} ,
\]
which are well-defined in view of condition (4b). Additionally, we define
\[
\CH(\ell^r) = \{\bs h\in \CG(\ell^r) : \bs h \notin \CE(\ell^r), \bs h \mod{\ell^{r-1}} \in \CE(\ell^{r-1})
\}  .
\]
Then, for each $\bs a\in \CA$ with $\bs a\mod \ell\in \CG(\ell)$, we have that
\[
\nu_{\ell,\bs a} = r \iff \bs a \mod{\ell^r} \in \CH(\ell^r).
\]
What is more, if $r=\nu_{\ell,\bs a}$, then $\delta_\ell(\bs a) = \Delta_{\ell^r}(\bs h)$ for all $\bs h \equiv \bs a \mod{\ell^r}$.
For the convenience of notation, given $q\in\N$, we define
\[
\CH(q) = \{\bs h \in (\Z/q\Z)^d : \bs h \mod{\ell^r}\in \CH(\ell^r) \ \text{whenever}\ \ell^r\|q\}
\]
and
\[
\CE(q) = \{\bs h \in (\Z/q\Z)^d : \bs h \mod{\ell^r}\in \CE(\ell^r) \ \text{whenever}\ \ell^r\|q\} .
\]
If $n=\rad(q)$ and $\bs a\mod q \in \CG(q)$, then
\begin{align*}
\bs a \mod q \in \CH(q)
&\iff \ell^{\nu_{\ell,\bs a}}\| q \quad\text{for all primes}\ \ell |n \\
&\iff q_{n, \bs a} = \prod_{\ell|n} \ell^{\nu_{\ell,\bs a}} = q,
\end{align*}
and $\delta_n(\bs a) = \Delta_q(\bs h)$ in that case, where $\bs a \mod q = \bs h$.
If $P^+(n)$ denotes the large prime divisor of $n$ with the convention that $P^+(1)=1$, then
\als{
\sum_{a\in \CA} w_{\bs a} P_{\bs a}^{(z)}
		= \sum_{P^+(n)\le z } \mu^2(n) \sum_{a\in \CA} w_{\bs a}\delta_n({\bs a})
		&= \sum_{P^+(n)\le z } \mu^2(n) \sum_{\substack{q\in \N \\ \rad(q)=n}}
		\sum_{\substack{ a\in \CA \\ q_{n,\bs a}=q }} w_{\bs a}\delta_n(\bs a)
		=S_1+S_2
}
say, where $S_1$ is the part of the sum with $q\le Q$ and $S_2$ is the rest of the sum.

Before estimating $S_1$ and $S_2$, we set
\[
f(b) = |\CE(b)|/b^{d-1},
\]
which is a multiplicative function, and note that, for any fixed $\kappa>0$ and $c>0$, we have
\eq{ss proof e1}{
\sum_{\substack{P^+(bn)\le z \\ \rad(b)|n}}
		\frac{\mu^2(n)c^{\omega(n)} f(b)^\lambda}{n^{1-\kappa/\log z}b}
		\ll (\log\log X)^{O(1)} e^{O(S)},
}
where the implied constants depend on $\kappa$ and $c$, and $S$ is defined as in the statement of Theorem \ref{singular series average general}. Indeed, we note that  if $n$ is square-free and $z$-smooth, then $n^{1/\log z} \le e^{\omega(n)}$. Therefore, writing $n=\rad(b)a$, we deduce that
\als{
\sum_{\substack{P^+(bn)\le z \\ \rad(b)|n}}
		\frac{\mu^2(n)c^{\omega(n)} f(b)^\lambda}{n^{1-\kappa/\log z}b}
	&\le  \sum_{P^+(b)\le z} \frac{(ce^\kappa)^{\omega(b)} f(b)^\lambda}{\rad(b)b}
		\sum_{P^+(a)\le z }
		\frac{\mu^2(a)(ce^\kappa)^{\omega(a)}}{a} \\
	&\ll(\log\log X)^{O(1)}
		\exp\left\{ ce^\kappa \sum_{\ell\le z} \sum_{r=1}^\infty
			\frac{f(\ell^r)^\lambda}{\ell^{1+r}} \right\}
	= (\log\log X)^{O(1)} e^{O(S)} ,
}
which proves \eqref{ss proof e1}.

Let us see now how to estimate $S_1$ and $S_2$. We start with the latter. We need to use \eqref{bt for A}, but the modulus $q$ might be too large. Using \eqref{bound-delta-n} and condition (1), we find that
\[
S_2 = \sum_{P^+(n)\le z} \mu^2(n) 	
		\sum_{\substack{ \bs a\in \CA \\ q_{n,\bs a}>Q}} w_{\bs a} \delta_n(\bs a)
	\ll \sum_{P^+(n)\le z} \frac{\mu^2(n)c_1^{\omega(n)}}{n} 	
		\sum_{\substack{ \bs a\in \CA \\ \bs a\mod n\in \CG(n) \\ q_{n,\bs a}>Q}} |w_{\bs a}| .
\]
Write $q_{n,\bs a} = nq'$, $q'\in\N$, and note that if $\ell^r\| q'$, then $\bs a\mod{\ell^r}\in \CE(\ell^r)$. So, for each $s\in\{0,1,\dots,r\}$, we have that $\bs a\mod{\ell^s}\in \CE(\ell^s)$, with the convention that $\CE(1)=\{1\}$. (Here, we used condition (4b).) Clearly, if $q_{n,\bs a}>Q$, then either $n>Q^{1/2}$ or $q'>Q^{1/2}$. In the latter case, the $z$-smoothness of $q'$ implies that $q'$ has a divisor $b\in(Q^{1/2},Q^{1/2}z]\subset(Q^{1/2},Q]$. Moreover, we have that $\rad(b)|n$ and $\bs a\mod{b} \in \CG(b)\cap \CE(b)$.  Therefore
\als{
S_2 &\ll \sum_{\substack{P^+(n)\le z \\ n>Q^{1/2} }} \frac{\mu^2(n)c_1^{\omega(n)}}{n} \cdot \W
		+ 	\sum_{\substack{P^+(b)\le z \\ Q^{1/2}<b\le Q}}
				\sum_{\substack{ P^+(n)\le z  \\  \rad(b)| n }} \frac{\mu^2(n)c_1^{\omega(n)}}{n}
				\sum_{\substack{\bs a\in \CA \\  \bs a\mod{b}\in \CG(b)\cap\CE(b) }} |w_{\bs a}| \\
	&\ll \W\sum_{\substack{P^+(n)\le z \\ n>Q^{1/2} }} \frac{\mu^2(n)c_1^{\omega(n)}}{n}
		+ (\log\log X)^{c_1}\sum_{\substack{P^+(b)\le z \\ Q^{1/2}<b\le Q}}
			\frac{c_1^{\omega(b)}}{\rad(b)}
				\sum_{\bs h\in \CG(b)\cap \CE(b)} \sum_{\substack{\bs a\in \CA
					\\  \bs a\equiv \bs h\mod b}} |w_{\bs a}| .
}
We may now apply \eqref{bt for A} to deduce that
\eq{S2 e1}{
S_2 \ll \W\sum_{\substack{P^+(n)\le z \\ n>Q^{1/2} }} \frac{\mu^2(n)c_1^{\omega(n)}}{n}
		+ \W (\log\log X)^{c_1}\sum_{\substack{P^+(b)\le z \\ b>Q^{1/2}}}
				\frac{|\CE(b)|}{\rad(b)|\CG(b)|} .
}
Note that $|\CG(b)|\ge c_2^{-\omega(b)}b^d$ for some absolute constant $c_2\ge1$, a consequence of relation \eqref{G}. Therefore H\"older's inequality and \eqref{ss proof e1} imply that
\als{
\sum_{\substack{P^+(b)\le z \\ b>Q^{1/2}}}
				\frac{|\CE(b)|}{\rad(b)|\CG(b)|}
		\le \sum_{\substack{P^+(b)\le z \\ b>Q^{1/2} }}
				\frac{c_2^{\omega(b)} f(b)}{b\rad(b)}
		&\le \left(  \sum_{\substack{P^+(b)\le z \\ b>Q^{1/2} }}
				\frac{c_2^{\frac{\lambda}{\lambda-1}\omega(b)}}{b} \right)^{1-\frac{1}{\lambda}}
			\left( \sum_{P^+(b)\le z}
				\frac{f(b)^\lambda}{\rad(b)^\lambda b}\right)^{\frac{1}{\lambda}} \\
		&\ll (\log\log X)^{O(1)} e^{O(S)} \cdot \left(  \sum_{\substack{P^+(b)\le z \\ b>Q^{1/2} }}
				\frac{c_2^{\frac{\lambda}{\lambda-1}\omega(b)}}{b} \right)^{1-\frac{1}{\lambda}} .
}
For any $c\ge1$ and any $\kappa>0$, we have that
\[
\sum_{\substack{P^+(b)\le z \\ b>Q^{1/2} }} \frac{c^{\omega(b)}}{b}
	\le \frac{1}{Q^{\kappa/(2\log z)}} \sum_{P^+(b)\le z} \frac{c^{\omega(b)}}{b^{1-\kappa/\log z}}
	\asymp \frac{1}{(\log X)^{\kappa/(16A^4)}} \prod_{p\le z}\left(1-\frac{c}{p^{1-1/\log z}} \right)^{-1}.
\]
Since $p^{1/\log z}=1+O(\log p/\log z)$ for $p\le z$, we deduce that
\[
\sum_{\substack{P^+(b)\le z \\ b>Q^{1/2} }} \frac{c^{\omega(b)}}{b}
	\ll \frac{(\log\log X)^c}{(\log X)^{\kappa/(16A^4)}} .
\]
Putting together the above estimates with $\kappa$ large enough implies that
\eq{S2 e3}{
S_2 \ll \frac{ e^{O(S)} \W }{(\log X)^C},
}
which of admissible size.

Next, we estimate $S_1$. We start by noticing that
\[
S_1 = \sum_{P^+(n)\le z } \mu^2(n) \sum_{\substack{q\le Q \\ \rad(q)=n}}
		\sum_{\substack{ a\in \CA \\ q_{n,\bs a}=q }} w_{\bs a}\delta_n(\bs a)  \\
	=   \sum_{\substack{ P^+(q)\le z \\ q\le Q }}
		\sum_{\bs h\in \CH(q)} \Delta_q(\bs h)
		\sum_{\substack{ \bs a\in \CA \\ \bs a\equiv \bs h\mod{q}}} w_{\bs a} .
\]
For the inner sum, we use the approximation
\[
\frac{W}{|\CG(q)|} + O(E(\CA;q)),
\]
which implies that
\[
S_1 = W  \sum_{\substack{ P^+(q)\le z \\ q\le Q }}
	\sum_{\bs h\in \CH(q)} \frac{\Delta_q(\bs h)}{|\CG(q)|}
	+ O\left(\sum_{\substack{q\le Q \\ P^+(q)\le z}} \sum_{\bs h \in \CH(q)}
		 |\Delta_q(\bs h)| E(\CA;q) \right).
\]
In the main term, we extend the summation of $q$ to infinity. Using the bound \eqref{bound-Delta-q}, we then find that
\[
S_1 = W  \sum_{P^+(q)\le z}
	\sum_{\bs h\in \CH(q)} \frac{\Delta_q(\bs h)}{|\CG(q)|}
	+ O\left( \W  \sum_{\substack{ P^+(q)\le z \\ q>Q}}
		\frac{c_1^{\omega(q)}|\CH(q)|}{\rad(q)|\CG(q)|} +
		\sum_{\substack{q\le Q \\ P^+(q)\le z}} \frac{c_1^{\omega(q)}|\CH(q)|}{\rad(q)}
			 E(\CA;q) 	\right).
\]
By H\"older's inequality and relation \eqref{bt for A}, we find that
\als{
\sum_{\substack{q\le Q \\ P^+(q)\le z}} \frac{c_1^{\omega(q)}|\CH(q)|}{\rad(q)} E(\CA;q)
	&\le E^{1-1/\lambda} \cdot \left(\sum_{\substack{q\le Q \\ P^+(q)\le z}}
	\left(\frac{c_1^{\omega(q)}|\CH(q)|}{q^{d-1}\rad(q)}\right)^\lambda q^{d-1}E(\CA;q)\right)^{1/\lambda} \\
	&\ll E^{1-1/\lambda} \cdot \left(\sum_{\substack{q\le Q \\ P^+(q)\le z}}
		\left(\frac{c_1^{\omega(q)}|\CH(q)|}{q^{d-1}\rad(q)}\right)^\lambda
		\frac{q^{d-1}\W}{|\CG(q)|} \right)^{1/\lambda} .
}
We set $n=\rad(q)$ and $q=bn$, so that $\rad(b)|n$ and $|\CH(q)|\le n^d |\CE(b)|$. Since $|\CG(q)|\ge c_2^{-\omega(q)}q^d$, we deduce that
\[
S_1 = W  \sum_{P^+(q)\le z}
	\sum_{\bs h\in \CH(q)} \frac{\Delta_q(\bs h)}{|\CG(q)|}
	+ O(R) ,
\]
where
\als{
R &= \W  \sum_{\substack{ P^+(bn)\le z \\ bn>Q \\ \rad(b)|n}}
		\frac{\mu^2(n) c_3^{\omega(n)} f(b)}{bn} +
		\W^{1/\lambda} E^{1-1/\lambda}
			\left(\sum_{\substack{P^+(bn)\le z \\ \rad(b)|n}}
				\frac{\mu^2(n) c_3^{\omega(n)} f(b)^\lambda}{bn} \right)^{1/\lambda}  \\
	&\le  \frac{\W}{Q^{\kappa/\log z}}
		 \sum_{\substack{ P^+(bn)\le z \\ \rad(b)|n}}
		\frac{\mu^2(n) c_3^{\omega(n)} f(b)}{(bn)^{1-\kappa/\log z}} +
		\W^{1/\lambda} E^{1-1/\lambda} (\log\log X)^{O(1)} e^{O(S)}
}
for some appropriate constant $c_3$, where we used \eqref{ss proof e1}. Applying H\"older's inequality and \eqref{ss proof e1}, we deduce that
\als{
\sum_{\substack{ P^+(bn)\le z \\ \rad(b)|n}}
		\frac{c_3^{\omega(n)} f(b)}{(bn)^{1-\kappa/\log z}}
	&\le \left(\sum_{\substack{ P^+(bn)\le z \\ \rad(b)|n}}
		\frac{\mu^2(n)}{(bn)^{1-\frac{\kappa\lambda}{\lambda-1}\frac{1}{\log z}} }\right)^{1-\frac{1}{\lambda}}
		\left(\sum_{\substack{ P^+(bn)\le z \\ \rad(b)|n}}
		\frac{\mu^2(n)c_3^{\lambda\omega(n)} f(b)^\lambda}{bn}\right)^{\frac{1}{\lambda}} \\
	&\ll (\log\log X)^{O(1)} e^{O(S)} .
}
Putting together the above estimate and \eqref{S2 e3}, we deduce that
\[
\sum_{a\in \CA} w_{\bs a} P_{\bs a}^{(z)}
	= W \sum_{P^+(q)\le z} \sum_{\bs h\in \CH(q)} \frac{\Delta_q(\bs h)}{|\CG(q)|}
		+ O \left( \frac{e^{O(S)} \W }{(\log X)^C}
			+ \W^{{1}/{\lambda}}E^{1-{1}/{\lambda}}(\log\log X)^{O(1)} e^{O(S)} \right) .
\]
The above estimate and relation \eqref{mainthm e1} imply Theorem \ref{singular series average general}, provided that we can show that
\eq{identity}{
\sum_{P^+(q)\le z} \sum_{\bs h\in \CH(q)} \frac{\Delta_q(\bs h) }{|\CG(q)|}  = 1.
}
We may assume that $S<\infty$; otherwise, the conclusion of Theorem \ref{singular series average general} is trivial. In particular, the series $\sum_{r\ge1} f(\ell^r)^\lambda/\ell^{r+1}$ converges for each $\ell\le z$, which implies that
\eq{ss proof e2}{
\lim_{r\to\infty} \frac{f(\ell^r)}{\ell^{r/\lambda}}   = 0
	\quad\implies\quad
	\lim_{r\to\infty} \frac{|\CE(\ell^r)|}{\ell^{r(d-1+1/\lambda)}} = 0 .
}
Now, using multiplicativity, we see immediately that
\als{
\sum_{{P^+(q)\le z}} \sum_{\bs h\in \CH(q)} \frac{\Delta_q(\bs h)}{|\CG(q)|}
	&=\prod_{\ell\le z} \left(1+\sum_{r\ge1} \sum_{\bs h\in \CH(\ell^r)} \frac{\Delta_{\ell^r}(\bs h)}{|\CG(\ell^r)|}\right) \\
	&=\prod_{\ell\le z} \left(1+\lim_{R\to\infty} \sum_{r=1}^R \sum_{\bs h \in \CH(\ell^r)}		
		\frac{\Delta_{\ell^r}(\bs h)}{|\CG(\ell^r)|}\right) \\	
	&=\prod_{\ell\le z} \left(1+
		\lim_{R\to\infty}  \sum_{\bs h\in \CG(\ell^R) \setminus  \CE(\ell^R) }
		\frac{\Delta_{\ell^R}(\bs h)}{|\CG(\ell^R)|}\right) .
}
Applying conditions (4c) and (4d), and recalling that $|\CG(\ell)|\gg \ell^d$, we find that
\als{
\lim_{R\to\infty} \sum_{\bs h\in \CG(\ell^R) \setminus\CE(\ell^R)} \frac{\Delta_{\ell^R}(\bs h)}{|\CG(\ell^R)|}
		&=  \lim_{R\to\infty} \sum_{\bs h\in \CG(\ell^R)} \frac{\Delta_{\ell^R}(\bs h)}{|\CG(\ell^R)|}
		- \lim_{R\to\infty} \sum_{\bs h\in \CG(\ell^R)\cap \CE(\ell^R) }
\frac{\Delta_{\ell^R}(\bs h)}{|\CG(\ell^R)|} \\
		&\ll  \limsup_{R\to\infty} \frac{ \|\Delta_{\ell^R} \|_\infty \cdot  |\CE(\ell^R)|}{|\CG(\ell^R)|} 	
			\ll_\ell \limsup_{R\to\infty} \frac{|\CE(\ell^R)|}{\ell^{dR}} = 0
}
by \eqref{ss proof e2}. So we deduce that relation \eqref{identity} does hold, thus completing the proof of Theorem \ref{singular series average general}.
\end{proof}


Next, we show Theorem \ref{singular series average} in the special case when $\Delta_\ell=0$ for all primes $\ell$, which is easy to deduce from Theorem \ref{singular series average general}. We need a preliminary result. Given a polynomial $f(x_1,\dots,x_d) \in\Z[x_1,\dots,x_d]$, we introduce the notation
\[
\rho_f(n) := \#\{\bs x \in(\Z/n\Z)^d : f(\bs x)\equiv 0\mod{n}\}.
\]
Then we have the following result, part (b) of which is an easy corollary of a result due to Stewart \cite{Stewart:1991}.


\begin{lma}\label{rho bound} Let $\ell$ be a prime and $f(x_1,\dots,x_d)\in\Z[x_1,\dots,x_d]$ a polynomial of degree $m$.
\begin{enumerate}
\item If $f$ is non-zero modulo $\ell$, then
\[
\rho_f(\ell) \le d m \ell^{d-1} .
\]
\item If $r\in\N$ and $v=\min\{r,\nu_\ell(\SC(f))\}$, then
\[
\rho_f(\ell^r)\le m^d (r+1)^{d-1} \ell^{v/m+r(d-1/m)}  .
\]
\end{enumerate}
\end{lma}


\begin{proof} (a) We use induction on $d$. When $d=1$, the result is straightforward. Assume that it is true for polynomials of $d-1$ variables. We write $f$ as a polynomial of $x_d$, whose coefficients are polynomials in $x_1,\dots,x_{d-1}$. At least one of these coefficients must be non-zero modulo $\ell$. Call this coefficient $g(x_1,\dots,x_{d-1})$. For each given choice of $x_1,\dots,x_{d-1}$, either $f$ is non-zero as a polynomial of $x_d$ modulo $\ell$, in which case we have $\le m$ choices for $x_d$ with $f(x_1,\dots,x_d)\equiv0\mod\ell$, or $f$ is zero as a polynomial of $x_d$ modulo $\ell$, in which case we have $\ell$ choices for $x_d$ but also $g(x_1,\dots,x_{d-1})\equiv 0\mod{\ell}$. Therefore
\[
\rho_f(\ell) \le m\cdot \ell^{d-1} + \ell\cdot \rho_g(\ell)	
	\le m\ell^{d-1} + \ell \cdot (d-1)m \ell^{d-2} = dm\ell^{d-1} .
\]
This completes the inductive step and hence the proof of part (a).

\medskip

(b) First, we deal with the case $d=1$, in which case we have to show that
\eq{rho base case}{
\rho_f(\ell^r)\le m \ell^{v/m+r(1-1/m)}  .
}
Let $f_k(x) = f(x) + k\ell^r$ and note that $\rho_{f_k}(\ell^r)=\rho_f(\ell^r)$ and $\min\{r,\nu_\ell(\SC(f_k))\} = v$. On the other hand, if $k\to\infty$, then the roots of $f_k$ over $\C$ tend to infinity too, so they cannot be roots of $f_k'$ at the same time for $k$ large enough. So, by taking $k$ large enough, we may assume that the discriminant of $f$ is non-zero.

Next, write $f(x)=\SC(f)\cdot g(x)$, where $g(x)$ is primitive (that is to say its content is 1). Clearly,
\[
\rho_f(\ell^r) =  \ell^{v} \rho_g(\ell^{r-v})  ,
\]
which reduces the lemma to the case when $\nu_\ell(\SC(f))=0$. The result then follows by \cite[Corollary 2 and eq. (44)]{Stewart:1991} when $m\ge2$, and trivially when $m=1$.

Next, we show the general case. We argue by induction on $d$. Assume that the lemma holds for polynomials of $d-1$ variables. As in the case $d=1$, writing $f(x_1,\dots,x_d)=\SC(f)\cdot g(x_1,\dots,x_d)$ allows us to assume that $f$ is primitive. There are polynomials $c_j(x_1,\dots,x_{d-1})$ of degree $\le m$ such that
\[
f(x_1,\dots,x_d) = \sum_{j=0}^m c_j(x_1,\dots,x_{d-1}) x_d^j .
\]
Clearly, since $f$ is primitive, there must be at least one $j_0\in\{0,1,\dots,m\}$ such that the content of $c_{j_0}$ is not divisible by $\ell$. For each $a_1,\dots,a_{d-1}\in\Z$, we write $C(a_1,\dots,a_{d-1})$ for the greatest common divisor of the polynomial values $c_j(a_1,\dots, a_{d-1})$, $0\le j\le m$. Then
\als{
\rho_f(\ell^r)
	&= \sum_{w=0}^r \sum_{\substack{0\le a_1,\dots,a_{d-1}<\ell^r \\
	\min\{r,\nu_\ell(C(a_1,\dots,a_{d-1}))\} = w }}
		\#\{0\le x_d<\ell^r : f(a_1,\dots,a_{d-1},x_d)=0 \} \\
	&\le \sum_{w=0}^r \sum_{\substack{0\le a_1,\dots,a_{d-1}<\ell^r \\
	\min\{r,\nu_\ell(C(a_1,\dots,a_{d-1}))\} = w }} m \ell^{w/m+r(1-1/m)}  ,
}
by the base case \eqref{rho base case}. Note that if $\min\{r,\nu_\ell(C(a_1,\dots,a_{d-1}))\} = w$, then $\ell^w | c_{j_0}(a_1,\dots,a_{d-1})$. So the number of such $a_1,\dots,a_{d-1}$ mod $\ell^r$ is at most
\[
\ell^{(d-1)(r-w)} \rho_{c_{j_0}}(\ell^w) \le m^{d-1} (w+1)^{d-2} \ell^{(d-1)(r-w)+ w(d-1-1/m)}
	\le m^{d-1} (r+1)^{d-2} \ell^{(d-1)r- w/m},
\]
by the induction hypothesis and our assumption that the content of $c_{j_0}$ is not divisible by $\ell$.
Therefore
\[
\rho_f(\ell^r)
	\le m^d (r+1)^{d-2} \sum_{w=0}^r \ell^{(d-1)r- w/m+w/m+(1-1/m)r}
	= m^d (r+1)^{d-1} \ell^{r(d-1/m)} ,
\]
which completes the induction hypothesis and, hence, the proof of the lemma.
\end{proof}

%

\begin{proof}[Proof of Theorem \ref{singular series average} when $\Delta_\ell=0$.] Without loss of generality, we may assume that the parameter $\eta$ lies in the interval $(0,1/2]$. We first check that we can apply Theorem \ref{singular series average general} under the hypothesis of Theorem \ref{singular series average} in this case. Conditions (1) and (2) hold by assumption. Condition (3') implies that condition (3) holds with $L_{\bs a} = L\cdot |D_1(\bs a) \cdots D_k(\bs a)|$, and (5) holds since a non-zero integer $n$ has $O(\log|n|)$ prime divisors. Finally, condition (4) is satisfied with
\[
\CE(\ell^r) = \{\bs h\in (\Z/\ell^r\Z)^d : \ell^r| F(\bs h) \},
\]
since we have assumed that $\Delta_\ell=0$ here and that $F(\bs a)\neq0$ for $\bs a\in \CA$ in condition (5'). In conclusion, Theorem \ref{singular series average general} is applicable. We need to control the quantity $S$ appearing in its statement. We will take $\lambda=1+1/m$, where $m=\deg(F)$. Since the content of $F$ is $\ll1$ by condition (5'), Lemma \ref{rho bound}(b) implies that
\[
\rho_F(\ell^r) \ll (r+1)^{d-1} \ell^{r(d-1/m)} = \ell^{r(d-1)}\cdot (r+1)^{d-1} \ell^{r(1-1/m)} .
\]
So
\[
\frac{(|\CE(\ell^r)| / \ell^{r(d-1)})^\lambda}{\ell^{r+1}}
	\le  \frac{(\rho_F(\ell^r) /\ell^{r(d-1)})^\lambda}{\ell^{r+1}}
	\ll \frac{(r+1)^{\lambda(d-1)}}{\ell^{1+r(1-\lambda(1-1/m))} } .
\]
Therefore, we see immediately that $S\ll 1$, and Theorem \ref{singular series average}  follows in this special case by Theorem \ref{singular series average general}.
\end{proof}


The remainder of this section is devoted to showing that Theorem \ref{singular series average} can be indeed reduced to the special case when $\Delta_\ell=0$. As before, we may assume that $\eta\in(0,1/2]$. The key step is proving that the quantities $\Delta_\ell$ in condition (4') satisfy the estimate
\eq{delta-goal}{
\Delta_\ell \ll   \frac{1}{\ell^{3/2}}  \quad(\ell \nmid LN),
}
where $L$ is as in condition (3') and $N$ is some appropriate non-zero integer of size $\le X^{O(1)}$. We will show how to construct $N$ later in this section. For now, let us see how \eqref{delta-goal} allows us to reduce Theorem \ref{singular series average} to the case when $\Delta_\ell=0$.

\begin{proof}[Deduction of Theorem \ref{singular series average} from \eqref{delta-goal}]
Note that $P=\prod_\ell(1+\Delta_\ell)$ converges absolutely by \eqref{delta-goal}. We define $\delta_\ell'(\bs a)$ via the relation
\[
\begin{cases}
	\ds  1+\delta_\ell(\bs a) = (1+\delta_\ell'(\bs a))(1+\Delta_\ell)
		&\text{if}\ \bs a\mod\ell\in \CG(\ell), \\
	\delta_\ell'(\bs a) = 0 									
		&\text{otherwise},
\end{cases}
\]
so that
\[
\sum_{\bs a\in \CA} w_{\bs a}P_{\bs a}
	= P\cdot \sum_{\bs a\in \CA} w_{\bs a} \gamma_{\bs a}\prod_\ell (1+\delta_\ell'(\bs a)) ,
\]
where
\[
\gamma_{\bs a}
	= \prod_{\substack{\ell \\ \bs a\mod\ell\notin\CG(\ell)}}\frac{1}{1+\Delta_\ell}
	= \prod_{\substack{\ell> Q  \\ \bs a\mod\ell\notin\CG(\ell)}}
		\frac{1}{1+\Delta_\ell}
				\quad(\bs a\in \CA) ,
\]
since we have assumed that $\CA\subset \{\bs a\in \Z^d:\bs a\mod{\ell}\in\CG(\ell)\}$ for all $\ell\le Q$.

We now show that we can apply Theorem \ref{singular series average} to the quantities $\delta_\ell'(\bs a)$. Condition (1) holds for $\delta_\ell'(\bs a)$ by definition. Conditions (4'c) and (4'd) imply that $\Delta_\ell\ll1/\ell$. Since we also have that $|1+\Delta_\ell|\gg1$ by (4'c), we deduce that
\[
\delta_\ell'(\bs a) = \delta_\ell(\bs a) + O\left(\frac{1}{\ell}\right) ,
\]
so that condition (2) holds for $\delta_\ell'(\bs a)$ too by the same condition for $\delta_\ell(\bs a)$. Condition (3') holds by \eqref{delta-goal} with $L$ replaced by $LN$. Defining
\[
1 + \Delta^\prime_{\ell^r}(\bs a)
	= \frac{1 + \Delta_{\ell^r}(\bs a)}{1+\Delta_\ell},
\]
we see that condition (4') holds for the $\ell^r$-periodic function $\Delta^\prime_{\ell^r}(\bs a)$, which has average $\Delta^\prime_{\ell} = 0$. Condition (5') is also easily seen to hold.

Applying Theorem \ref{singular series average} to the sequence $\delta_\ell'(\bs a)$, with the weights $w_{\bs a}$ replaced by $w_{\bs a}\gamma_{\bs a}$, we get
that
\[
\sum_{\bs a\in \CA} w_{\bs a} P_{\bs a}  =  P \cdot \left( W'
	+ O\left(  \frac{\W'}{(\log X)^C} + M' X^\epsilon
		+ (\log\log X)^{O(1)} \W'^{1-1/(m+1)} E'^{1/(m+1)} \right) \right),
\]
where $W'$, $M'$ and $E'$ are defined as $W$, $M$ and $E$, with the difference that $w_{\bs a}$ is replaced by $w_{\bs a}\gamma_{\bs a}$, and $\W':=\W\cdot \max_{\bs a\in\CA} |\gamma_{\bs a}|$.

We will show that $\gamma_{\bs a}$ is very close to 1 and, as a result, relate $W'$ to $W$, $M'$ to $M$, $E'$ to $E$, and $\W'$ to $\W$. Note that
\[
\sum_{\ell>Q} |\Delta_\ell|
	\ll \sum_{\ell>Q}  \frac{1}{\ell^{1+\eta}}
		+ 	\sum_{\substack{\ell> Q \\ \ell| LN}} \frac{1}{\ell}
	\ll \frac{1}{Q^\eta} + \frac{(\log X)^{O(1)}}{Q} \ll \frac{1}{(\log X)^{(m+1)C}} ,
\]
by \eqref{delta-goal} and the fact that
\[
\omega(LN) \le \omega(L)+\omega(N)\le \omega(L)+\frac{\log|N|}{\log 2} \ll (\log X)^{O(1)},
\]
a consequence of condition (5'), and our assumption that $N \leq X^{O(1)}$. Therefore,
\[
\gamma_{\bs a} = 1 + O\left( \frac{1}{(\log X)^{(m+1)C}} \right) .
\]
So, we see immediately that
\[
W'= \sum_{\bs a\in \CA} w_{\bs a}\gamma_{\bs a} = W+ O\left(\frac{\W}{(\log X)^{(m+1)C}}\right)
\]
and, if $E$ is as in the statement of Theorem \ref{singular series average}, then relation \eqref{bt for A} implies that
\als{
\sum_{q\le Q} q^{d-1} \max_{\bs g\in \CG(q)}
	\left| \sum_{\substack{\bs a\in \CA \\ \bs a\equiv \bs g\mod q}}w_{\bs a} \gamma_{\bs a}- \frac{W'}{|\CG(q)|}\right|
	&\ll E + \sum_{q\le Q} q^{d-1} \cdot \frac{\W}{|\CG(q)| (\log X)^{C(m+1)}}  \\
	&\ll E +  \W\cdot \frac{(\log\log X)^{O(1)}}{(\log X)^{C(m+1)}},
}
since $|\CG(q)|\ge q e^{-O(\omega(q))}$ by \eqref{G} and the fact that $|\CG(\ell^r)|=\ell^{(r-1)d}|\CG(\ell)|$. This proves Theorem \ref{singular series average}. 
\end{proof}

The crucial result for the construction of the number $N$ and the deduction of \eqref{delta-goal} is provided by the following theorem. Recall that $H(f)$ denotes the height of a polynomial $f$.

%

\begin{thm}\label{deligne} Let $f(x_1,\dots,x_d)\in \Z[x_1,\dots,x_d]$ of degree $m$. If $f$ is not of the form $c\cdot g(x_1,\dots,x_d)^2$, where $c\in\Q$ and $g\in\Q[x_1,\dots,x_d]$, then there is some $B\in\Z$ such that $1\le|B| \le H(f)^{O_{d,m}(1)}$ and for all primes $\ell\nmid B$,
\[
\sum_{x_1,\dots,x_d\in\Z/\ell\Z} \leg{f(x_1,\dots,x_d)}{\ell} \ll_{d,m} \ell^{d-1/2} .
\]
\end{thm}


Before proving Theorem \ref{deligne}, let us see how it can be used to construct $N$.

\begin{proof}[Proof of \eqref{delta-goal}] We shall prove this relation with $N=\SC(F)B_1\cdots B_k$, where $B_j$ is the number $B$ associated to the polynomial $D_j$ by Theorem \ref{deligne}. Such a number exists because Gauss's lemma and our assumption that the polynomials $\pm D_j/\SC(D_j)$ are not squares in $\Z[x_1,\dots,x_d]$ (see condition (5')) imply that $D_j$ is not of the form $c\cdot g^2$, where $c\in\Q$ and $g$ is a polynomial over $\Q$ in $d$ variables. Moreover, $\SC(F)\ll1$ and $B_j \le H(D_j)^{O(1)} \le X^{O(1)}$, where we used condition (5) again. Therefore, $N\ll X^{O(1)}$, as claimed.

Now, fix a prime number $\ell\nmid LN$. Note that if $\ell\nmid F(\bs n) D_1(\bs n)\cdots D_k(\bs n)$, then $\bs n\mod\ell \in \CG(\ell)$ and
\[
\Delta_{\ell^r}(\bs n)
	= \delta_\ell(\bs n)
	= \frac{\lambda_1 \leg{D_1(\bs n)}{\ell} +  \cdots + \lambda_k \leg{D_k(\bs n)}{\ell} }{\ell}
		+ O\left(\frac{1}{\ell^{1+\eta}}\right)
\]
for all $r\ge1$, by conditions (3') and (4'b). Set $D_0=F$. Since $\|\Delta_{\ell^r}\|_\infty \ll1/\ell$ and $\CG(\ell^r) = \{\bs g\in\Z/\ell^r\Z:\bs g\mod{\ell}\in \CG(\ell)\}$, we have that
\als{
\frac{1}{|\CG(\ell^r)|} \sum_{\bs n\in \CG(\ell^r)} \Delta_{\ell^r}(\bs n)
	&= \frac{1}{|\CG(\ell^r)|\ell} \sum_{j=1}^k \lambda_j
		\sum_{\bs n\in \CG(\ell^r)}  \leg{D_j(\bs n)}{\ell} \\
	&	\quad
		+ O\left(\frac{1}{\ell^{1+\eta}}
		+ \sum_{j=0}^k \frac{\#\{\bs n\in (\Z/\ell^r\Z)^d: \ell|D_j(\bs n)\}}{|\CG(\ell^r)| \ell} \right) \\
	&= \frac{1}{|\CG(\ell)|\ell} \sum_{j=1}^k \lambda_j
		\sum_{\bs n\in \CG(\ell)}  \leg{D_j(\bs n)}{\ell}
		+ O\left(\frac{1}{\ell^{1+\eta}} + \sum_{j=0}^k \frac{\rho_{D_j}(\ell)}{|\CG(\ell)|\ell} \right) .
}
The polynomial $D_0$ is non-zero mod $\ell$ because $\ell\nmid \SC(F)=\SC(D_0)$, and the same is true for $D_1,\dots,D_k$, because $\ell\nmid B_1\cdots B_k$, whence
\[
 \sum_{j=0}^k \frac{\rho_{D_j}(\ell)}{|\CG(\ell)|\ell} \ll \frac{1}{|\CG(\ell)| \ell} .
\]
Combining the above estimates with \eqref{G-strong} that states that $|\CG(\ell)|=\ell^d+O(\ell^{d-\eta})$, we conclude that
\[
\frac{1}{|\CG(\ell^r)|} \sum_{\bs n\in \CG(\ell^r)} \Delta_{\ell^r}(\bs n)
	=  \frac{1}{\ell^{d+1}} \sum_{j=1}^k \lambda_j
		\sum_{\bs n\in (\Z/\ell\Z)^d} \leg{D_j(\bs n)}{\ell}
		+ O\left(\frac{1}{\ell^{1+\eta}}\right) .
\]
Finally, the condition that $\ell\nmid B_1\cdots B_k$ implies that $D_j$ cannot be of the form $cg^2$ mod $\ell$. Applying Theorem \ref{deligne} and letting $r\to\infty$ then completes the proof of \eqref{delta-goal}.
\end{proof}

In order to complete the proof of \eqref{delta-goal}, and thus of Theorem \ref{singular series average}, it remains to prove Theorem \ref{deligne}. We need a preliminary result.

%

\begin{lma}\label{squares} Let $d,m\in\N$ and set $N=(2m+1)^d-(m+1)^d$. There are homogenous polynomials $S_1,\dots,S_N$ over $\Z$ in $(2m+1)^d$ variables, depending at most on $m$ and $d$, with the following property. If $K$ is a field of characteristic different from 2 and
\[
f(x_1,\dots,x_d) = \sum_{0\le i_1,\dots,i_d\le 2m}  c_{i_1,\dots,i_d} x_1^{i_1}\cdots x_d^{i_d} \in K[x_1,\dots,x_d]
\]
with $c_{0,\dots,0}\neq0$, then $f$ is of the form $c\cdot g^2$, where $c\in K$ and $g$ is an element of the ring $K[x_1,\dots,x_d]$, if, and only if, $S_j(\{c_{i_1,\dots,i_d}:0\le i_1,\dots,i_d\le 2m\})=0$ for all $j\in\{1,\dots,N\}$.
\end{lma}

%

\begin{proof} We shall denote $(i_1,\dots,i_d)$ by $\bs i$, and $(0,\dots,0)$ by $\bs 0$. Also, we write $\bs i\le \bs j$ if $i_n\le j_n$ for all $n\in\{1,\dots,d\}$. Set $\tilde{f}=f/ c_{\bs 0}$. The coefficients of $\tilde f$ are the numbers $\tilde{c}_{\bs i}:= c_{\bs i}/c_{\bs 0}$. In particular, $\tilde{c}_{\bs 0}=1$. Clearly, $f$ is of the form $cg^2$ if, and only if, $\tilde f$ is of the same form. If, now, $\tilde{f}=cg^2$, then we must have that $g(\bs 0)^2=1/c$. This means that we may restrict our attention to studying whether the equation $\tilde{f}=g^2$ has a solution with $g(\bs 0)=1$. This condition is equivalent to the existence of coefficients $a_{\bs i}\in K$,  $0\le i_1,\dots,i_d\le m$, such that $a_{\bs 0}=1$ and
\eq{squares-e1}{
\sum_{\substack{\bs i+\bs j = \bs k}} a_{\bs i} a_{\bs j}= \tilde{c}_{\bs k} ,
}
for $0\le k_1,\dots,k_d\le 2m$. We claim that the conditions \eqref{squares-e1} for $k_1,\dots,k_d\le m$ are altogether equivalent to having that
\[
a_{\bs i} = P_{\bs i}( \{\tilde{c}_{\bs r} : \bs r\le \bs i\})
\]
for $\bs i\neq\bs0$, where $P_{\bs i}$ is a polynomial in $(i_1+1)\cdots (i_d+1)$ variables, whose coefficients are of the form $2^{v} n$ with $n,v\in\Z$. Indeed, since $a_{\bs 0}=1$, we have that
\[
2a_{\bs k} = \tilde{c}_{\bs k} - \sum_{\substack{\bs i+\bs j=\bs k \\ \bs i,\bs j\neq\bs 0}} a_{\bs i} a_{\bs j} \quad(k_1,\dots,k_d\le m) .
\]
Proceeding by induction on $\sum_{n=1}^d k_n$ proves our claim about the polynomials $P_i$ (where it is clear that the coefficients of $P_i$ depends only on $m$ and $d$). Finally, we substitute the expressions we have found for $a_{\bs i}$ to \eqref{squares-e1} when at least one of the $k_n$'s is $>m$. This implies that relation \eqref{squares-e1} with $0\le k_1,\dots,k_d\le 2m$ is actually equivalent to the relations
\[
\tilde{c}_{\bs k} =  \sum_{\substack{\bs i+\bs j = \bs k}} P_{\bs i}( \{\tilde{c}_{\bs r} : \bs r\le \bs i\})
	\cdot P_{\bs j}( \{\tilde{c}_{\bs r} : \bs r \le \bs j\})
	\quad (m<\max\{k_1,\dots,k_d\}\le 2m) .
\]
Multiplying by a high enough power of $c_{\bs 0}$ the above equations yields $(2m+1)^d-(m+1)^d$ homogeneous polynomial equations in the coefficients of $f$ that are equivalent to $f$ being of the form $cg^2$. This completes the proof of the lemma.
\end{proof}

%

\begin{proof}[Proof of Theorem \ref{deligne}] All implied constants might depend on $d$ and $m$. It is easy to see by induction on $d$ that there are integers $n_1,\dots,n_d\in[0,m]$ such that $f(n_1,\dots,n_d)\neq0$. So, replacing $f$ by $\widetilde{f}(x_1,\dots,x_d) = f(x_1+n_1,\dots,x_d+n_d)$, we may assume without loss of generality that $f(\bs 0)\neq0$. Since $f$ is not of the form $cg^2$ over $\Q$, Lemma \ref{squares} implies that there is some integer polynomial expression in the coefficients of $f$, let's call it $B'$, which is not zero. Moreover, if $\ell\nmid B'$, where $\ell$ is an odd prime, and $\ell\nmid f(0,\dots,0)$, then the same lemma implies that $f(x_1,\dots,x_d)$ is not of the form $c\cdot g(x_1,\dots,x_d)^2$ modulo $\ell$. We will show that the theorem holds with $B=2B'f(\bs0) = H(f)^{O(1)}$. So, we need to show that if $\ell\nmid B$, then
\eq{deligne-e1}{
\sum_{x_1,\dots,x_d\in \Z/\ell\Z} \leg{f(x_1,\dots,x_d)}{\ell} \ll \ell^{d-1/2} .
}
It suffices to show that if $\ell$ is an odd prime such that $\ell\nmid f(\bs 0)$ and modulo which $f$ is not of the form $cg^2$, then \eqref{deligne-e1} is true. Fix such a prime $\ell$. We argue by induction on $d$. If $d=1$, then this follows by \cite[Theorem 11.13, p. 281]{IK} applied to the curve $y^2=f(x_1)$. (Note that the condition there that the polynomial $y^2-f(x_1)$ is absolutely irreducible is equivalent to $f$ not being of the form $cg^2$.) Assume now that the theorem is true for polynomials of $<d$ variables. We write
\eq{deligne-e2}{
f(x_1,\dots,x_d) = \sum_{j=0}^{m} c_j(x_1,\dots,x_{d-1}) x_d^j.
}
We have that $c_0(\bs0)=f(\bs0)\not\equiv 0\mod{\ell}$. In particular, $c_0$ is non-zero mod $\ell$.
We distinguish two cases.

\medskip

\textit{Case 1: $c_0f$ is a perfect square mod $\ell$.} In this case, we can reduce to the case of $d-1$ variables: we have that
\als{
\sum_{x_1,\dots,x_d\in \Z/\ell\Z} \leg{f(x_1,\dots,x_d)}{\ell}
	&= \sum_{\substack{x_1,\dots,x_d\in \Z/\ell\Z \\ c_0(x_1,\dots,x_{d-1})\neq0 \\ f(x_1,\dots,x_d)\neq0}}
		\leg{c_0(x_1,\dots,x_{d-1})}{\ell} + O(\ell\cdot \rho_{c_0}(\ell) + \rho_f(\ell) ) \\
	&= \ell \sum_{x_1,\dots,x_{d-1}\in \Z/\ell\Z}
		\leg{c_0(x_1,\dots,x_{d-1})}{\ell} + O(\ell^{d-1}),
}
by Lemma \ref{rho bound}(a). Since $c_0f$ is a perfect square and $f$ is not of the form $cg^2$ mod $\ell$, we must have that $c_0$ is not of the form $cg^2$ mod $\ell$ either. Moreover, $\ell\nmid c_0(\bs 0)=f(\bs 0)$, and the induction hypothesis implies \eqref{deligne-e1} in this case.

\medskip

\textit{Case 2: $c_0f$ is not a perfect square mod $\ell$.} We claim that in this case there are $O(\ell^{d-2})$ choices of $n_1,\dots,n_{d-1}\in\Z/\ell\Z$ such that the polynomial $f(n_1,\dots,n_{d-1},x_d)$ is of the form $c\cdot g(x_d)^2$ mod $\ell$ as a polynomial of $x_d$. This suffices to deduce \eqref{deligne-e1}. Indeed, if $n_1,\dots,n_{d-1}$ are such that $f(n_1,\dots,n_{d-1},x_d)$ is not of the form $c\cdot g(x_d)^2$ mod $\ell$, then applying \eqref{deligne-e1} with $d=1$ implies that
\[
\sum_{x_d\in \Z/\ell\Z} \leg{f(n_1,\dots,n_{d-1},x_d)}{\ell} \ll \ell^{1/2} ,
\]
and the proof of the inductive step is completed. Thus, it suffices to prove our claim.

Fix $n_1,\dots,n_{d-1}$ such that
 $f(n_1,\dots,n_{d-1},x_d)$ is of the form $c\cdot g(x_d)^2$ as a polynomial of $x_d$. The coefficients of $f(n_1,\dots,n_{d-1},x_d)$ as a polynomial of $x_d$ are given by \eqref{deligne-e2}. By Lemma \ref{rho bound}(a), there are only $O(\ell^{d-2})$ choices of $n_1,\dots,n_{d-1}$ such that $c_0(n_1,\dots,n_{d-1})\equiv 0\mod{\ell}$. Assume, now, that $c_0(n_1,\dots,n_{d-1})\not\equiv 0\mod{\ell}$. Then, following the proof of Lemma \ref{squares}, we see that $F(x_d)=f(n_1,\dots,n_{d-1},x_d)/c_0(n_1,\dots,n_{d-1})$ must be of the form $(1+a_1x_d+\cdots+a_{m'}x_d^{m'})^2$, for certain integers $a_j$ which are polynomial expressions in the coefficients of $F$, that is to say the $a_j$'s are polynomials in the rational functions $c_i(n_1,\dots,n_{d-1})/c_0(n_1,\dots,n_{d-1})$, $1\le i\le d$. Multiplying through by $c_0(n_1,\dots,n_{d-1})^{2k}$ for a large enough $k$, we see that
\eq{deligne-e3}{
c_0(n_1,\dots,n_{d-1})^{2k-1} f(n_1,\dots,n_{d-1},x_d) \equiv h(n_1,\dots,n_{d-1},x_d)^2
		\quad \mod{\ell}
}
for all $x_d\in\Z/\ell\Z$, where $h$ is some polynomial in $\Z[x_1,\dots,x_d]$. However, we know that the polynomial
\[
c_0(x_1,\dots,x_{d-1})^{2k-1} f(x_1,\dots,x_d) -h(x_1,\dots,x_d)^2
\]
is non-zero in the polynomial ring $(\Z/\ell\Z)[x_1,\dots,x_d]$, by our assumption that $c_0f$ is not a perfect square and the fact that $(\Z/\ell\Z)[x_1,\dots,x_d]$ is a unique factorisation domain. Consequently, Lemma \ref{rho bound}(a) implies that the number of $n_1,\dots,n_{d-1}$ for which \eqref{deligne-e3} holds for all $x_d\in\Z/\ell\Z$ is $\ll \ell^{d-2}$.
This completes the proof of our claim that there are $O(\ell^{d-2})$ choices of $n_1,\dots,n_{d-1}\in\Z/\ell\Z$ such that the polynomial $f(n_1,\dots,n_{d-1},x_d)$ is of the form $c\cdot g(x_d)^2$ as a polynomial of $x_d$. Hence, \eqref{deligne-e1} follows in this last case too. This completes the proof of Lemma \ref{deligne}.
\end{proof}


\section{An auxiliary result} \label{auxiliary}

We prove here the promised estimate needed to handle the main and the error term in the proof of Theorems \ref{Koblitz} and \ref{aliquot}. The main input to our result is the following estimate about primes in short arithmetic progressions, proven in \cite{Kou}.


\begin{lma}\label{bv-short}Fix $\epsilon>0$ and $A\ge1$. For $x\ge h\ge2$ and $1\le Q^2\le h/x^{1/6+\epsilon}$, we have that
\[
\int_x^{2x}\sum_{q\le Q} E(y,h;q) \dee y \ll\frac{xh}{(\log x)^A},
\]
where $E(y,h;q)$ is defined by \eqref{def-E}.
\end{lma}


\begin{lma}\label{W}
Fix $\epsilon>0$, $A\ge1$ and two integers $d\ge2$ and $m\in\{0,1,\dots,d\}$. Given a $d$-tuple $\bs p$ in the set
\[
\CP_d'(x) := \{(p_1,\dots,p_d): x<p_1\le 2x,\ |p_{j+1}-p_j-1|<2\sqrt{p_j}\ (1\le j\le d)\ \text{with}\ p_{d+1}=p_1\},
\]
we let
\[
w_{\bs p} =  \prod_{j=1}^m  \frac{1}{\sqrt{p_j}} \sqrt{1-\left(\frac{p_j+1-p_{j+1}}{2\sqrt{p_j}}\right)^2} .
\]
Then,
\[
\sum_{q\le x^{1/6-\epsilon}} q^{d-1} \max_{\bs a\in ((\Z/q\Z)^*)^d}
	\left|\sum_{\substack{\bs p \in \CP_d'(x) \\ \bs p\equiv \bs a\mod q} } w_{\bs p}
	 - \frac{I_{d,m}}{\phi(q)^d}  \cdot \int_x^{2x} \frac{u^{(d-m-1)/2} \dee u}{(\log u)^d} \right|
	 \ll_{A,\epsilon,d} \frac{x^{(d-m+1)/2 }}{(\log x)^A} 	
\]
for all $x\ge3$, where
\[
I_{d,m} = 2^{d-1} \idotsint\limits_{\substack{|t_j|\le 1\ (1\le j\le d) \\ t_1+\cdots+t_d=0 }}
		\prod_{j=1}^m \sqrt{1-t_j^2}  \ \dee t_1\cdots \dee t_{d-1}  .
\]
\end{lma}


\begin{proof} We fix a parameter $B=B(A)$ and set $\eta=1/(\log x)^B$. Furthermore, we set $N = \fl{(\log x)^B}-1$, so that $(N+1)\eta\le 1<(N+2)\eta$. Note that if $\bs p\in \CP_d'(x)$, then $|p_{j+1}-p_j|\ll\sqrt{p_j}$ for all $j$. In particular, $p_{j+1}\asymp p_j$, which implies that $p_j\asymp p_1\asymp x$ for all $j$. So, we find that $|p_{j+1}-p_j|\ll\sqrt{x}$ and, consequently, $|p_j-p_1|\ll\sqrt{x}$ for all $j\in\{1,\dots,d\}$. In particular, $\sqrt{p_j}= \sqrt{p_1}  + O(1)$ by the Mean Value Theorem.
 If we let
\[
\CQ = \{\bs p: x<p_1\le (1+N\eta) x,\quad
			|p_{j+1}-p_j|<2\sqrt{p_1} N\eta
				\ (1\le j\le d)\} ,
\]
with the usual convention that $p_{d+1}=p_1$, then we see that $\CQ\subset \CP_d'(x)$ and that
\[
\sum_{\substack{\bs p \in \CP_d'(x) \setminus \CQ \\ \bs p\equiv a\mod{q} }}
	 \prod_{j=1}^m  \frac{1}{\sqrt{p_j}} \sqrt{1-\left(\frac{p_j+1-p_{j+1}}{2\sqrt{p_j}}\right)^2}
		\ll \frac{x^{(d-m+1)/2}}{q^d\cdot (\log x)^B} .
\]
For each $\bs r\in ([1,N]\times[-N+1,N]^{d-1})\cap\Z^d$ with $|r_2+\cdots+r_d|\le N$, we define
\als{
\CQ(\bs r)
	= \Bigg\{\bs p: &\  x(1+(r_1-1)\eta)<p_1\le x(1+r_1\eta), \\
			&  (r_{j+1}-1)\eta < \frac{p_{j+1}-p_j}{2\sqrt{x(1+r_1\eta)}} \le  r_{j+1} \eta
				\ (1\le j < d) \Bigg\} .
}
If $r_{d+1}:= -(r_2+\cdots+r_d)$, then
\[
\frac{p_{d}-p_{d+1}}{2 \sqrt{x(1+r_1 \eta)}}
	= \frac{p_d-p_1}{2 \sqrt{x(1+r_1 \eta)}}
	= \sum_{j=1}^{d-1} \frac{p_{j+1} - p_j}{2 \sqrt{x(1+r_1 \eta)}}
	= \sum_{j=2}^d r_j\eta + O(\eta)
	= -\eta r_{d+1} + O(\eta) ,
\]
Let $\CH$ be the set of $\bs r\in ([1,N]\times[-N+C,N-C]^{d-1})\cap\Z^d$ with $|r_{d+1}|\le N-C$, where $C$ is some large constant, and call $\CQ'$ the union of $\CQ(\bs r)$ over $\bs r\in \CH$. If $C$ is large enough, then it is easy to see that $\CQ'\subset \CQ$. Moreover, we have that
\[
\sum_{\substack{\bs p \in \CQ \setminus \CQ' \\ \bs p\equiv a\mod{q} }} w_{\bs p}
		\ll \frac{x^{(d-m+1)/2}}{q^d\cdot (\log x)^B} .
\]
Since $|\CH|\asymp \eta^{-d}$, it suffices to show that
\eq{W e1}{
\sum_{q\le x^{1/6-\epsilon}}
	q^{d-1} \max_{\bs a\in ((\Z/q\Z)^*)^d}
		\left|\sum_{\substack{\bs p \in  \CQ(\bs r) \\ \bs p\equiv a\mod{q} }} w_{\bs p}
			- \frac{2^{d-1} I(\bs r)}{\phi(q)^d} \right|	
	\ll \frac{\eta^d x^{(d-m+1)/2}}{(\log x)^{A}} ,
}
where, for each $\bs r = (r_1, \dots, r_d) \in \CH$, we define
\[
x_1=x+r_1\eta x
\]
and
\[
I(\bs r) = \int_{x_1-\eta x}^{x_1}
			\frac{u^{(d-m-1)/2}\dee u}{(\log u)^d}
			\idotsint\limits_{\substack{(r_{j+1}-1)\eta\le t_j\le r_{j+1}\eta \\
				(1\le j\le d-1) \\ t_1+\cdots+t_d=0}}
		\, \prod_{j=1}^m \sqrt{1-t_j^2}  \ \dee t_1\cdots \dee t_{d-1} .
\]
If $\bs r$ and $x_1$ are as above and $\bs p\in\CQ(\bs r)$, then
\[
\frac{p_{j+1}-p_j-1}{2\sqrt{p_j}}
	= \frac{\eta r_{j+1}\sqrt{x_1} + O(\eta\sqrt{x_1})}{(1+O(\eta))\sqrt{x_1}}
		= \eta r_{j+1} + O(\eta) \quad(1\le j\le d) ,
\]
where we used the fact that $r_j\ll N\asymp 1/\eta$. Thus, if $C$ is large enough, our assumption that $|r_j|\le N-C=1/\eta-C+O(1)$ implies that $|1-\eta|r_{j+1}|\gg \eta$. Applying the Mean Value Theorem we then find that
\als{
\prod_{j=1}^m \sqrt{1-\left(\frac{p_{j+1}-p_j-1}{2\sqrt{p_j}}\right)^2 	}
	&= \prod_{j=1}^m \sqrt{1-(r_{j+1}\eta)^2}  + O\left(\sum_{j=1}^m \frac{\eta}{\sqrt{1-|r_{j+1}\eta|}} \right) \\
	&= \frac{1}{\eta^{d-1}}
		\idotsint\limits_{\substack{(r_{j+1}-1)\eta\le t_j\le r_{j+1}\eta \\ (1\le j\le d-1) \\ t_1+\cdots+t_d=0 }}
		\, \prod_{j=1}^m \sqrt{1-t_j^2}  \ \dee t_1\cdots \dee t_{d-1}
		 	+ O(\sqrt{\eta})
}
for all $\bs p\in \CQ(\bs r)$, where the condition $t_d=-(t_1+\cdots+t_{d-1})$ comes from the fact that $r_{d+1} = - \sum_{j=2}^d r_j$. Moreover, we have that
\als{
\prod_{j=1}^m \frac{1}{\sqrt{p_j}}
	&= \frac{(\log p_1)\cdots (\log p_d)}{(\log x_1)^dx_1^{m/2}} (1+O(\eta)) \\
	&= \frac{1}{\eta x \cdot x_1^{(d-1)/2}}
		\left(\prod_{j=1}^d \log p_j\right)
		\int_{x_1-\eta x}^{x_1}
			\frac{u^{(d-m-1)/2}}{(\log u)^d}  \dee u +O(\eta x^{-m/2}) .
}
Therefore, we conclude that
\[
\sum_{\substack{\bs p \in \CQ(\bs r)\\ \bs p\equiv a\mod{q} }}
	 \prod_{j=1}^m  \frac{1}{\sqrt{p_j}} \sqrt{1-\left(\frac{p_j+1-p_{j+1}}{2\sqrt{p_j}}\right)^2} \\
	 = \frac{I(\bs r)\cdot S(\bs r;q,\bs a)}{\eta^d x x_1^{(d-1)/2} }
	 	+ O\left(\frac{\eta^{d+1/2} x^{(d-m+1)/2}}{q^d} \right) ,
\]
where
\[
S(\bs r;q,\bs a) = \sum_{\substack{ \bs p\in \CQ(\bs r) \\ \bs p\equiv \bs a\mod q}} \prod_{j=1}^d \log p_j .
\]
So, taking $B\ge2A+2$, we see that \eqref{W e1} is reduced to showing that
\[
\sum_{q\le x^{1/6-\epsilon}}
	q^{d-1} \max_{\bs a\in ((\Z/q\Z)^*)^d}
		\left| \frac{S(\bs r;q,\bs a)}{\eta^d x x_1^{(d-1)/2}} - \frac{2^{d-1}}{\phi(q)^d} \right| \ll \frac{\eta^d x^{(d-m+1)/2}}{I(\bs r)\cdot (\log x)^A} .
\]
Since
\[
I(\bs r) \ll  \frac{\eta^d x^{(d-m+1)/2}}{(\log x)^d},
\]
it is enough to prove that
\eq{W e2}{
	\sum_{q\le x^{1/6-\epsilon}}
	q^{d-1} \max_{\bs a\in ((\Z/q\Z)^*)^d}
		\left| S(\bs r;q,\bs a)  - \frac{2^{d-1} \eta^d x x_1^{(d-1)/2}}{\phi(q)^d} \right| 	
			\ll \frac{ \eta^d x^{(d+1)/2}}{(\log x)^A} .
}

In order to prove \eqref{W e2}, we start by observing that
\[
S(\bs r;q,\bs a)	
	= \sum_{\substack{ x_1-\eta x<p_1\le x_1\\ p_1\equiv a_1\mod q}} (\log p_1)
		\sum_{\substack{ r_2- 1 < \frac{p_2 - p_1}{2\eta\sqrt{x_1}} \le r_2 \\ p_2\equiv a_2\mod q} } (\log p_2)
			\cdots
				\sum_{\substack{r_d- 1 < \frac{p_d - p_{d-1}}{2\eta\sqrt{x_1}} \le r_d \\ p_d\equiv a_d\mod q} } \log p_d
\]
We replace the last sum by $2\eta\sqrt{x_1}/\phi(q)$, which introduces a total error term of size
\als{
\sum_{\substack{ x_1-\eta x<p_1\le x_1 \\ p_1\equiv a_1\mod q}} (\log p_1)
		\sum_{\substack{ r_2-1 < \frac{p_2 - p_1}{2\eta\sqrt{x_1}} \le r_2 \\ p_2\equiv a_2\mod q} } (\log p_2)
			\cdots \sum_{\substack{ r_{d-1}-1 < \frac{p_{d-1} - p_{d-2}}{2\eta\sqrt{x_1}} \le r_{d-1} \\ p_{d-1}\equiv a_{d-1}\mod q} }
				 (\log p_{d-1}) \\
		\times E(p_{d-1}+2(r_d-1)\eta \sqrt{x_1},2\eta\sqrt{x_1};q ) ,
}
where $E(x,h;q)$ is defined by \eqref{def-E}. We estimate this error term by fixing $p_{d-1}$ and summing first over $p_1,\dots,p_{d-2}$. Note that $p_{d-1}$ lies in an interval of size $O(\eta x)$ around $x_1$ and, given $p_{d-1}$, the primes $p_1,\dots,p_{d-2}$ lie in intervals of length $O(\eta\sqrt{x})$ around $p_{d-1}$. Using the Brun-Titschmarsh inequality for the sums over $p_1, \dots, p_{d-2}$, we find that the error term is bounded by some absolute constant times
\[
\left(\frac{\eta\sqrt{x}}{\phi(q)}\right)^{d-2}
	\sum_{\substack{x_1-O(\eta x)<p_{d-1}\le x_1+O(\eta x) \\ p_{d-1}\equiv a_{d-1}\mod q}} (\log p_{d-1})
	E(p_{d-1}+2(r_d-1)\eta\sqrt{x_1},2\eta\sqrt{x_1} ; q ) .
\]
For each $y$ within $\eta^2\sqrt{x_1}$ of $p_{d-1}+2(r_d-1)\eta\sqrt{x_1}$, the Brun-Titschmarsh inequality implies that
\[
E(y,2\eta \sqrt{x_1};q) - E(p_{d-1}+2(r_d-\eta)\sqrt{x_1},2\eta\sqrt{x};q )
	\ll  \frac{\eta^2 \sqrt{x}}{\phi(q)} .
\]
Therefore
\als{
E(p_{d-1}+2(r_d-\eta)\sqrt{x_1},2\eta\sqrt{x_1};q )
	 = \frac{1}{\eta^2\sqrt{x_1}} \int_{p_{d-1}+2(r_d-1)\eta\sqrt{x_1}}^{p_{d-1}+ 2(r_d-1)\eta\sqrt{x_1}+\eta^2\sqrt{x_1}}
		E(y,2\eta \sqrt{x_1};q) \dee y  \\
	\qquad\qquad	+ O\left(\frac{\eta^2 \sqrt{x}}{\phi(q)}\right) .
}
for $y \in [p_{d-1}+2(r_d-1)\eta\sqrt{x_1}, p_{d-1}+ 2(r_d-1)\eta\sqrt{x_1}+\eta^2\sqrt{x_1}]$. Summing this over $p_{d-1} \in[x_1-O(\eta x),x_1+O(\eta x)]$ and reversing the sums, we find that the total error introduced by replacing the sum over $p_d$ by $2 \eta \sqrt{x_1}/\phi(q)$ in $S(\bs r;q,\bs a)$ is
\als{
&\ll \frac{1}{\eta^2\sqrt{x}}\cdot \left(\frac{\eta\sqrt{x}}{\phi(q)}\right)^{d-2} \int_{x/2}^{3x} E(y,2\eta\sqrt{x_1};q)
	\sum_{\substack{ -\eta^2\sqrt{x} \le p_{d-1} +2(r_d-1)\eta\sqrt{x_1} -y  \le 0 \\ p_{d-1}\equiv a_{d-1}\mod q}} \log p_{d-1}  \;\dee y \\
&\qquad + \frac{\eta^{d+1} x^{(d+1)/2} }{\phi(q)^d} \\
&\ll   \frac{(\eta \sqrt{x})^{d-2}}{\phi(q)^{d-1}}
		 \int_{x/2}^{3x} E(y,2\eta\sqrt{x_1};q)   \dee y  + \frac{\eta^{d+1} x^{(d+1)/2} }{\phi(q)^d} .
}
In conclusion, we have proven that
\als{
S(\bs r;q,\bs a)	
	= \frac{2\eta\sqrt{x_1}}{\phi(q)}
		\sum_{\substack{ x_1-\eta x<p_1\le x_1 \\ p_1\equiv a_1\mod q}} (\log p_1)
		\sum_{\substack{ r_2-1 < \frac{p_2 - p_1}{2\eta\sqrt{x_1}} \le r_2 \\ p_2\equiv a_2\mod q} } (\log p_2)
			\cdots
				\sum_{\substack{ r_{d-1}-1 < \frac{p_{d-1} - p_{d-2}}{2\eta\sqrt{x_1}} \le r_{d-1}
					\\ p_{d-1}\equiv a_{d-1}\mod q} } \log p_{d-1} \\
		+ O\left(  \frac{(\eta \sqrt{x})^{d-2}}{\phi(q)^{d-1}}
		 \int_{x/2}^{3x} E(y,2\eta\sqrt{x_1};q)   \dee y  +  \frac{\eta^{d+1} x^{(d+1)/2} }{\phi(q)^d}
		 	\right).
}
Next, we replace the sum over $p_{d-1}$ by $2\eta\sqrt{x_1}/\phi(q)$. The total error term produced can be shown to be
\[
\ll    \frac{(\eta \sqrt{x})^{d-2}}{\phi(q)^{d-1}}
		 \int_{x/2}^{3x} E(y,2\eta\sqrt{x_1};q)   \dee y  +   \frac{\eta^{d+1} x^{(d+1)/2} }{\phi(q)^d}
\]
by following the above argument. We continue this way to deduce that
\als{
S(\bs r;q,\bs a)	
	&= \left(\frac{2\eta\sqrt{x_1}}{\phi(q)} \right)^{d-1}
		\sum_{\substack{ x_1-\eta x<p_1\le x_1 \\ p_1\equiv a_1\mod q}} (\log p_1) \\
	&\quad	+ O\left(   \frac{(\eta \sqrt{x})^{d-2}}{\phi(q)^{d-1}}
		 \int_{x/2}^{3x} E(y,2\eta\sqrt{x_1};q)   \dee y  +   \frac{\eta^{d+1} x^{(d+1)/2} }{\phi(q)^d}  \right) \\
	&= \frac{2^{d-1}\eta^d x x_1^{(d-1)/2}}{\phi(q)^d}
		 + O\left(  \frac{(\eta \sqrt{x})^{d-1}}{\phi(q)^{d-1}} E(x_1-\eta x, \eta x;q) \right) \\
		&\qquad + O\left(\frac{(\eta \sqrt{x})^{d-2}}{\phi(q)^{d-1}} \int_x^{3x} E(y,2\eta\sqrt{x_1};q)   \dee y
		 	+  \frac{\eta^{d+1} x^{(d+1)/2} }{\phi(q)^d}  \right) .
}
In view of the above formula, relation \eqref{W e2} follows by the Bombieri-Vinogradov theorem and by Lemma \ref{bv-short} with $Q = x^{1/6 - \epsilon}$.
\end{proof}


\bibliographystyle{abbrv}

\def\cprime{$'$}

\end{document}